\documentclass[a4paper,11pt,twoside]{amsart}
\usepackage[english]{babel}
\usepackage[utf8]{inputenc}

\usepackage[a4paper,inner=3cm,outer=3cm,top=4cm,bottom=4cm,pdftex]{geometry}
\usepackage{fancyhdr}
\usepackage{titlesec}
\usepackage{enumerate}
\usepackage{color}
\usepackage{bold-extra}
\titleformat{\section}{\normalfont\scshape\centering}{\thesection}{1em}{}
\titleformat{\subsection}{\bfseries}{\thesubsection}{1em}{}

\usepackage{comment}
\usepackage{graphics}
\usepackage{aliascnt}
\usepackage[pdftex,citecolor=green,linkcolor=red]{hyperref}

\usepackage{amsmath}
\usepackage{amsfonts}
\usepackage{amssymb}
\usepackage{amsthm}
\usepackage{comment}
\usepackage{mathtools}

\newtheorem{theorem}{Theorem}[section]

\newtheorem{lemma}[theorem]{Lemma}
\newtheorem{claim}[theorem]{Claim}
\newtheorem{proposition}[theorem]{Proposition}
\newtheorem{information}[theorem]{Information}
\theoremstyle{definition}
\newtheorem{definition}[theorem]{Definition}
\newtheorem{remark}[theorem]{Remark}

\numberwithin{equation}{section}

\newcommand{\Meas}{\textup{Meas}}

\renewcommand\d{\textnormal{d}}
\usepackage[normalem]{ulem}

\author{Olli J\"arviniemi}
\newcommand{\ceil}[1]{\lceil #1 \rceil}
\newcommand{\floor}[1]{\lfloor #1 \rfloor}

\newcommand{\all}{\lll}
\title{On large differences between consecutive primes}

\date{}

\begin{document}
\begin{abstract}
We show that
$$\sum_{\substack{p_n \in [x, 2x] \\ p_{n+1} - p_n \ge x^{1/2}}} (p_{n+1} - p_n) \ll x^{0.57+\epsilon}$$
and
$$\sum_{\substack{p_n \in [x, 2x] \\ p_{n+1} - p_n \ge x^{0.45}}} (p_{n+1} - p_n) \ll x^{0.63+\epsilon},$$
where $p_n$ is the $n$th prime number. The proof combines Heath-Brown's recent work with Harman's sieve, improving and extending his results. We give applications of the results to prime-representing functions, binary digits of primes and approximation of reals by multiplicative functions.
\end{abstract}

\maketitle

\section{Introduction}
\label{sec:intro}

A central problem in number theory is understanding the distribution of prime numbers. Notable work in this area include Baker-Harman-Pintz's~\cite{BHP} result on intervals of length $x^{0.525}$ containing prime numbers and Jia's work~\cite{jia} showing that almost all intervals of length $x^{1/20}$ contain primes, both being preceded by numerous weaker results on the problems.

Both of these results may be viewed as instances of the problem of bounding the number of intervals of length $x^c$ without primes. The case $c = 1/2$ is of special interest, as even under the Riemann hypothesis it is not known that intervals of length $\sqrt{x}$ necessarily contain primes.

The current best result for $c = 1/2$ is given in the recent work of Heath-Brown~\cite{HB-V}, where he shows that there are at most $X^{3/5 + \epsilon}$ intervals $[x, x + x^{1/2}]$ with $x \in \mathbb{Z} \cap [X, 2X]$ that do not contain primes. Heath-Brown's result relies on his mean square estimate (see~\cite[Proposition 1]{HB-V} or Proposition \ref{prop:HB_MVT} below) for the product of two Dirichlet polynomials, one of which is sparse. Heath-Brown's work improves the previous result of Matomäki~\cite{matomaki}, who obtained an exceptional set of size $X^{2/3}$. See \cite{wolke2}, \cite{HB1978}, \cite{HB1979} and \cite{peck2} for earlier results. 

Heath-Brown's argument does not utilize Harman's sieve, in contrast to Matomäki's proof. One may strengthen the result by combining Heath-Brown's methods with Harman's sieve. We show the following.

\begin{theorem}
\label{thm:0.5}
Let $p_n$ denote the $n$th prime. We have
$$\sum_{\substack{p_n \in [x, 2x] \\ p_{n+1} - p_n \ge x^{1/2}}} (p_{n+1} - p_n) \ll x^{0.57+\epsilon}$$
for any fixed $\epsilon > 0$.
\end{theorem}

It is the best to view the bound as $x^{1/2 + 0.07 + \epsilon}$, so that the ``excess'' is $30\%$ smaller than in Heath-Brown's result $x^{1/2 + 0.1 + \epsilon}$.

We further demonstrate that Heath-Brown's method adapts to intervals shorter than $\sqrt{x}$.

\begin{theorem}
\label{thm:0.45}
Let $p_n$ denote the $n$th prime. We have
$$\sum_{\substack{p_n \in [x, 2x] \\ p_{n+1} - p_n \ge x^{0.45}}} (p_{n+1} - p_n) \ll x^{0.63+\epsilon}$$
for any fixed $\epsilon > 0$.
\end{theorem}

The Lindelöf hypothesis would imply a bound of $x^{1 - c + \epsilon}$ for intervals of length $x^c$~\cite{yu}. Hence the bound in Theorem~\ref{thm:0.45} should be viewed as $x^{(1 - 0.45) + 0.08 + \epsilon}$, the excess being of similar size as in Theorem~\ref{thm:0.5}.

For intervals shorter than $\sqrt{x}$, previously Peck~\cite{peck} has given a bound of $x^{1.25 - c + \epsilon}$ for intervals of length $x^c$ for any $1/4 < c \le 1/2$. Islam~\cite{islam} gives the bound $x^{2/3 + 5(1/2 - c)}$ for $c < 1/2$, improving on Peck's result for $c > 1/2 - 1/48$. Simultaneously to our work Stadlmann \cite{stadlmann} has given the bound $x^{1.23 - c + \epsilon}$ for $c > 0.23$ (also by using Heath-Brown's mean value theorem from \cite{HB-V}).

Theorem~\ref{thm:0.45} gives a substantial improvement on previous results, demonstrating the incredible strength of Heath-Brown's new methods: the excess in Theorem \ref{thm:0.45} is less than a third of the excess in Peck's result, and, with the exception of Heath-Brown's result~\cite{HB-V}, the bound in Theorem \ref{thm:0.45} for intervals of length $x^{0.45}$ is stronger than any previous bound for intervals of length $\sqrt{x}$.

Certainly, with more work one could obtain bounds for intervals of length $x^c$ with any $0.45 < c < 0.5$ (beating the bound $x^{0.63}$ one gets from Theorem \ref{thm:0.45}). One can also extend the results for even shorter intervals (see also \cite{stadlmann}). The $\epsilon$ term in the exponents could be dropped with a bit more work, and by more effort one could improve the exponents slightly.

The proofs of Theorems~\ref{thm:0.5} and~\ref{thm:0.45} actually show that there are few intervals of length $x^c$ which contain $o(x^c/\log x)$ primes. See Theorem~\ref{thm:many} for a precise formulation.

We present a couple of applications of the results (see Section \ref{sec:applications} for more detailed discussion). First, we show that there are prime-representing functions of the form $\lfloor A^{\alpha^n} \rfloor$ for any $\alpha \ge 20/11$.

\begin{theorem}
\label{thm:PRF}
Let $\alpha \ge 20/11 = 1.818\ldots$ be fixed. There exists $A > 1$ such that $\lfloor A^{\alpha^n} \rfloor$ is a prime for all $n \in \mathbb{Z}_+$. 
\end{theorem}

Previous results on the problem include those of Mills~\cite{mills} (the first such result, giving $\alpha = 3$), Matomäki~\cite{matomaki-PRF} (allowing $\alpha \ge 2$) and Islam~\cite{islam} (with $\alpha \ge 1.946\ldots$).

Second, we show that there are infinitely many primes with very many (or very few) ones in their binary representation.

\begin{theorem}
\label{thm:bin}
Let $d \in \{0, 1\}$. There are infinitely many primes $p$ such that at least $74.2\%$ of the digits of the binary representation of $p$ are equal to $d$.
\end{theorem}

In \cite{naslund} it is noted that the bound $3/4 - 1/80 - \epsilon = 73.75\% - \epsilon$ for Theorem \ref{thm:bin} follows from the result of Baker, Harman and Pintz \cite{BHP} on primes in intervals of length $x^{1/2 + 1/40}$. To our knowledge this was the best previous bound on the problem. If there were primes in intervals of length $\sqrt{x}$, the same method would give the bound $75\% - \epsilon$.

Finally, we note an improvement on approximation of real numbers by multiplicative functions.

\begin{theorem}
\label{thm:approx}
Let $\epsilon > 0$ and $\alpha > 1$ be given. There are infinitely many integers $n$ such that
$$\left|\frac{\sigma(n)}{n} - \alpha\right| < n^{-0.55 + \epsilon},$$
where $\sigma(n)$ is the sum of divisors of $n$.
\end{theorem}

The previous best result is due to Harman \cite{harman-approx} with the bound $n^{-0.52}$.

\subsection{Overview of the method}
\label{sec:overview}

Our proof largely follows the one given by Heath-Brown in~\cite{HB-V}, with the modification that we in addition utilize Harman's sieve. We give an overview of the proof below. For convenience we mostly consider the case of intervals of length $x^{0.5}$.

First, we perform elementary manipulations, reducing to showing that for all but roughly $O(x^{0.07})$ integers $m \approx \sqrt{x}$ we have
\begin{align}
\label{eq:overview_1}
S(\mathcal{A}(m), 2\sqrt{x}) \ge \epsilon \frac{|\mathcal{A}(m)|}{|\mathcal{B}(m)|} S(\mathcal{B}(m), 2\sqrt{x}),
\end{align}
where $\mathcal{A}(m)$ is a certain interval roughly of length $\sqrt{X}$  associated to $m \in \mathbb{Z}$, $\mathcal{B}(m)$ is an interval of length $x^{1 - o(1)}$, and $S(\mathcal{C}, z)$ counts the number integers in $\mathcal{C}$ which have no prime factors smaller than $z$. This follows the usual approach to applying Harman's sieve, where the set of interest is compared to a larger set.

We then decompose the terms $S(\mathcal{A}(m), 2\sqrt{x})$ and $S(\mathcal{B}(m), 2\sqrt{x})$ in~\eqref{eq:overview_1} by the Buchstab identity, which states that
\begin{align*}
S(\mathcal{C}, z) = S(\mathcal{C}, z') - \sum_{z' \le p < z} S(\mathcal{C}_p, p),
\end{align*}
where $\mathcal{C}_n = \{k \in \mathbb{N} : kn \in \mathcal{C}\}$. This reduces the problem to obtaining asymptotics of the form
\begin{align}
\label{eq:overview_2}
\sum_{\substack{p_1, \ldots , p_n \\ p_i \in [x^{\alpha_i}, x^{\beta_i}] \\ p_n < \ldots < p_1}} S(\mathcal{A}_{p_1 \cdots p_n}(m), z) \approx \frac{|\mathcal{A}|}{|\mathcal{B}|}\sum_{\substack{p_1, \ldots , p_n \\ p_i \in [x^{\alpha_i}, x^{\beta_i}] \\ p_n < \ldots < p_1}} S(\mathcal{B}_{p_1 \cdots p_n}(m), z)
\end{align}
for all but $O(x^{0.07})$ exceptional values of $m$, where $z$ is either a function of $x$ or $z = p_n$.

We use a method of Heath-Brown~\cite[Proposition 2]{HB-V} to link asymptotics of the form~\eqref{eq:overview_2} with the problem of bounding the mean value of certain type of Dirichlet polynomials. More specifically, in order to show that~\eqref{eq:overview_2} holds for all but $O(x^{0.07})$ exceptional $m$, it suffices to show (roughly)
\begin{align}
\label{eq:overview_3}
\int_{|t| \in [T_0, T]} |F(it)M(it)| \d t = o\left(\frac{Rx}{\log x}\right),
\end{align}
where $R = x^{0.07}$, $T = \sqrt{x}$, $T_0 = (\log x)^A$, $F(s)$ is the Dirichlet polynomial of length $x$ whose coefficients correspond to the summands in~\eqref{eq:overview_2} and $M(s)$ is an arbitrary polynomial of the form
$$M(s) = \sum_{i = 1}^R \zeta_im_i^{-s}, \qquad |\zeta_i| = 1, m_i \approx \sqrt{x}.$$
Note that $F(s)$ factorizes as a product of at least $n$ polynomials corresponding to the sums over $p_i$ in~\eqref{eq:overview_2}.

To prove~\eqref{eq:overview_3}, we use a mean value theorem due to Heath-Brown~\cite[Proposition 1]{HB-V}. The mean value theorem applies to mean squares of the form
\begin{align}
\label{eq:overview_4}
\int_{-T}^{T} |Q(it)M(it)|^2 \d t.
\end{align}
Our strategy for showing~\eqref{eq:overview_3} is thus factorizing $F(s)$ as $F(s) = P(s)Q(s)$, applying the Cauchy-Schwarz inequality to obtain
\begin{align}
\label{eq:overview_5}
\int_{-T}^{T} |F(it)M(it)| \d t \le \sqrt{\int_{-T}^{T} |P(it)|^2 \d t}\sqrt{\int_{-T}^{T} |Q(it)M(it)|^2 \d t}
\end{align}
and bounding the latter integral by Heath-Brown's mean value theorem. The former mean square in~\eqref{eq:overview_5} is bounded by further factorizing $P(s)$ and using various pointwise bounds and large value theorems to the factors.

We note that this is a simplification of the actual proof. In practice we start by assuming $F$ factorizes as $F(s) = A(s)B(s)C(s)$ and decompose the integral over $t$ according to the sizes of $|A(it)|, |B(it)|, |C(it)|$. In each of the resulting cases $t \in \mathcal{T}$ we may, in addition to applying the Cauchy-Schwarz argument as above, simply bound 
\begin{align}
\label{eq:overview_6}
\int_{\mathcal{T}} |F(it)M(it)| \d t \le R|\mathcal{T}| \max_{t \in \mathcal{T}} |A(it)B(it)C(it)|.
\end{align}
It suffices that at least one of these strategies yields a bound small enough to imply~\eqref{eq:overview_3}.

With these strategies, we are able to find a set such that if $F(s)$ factorizes as $F(s) = A(s)B(s)C(s)$ with the triplet $(A, B, C)$ of lengths lying in this set, then~\eqref{eq:overview_3} holds. In the case of intervals of length $\sqrt{x}$ the set of admissible $(A, B, C)$ is relatively simple, corresponding to a hexagon in a $(\log_x(A), \log_x(B))$ coordinate system. (Note that the length of $C(s)$ is essentially determined by the lengths of $A(s)$ and $B(s)$, as $F \approx x$.) For intervals of length $x^{0.45}$ the set is much more complicated and best described as a union of intersections of half-planes in the above coordinate system.

Recalling that $F(s)$ has at least $n$ factors corresponding to each of $p_i$ in~\eqref{eq:overview_2} and that there are possibly many ways of grouping the factors of $F(s)$ into three polynomials $A(s), B(s), C(s)$, we obtain numerous ranges of $n, \alpha_i, \beta_i$ for which the asymptotic~\eqref{eq:overview_2} holds.

However, we cannot establish an asymptotic of type~\eqref{eq:overview_2} for all $n, \alpha_i$ and $\beta_i$. Hence we apply Harman's sieve, discarding certain sums arising from the applications of the Buchstab identity, making sure that the resulting ``loss'' is less than $1 - \epsilon$. It would not be feasible to do this by hand and hence we perform this step with a computer calculation.

Of course, we cannot perform a check over all possible lengths of the factors of $F(s)$ (of which there are unboundedly many), and hence we have to manage with merely upper and lower bounds of the form $x^{\alpha} \le P \le x^{\beta}$ for the relevant polynomials $P(s)$. To overcome this issue, we perform an extensive casework, allowing us to reduce to cases where the differences $\beta - \alpha$ are small. In each case, we consider different ways of combining the factors of $F(s)$ to write $F(s) = A(s)B(s)C(s)$ and check whether our bounds on the lengths of polynomials are strong enough to imply that the resulting triplet $(A, B, C)$ necessarily lies in the set obtained before. We then sum the loss over those cases where we cannot find such a factorization $F(s) = A(s)B(s)C(s)$ and check that it indeed is less than one.

In the case of Theorem~\ref{thm:0.5} we obtain the result without using Heath-Brown's identity (except when establishing certain theoretical results). In contrast, for intervals of length $x^{0.45}$ we incorporate the Heath-Brown decomposition into our computer calculation. This is done simply by performing a casework on the lengths of the resulting polynomials.

The running times of the computations being roughly 15 minutes and 30 hours (for Theorems~\ref{thm:0.5} and~\ref{thm:0.45}, respectively) on a usual consumer laptop. Implementations in C++ are available with the arXiv version of the paper.

The organization of the paper is as follows. 

We present notation and our choice of parameters in Section~\ref{sec:parameters}. An exposition of the key tools is given in Section~\ref{sec:heath-brown}. 

We perform a reduction to~\eqref{eq:overview_1} in Section~\ref{sec:to_buchstab}. 

We then link asymptotic formulas of the form~\eqref{eq:overview_2} to mean value bounds as in~\eqref{eq:overview_3} in Section~\ref{sec:dir}. This is a somewhat standard procedure based on tools such as dyadic decomposition, Perron's formula, the Heath-Brown decomposition and so on, though the implementation is technical. We use, in particular, Shiu's bound on the moments of the divisor function in short intervals to bound various error terms.

Having reduced the problem to Dirichlet polynomials, we lay out various tools (such as coefficient bounds, pointwise bounds and bounds for moments of zeta sums) in Section~\ref{sec:tools}. Using these and Huxley's large value theorem, admissible ranges of $(A, B, C)$ are obtained in Section~\ref{sec:ranges}.

We then discuss the application of Harman's sieve, starting with the case $c = 0.5$ in Section~\ref{sec:0.5}. We start with theoretical results and then present the computational procedure and its results. The procedure is adapted to the case $c = 0.45$ in Section~\ref{sec:0.45}.

Discussion and proofs of the applications (Theorems~\ref{thm:PRF},~\ref{thm:bin} and~\ref{thm:approx}) are given in Section~\ref{sec:applications}.

We remark that our research procedure relied heavily on numerical computations. There is, a priori, numerous ways one may bound the integral $\int |F(it)M(it)| \d t$ (such as via Cauchy-Schwarz's inequality as in~\eqref{eq:overview_5}, the $L^1$-type estimate~\eqref{eq:overview_6}, Hölder's inequality as in Lemma~\ref{lem:two_zetas}), numerous ways to bound $\mathcal{T}$ in~\eqref{eq:overview_6} via large value theorems (there are multiple polynomials one may apply the bound to, one may raise those polynomials to some power, one may apply large value theorems of Hal\'asz-Montgomery, Huxley or Jutila), numerous ways one may combine the factors of $F(s)$ to obtain a product $A(s)B(s)C(s)$ and so on (not to mention that initially we, of course, did not know how strong of a result one can prove in Theorems~\ref{thm:0.5} and~\ref{thm:0.45}). We wrote several programs to guide our intuitions and search through the vast search space, and many of the key results and their proofs (such as Proposition~\ref{prop:range_0.5} and~\ref{prop:range_0.45}) were found with the help of such computations. And while in many cases the final argument is, once identified, relatively simple, due to its sheer complexity the application of Harman's sieve relies on a computer calculation.

\subsection{Choice of parameters and notation}
\label{sec:parameters}

Throughout the paper the length of a Dirichlet polynomial $P(s)$ is denoted by the same letter $P$, and $P$ may also refer to the Dirichlet polynomial itself. We often denote the support of the coefficients of $P(s)$ by $[P, P']$ (or $(P, P']$ etc.).

The letters $p$ and $q$ denote prime numbers, $\epsilon$ denotes a small positive constant, not necessarily the same at each occurrence, and $x$ is a large parameter.

Let
$$c \in \{0.45, 0.5\},$$
with $c = 0.5$ in the case of Theorem~\ref{thm:0.5} and $c = 0.45$ in Theorem~\ref{thm:0.45}. We define the following parameters:
$$\delta_0 = \frac{x^{c-1}}{(\log \log x)^2},$$
$$\delta_1 = \exp(-\sqrt{\log x}),$$
$$z_1 = \exp(\log x / (\log \log x)^5),$$
$$z_2 = \exp(\log x / (\log \log x)^3),$$
$$\eta = \exp(-(\log \log \log x)^2),$$
$$S = \exp((\log \log x)^{17}),$$
$$T = x^{1 - c}S^2,$$
$$T_0 = \exp(\sqrt{\log x}/3),$$
$$L_{\zeta} = \begin{cases}x^{(1-c)/2} = x^{1/4} &\text{ if } c = 0.5, \\ x^{(1 - c)/2}\exp(\log x / \sqrt{\log \log x}) = x^{0.275}\exp(\log x / \sqrt{\log \log x}) &\text{ if } c = 0.45\end{cases}$$
and 
\begin{align}
\label{eq:def-R}
R = \begin{cases} x^{0.07 + \nu} &\text{ if } c = 0.5, \\ x^{0.18 + \nu} &\text{ if } c = 0.45,\end{cases}
\end{align}
where $\nu > 0$ is an arbitrarily small but fixed constant.

Furthermore, we define
$$H' = x^c(\log \log x)^{-4},$$
and for an integer $m > 0$ we let
$$\mathcal{A}(m) = \{n \in \mathbb{Z}_+ : mH' < n \le mH'(1 + \delta_0)\}$$
and
$$\mathcal{B}(m) = \{n \in \mathbb{Z}_+ : mH' < n \le mH'(1 + \delta_1)\}.$$
We will always consider only those $m$ with $m \in [x/H', 3x/H']$. The variable $H'$ should not be confused with the letter $H$ used to refer to the length of a Dirichlet polynomial $H(s)$ introduced later in the proof. 

For a set $\mathcal{C}$ of integers, we let $\mathcal{C}_d$ denote $\{n \in \mathbb{Z}_+ : dn \in \mathcal{C}\}$ and $S(\mathcal{C}, z)$ denote the number of integers in $\mathcal{C}$ which have no prime factors smaller than $z$. We will write $\mathcal{A}_d(m)$ and $\mathcal{B}_d(m)$ instead of (the technically correct) $\mathcal{A}(m)_d$ and $\mathcal{B}(m)_d$.

For the convenience of the reader, here is a brief account on the reasons and constraints behind the choices of parameters above.

We reduce the problem to considering primes in the intervals $\mathcal{A}(m)$ and comparing these intervals to the longer intervals $\mathcal{B}(m)$. The parameters $H'$, $\delta_0$ and $\delta_1$ are relevant for this step, being chosen so that $\mathcal{A}(m)$ is slightly shorter than $x^c$ and that $\mathcal{B}(m)$ is long enough that we have asymptotic formulas for the number of primes in $\mathcal{B}(m)$. 

At the beginning of the proof we sieve out prime factors less than $z_1$. This is important for keeping the sizes of the coefficients of Dirichlet polynomials small, and is achieved if $z_1 = x^{f(x)}$ for $f(x)$ tending to zero fast enough. At certain places we use a simple sieve to replace $1_{p \mid n \implies p \ge z_1}$ with
$$\sum_{\substack{d \mid n \\ d < z_2 \\ p \mid d \implies p < z_1}} \mu(d),$$ the latter being occasionally more convenient to work with. This procedure requires $z_2$ to be somewhat larger than $z_1$ (namely $\log z_2 \ge (\log \log x)^{1+\epsilon} \log z_1)$. The pair $(z_1, z_2)$ chosen above satisfies these constraints.

It is convenient to discard polynomials whose length is too close to certain reals $s = s(x)$. We are able to discard lengths lying in $[sx^{-\eta}, sx^{\eta}]$ as long as $\eta^{-1}$ is larger than $(\log \log x)^A$ for some fixed (but large) $A$. Hence the choice of $\eta$ above.

The parameter $S$ encompasses many small losses and additional factors arising in the course of the proof, for example $\log$-powers arising from dyadic decompositions or upper bounds on $\tau(k)$ for integers $k$ which are $z_1$-rough. This imposes lower bounds on $S$ of the form $(\log x)^{O(1)}$ or $2^{\log x / \log z_1}$. A choice of the form $S = \exp((\log \log x)^C)$ for a large enough constant $C$ works. Any losses of powers of $S$ are insignificant, as the Vinogradov pointwise bound for Dirichlet polynomials wins $\exp((\log x)^{\alpha})$ with $\alpha > 0$.

The parameter $T$ corresponds to the length of the range of integration, chosen to be essentially $x^{1 - c}$. As noted above, powers of $S$ are insignificant and not worth too much attention.

As is common for Dirichlet polynomial methods, we handle the case $|t| \le T_0$ separately, as one obtains cancellations in sums such as $\sum p^{-it}$ only for large enough $|t|$. The specific value of $T_0$ is not too important.

The Heath-Brown decomposition essentially allows one to assume that any polynomials longer than a certain power of $x$ are ``zeta sums'', the benefit being that the fourth moment of the zeta function is known. This is useful when applied to zeta sums longer than $\sqrt{T} \approx x^{(1-c)/2}$, the savings being the larger the longer the zeta sums. The parameter $L_{\zeta}$ denotes the threshold starting from which we are interested in zeta sums. What we call zeta sums are not quite sums of the form $\sum 1/n^s$ (see Definition \ref{def:zeta_sum}), and for $c = 0.45$ our fourth moment estimate for the zeta sums is slightly lossy. Hence we leave a margin of $\exp(\log x / \sqrt{\log \log x})$, the constraints behind this term being that it is larger than $\max_{k \le x} \tau(k) = \exp(O(\log x / \log \log x))$ while being less than $x^{\epsilon}$.

The value of $R$ is such that $Rx^c$ corresponds to the bounds in Theorems~\ref{thm:0.5} and~\ref{thm:0.45}.

In the proof we will encounter many situations where a quantity $X$ is bounded by $Y$ up to losses of $S^{o(1)}$. We hence introduce the following notation:
\begin{align}
\label{eq:all}
X \all Y \Leftrightarrow \text{ there exists } \epsilon > 0 \text{ such that } S^{\epsilon}X \ll Y.
\end{align}

\section{Key tools}
\label{sec:heath-brown}

The following two propositions form the core of the method employed in this work. The first one is~\cite[Proposition 2]{HB-V} formulated slightly more generally. We give a proof below.

\begin{proposition}
\label{prop:dir_to_arit}
Let $0 < c < 1$ be a fixed constant, let $x$ be large and define $H', S, T$ and $T_0$ as in Section~\ref{sec:parameters}. Let $0 < \epsilon \le 1$ be fixed.

Let $F(s) = \sum_{k} c_kk^{-s}$ for some $c_k \in \mathbb{C}$ supported on $k \in [x, 2x]$. Assume that there exists a constant $C \in \mathbb{Z}_+$ such that
$$|c_k| \ll (\log x)^C\tau(k)^C.$$

Assume there exists $R$ such that for any distinct integers $m_1, \ldots, m_R \in [x/H', 3x/H']$ and any complex numbers $\zeta_1, \ldots , \zeta_R$ of magnitude $1$ we have
\begin{align}
\label{eq:dir_to_arit}
\int_{T_0 \le |t| \le T} |F(it)M(it)| \d t \le \frac{Rx}{S^{\epsilon}},
\end{align} 
where
$$M(s) = \sum_{i = 1}^R \zeta_i m_i^{-s}.$$

Then, for all but $O(R)$ integers $m \in [x/H', 3x/H']$ one has
\begin{align}
\label{eq:dir_to_arit_c}
\sum_{k \in \mathcal{A}(m)} c_k = \frac{\delta_0}{\delta_1}\sum_{k \in \mathcal{B}(m)} c_k + O\left(\frac{\delta_0 x}{S^{\epsilon}}\right).
\end{align}
\end{proposition}

Note that the left hand side of \eqref{eq:dir_to_arit_c} is bounded from above by $\delta_0 x(\log x)^{O(1)}$ (and is heuristically of this magnitude for many choices of $c_k$), so \eqref{eq:dir_to_arit_c} corresponds to an asymptotic formula for the average of $c_k$ in a short interval with savings of $S^{\epsilon}$ in the error term. Even though we have fixed the choices of $H', S, T$ and $T_0$ here, the result applies for a wider range of parameters. In this work we will be applying the result with $c \in \{0.45, 0.5\}$.

The second vital tool is Heath-Brown's mean value theorem~\cite[Theorem 4(iii)]{HB-MVT}.

\begin{proposition}
\label{prop:HB_MVT}
Let $T \ge 1$ and let $m_1, \ldots , m_R \in (0, T]$ be distinct integers. Let $\zeta_1, \ldots , \zeta_R$ be complex numbers of modulus $1$. Then, for any $N \in \mathbb{Z}_+$ and $q_1, \ldots, q_N \in \mathbb{C}$ we have
\begin{align*}
\int_0^T \left|\sum_{k = 1}^R \zeta_k m_k^{-it}\right|^2\left|\sum_{n \le N} q_nn^{-it}\right|^2 \d t \ll_{\epsilon} \left(N^2R^2 + (NT)^{\epsilon}(NRT + NR^{7/4}T^{3/4})\right)\max_n |q_n|^2
\end{align*}
for any $\epsilon > 0$.
\end{proposition}
For $R \le T^{1/3}$ the term $NR^{7/4}T^{3/4}$ is smaller than $NRT$ and may thus be dropped. This is the case for our choice of parameters.

Furthermore, at a couple of occasions we apply a bound on the moments of the divisor function on short intervals. This lemma follows from the more general result of Shiu~\cite{shiu}.

\begin{lemma}
\label{lem:shiu}
Let $\delta > 0$ and $N \in \mathbb{Z}_+$ be fixed. For any $X, Y, z \ge 2$ with $X^{\delta} \le Y \le X$ we have
\begin{align*}
\sum_{\substack{X < n \le X + Y \\ p \mid n \implies p \ge z}} \tau(n)^N \ll \frac{Y}{\log X} \left(\frac{\log X}{\log z}\right)^{2^N}.
\end{align*}
\end{lemma}

As our formulation of Proposition~\ref{prop:dir_to_arit} is more general than~\cite[Proposition 2]{HB-V}, we give a proof below (even though the proof is essentially the same as in~\cite{HB-V}).

\begin{proof}[Proof of Proposition~\ref{prop:dir_to_arit}]
Let $m \in [x/H', 3x/H']$. We start with an application of Perron's formula (see e.g. \cite[Lemma 1.1]{harman}), obtaining
\begin{align*}
\sum_{k \in \mathcal{A}(m)} c_k = \frac{1}{2\pi i}\int_{-iT}^{iT} F(s)\frac{(1 + \delta_0)^s - 1}{s}(H'm)^s ds + O(E), 
\end{align*}
where the error $E$ is bounded by
\begin{align*}
E \ll \sum_{x \le k \le 2x} |c_k|\left(\frac{1}{\max(1, T|\log(mH'(1 + \delta_0)/k)|)} + \frac{1}{\max(1, T|\log(mH'/k)|)}\right).
\end{align*}

We first bound this error. For a parameter $J \ge x^{\epsilon}$, $k \in [x, 2x]$ and $m \in [x/H', 3x/H']$, the condition $J < |mH'(1 + \delta_0) - k| \le 2J$ implies
\begin{align*}
\frac{1}{T|\log (mH'(1 + \delta_0)/k)|} \ll \frac{x}{JT},
\end{align*}
and hence the corresponding terms contribute $ \ll xT^{-1}(\log x)^{O_C(1)}$ by Lemma~\ref{lem:shiu}. Summing over dyadic ranges of $J$ gives a contribution of $\ll xT^{-1}(\log x)^{O(1)}$. The case $J < |mH' - k| \le 2J$ is similar. Finally, for the case where $|mH'(1 + \delta_0) - k| < x^{\epsilon}$ or $|mH' - k| < x^{\epsilon}$, we bound the contribution by $2|c_k|$, obtaining an error of $\ll x^{\epsilon}(\log x)^{O(1)}$, again by Shiu's bound (Lemma \ref{lem:shiu}). Hence, the error is
$$E \ll \frac{x}{T}(\log x)^{O(1)} + x^{\epsilon}(\log x)^{O(1)} = O(\delta_0 x/S^{\epsilon}).$$

A similar analysis applies to $\mathcal{B}(m)$, leading to
\begin{align*}
\sum_{k \in \mathcal{B}(m)} c_k = \frac{1}{2\pi i}\int_{-iT}^{iT} F(s)\frac{(1 + \delta_1)^s - 1}{s}(H'm)^s ds + O\left(\frac{\delta_0 x}{S^{\epsilon}}\right). 
\end{align*}

Writing $\Delta(n) = \Delta(n, m) = 1_{n \in \mathcal{A}(m)} - \frac{\delta_0}{\delta_1}1_{n \in \mathcal{B}(m)}$, it then follows that
\begin{align*}
\sum_{k \in \mathbb{Z}_+} c_k\Delta(k) = \frac{1}{2\pi i} \int_{-iT}^{iT} F(s)G(s)m^{s} ds + O\left(\frac{\delta_0 x}{S^{\epsilon}}\right),
\end{align*}
where
\begin{align*}
G(s) = \left(\frac{(1 + \delta_0)^s - 1}{s} - \frac{\delta_0}{\delta_1}\frac{(1 + \delta_1)^s - 1}{s}\right)(H')^s.
\end{align*}
For bounding the contribution of small values of $|s|$, we note that if $0 \le \mu \le 1$ and $t$ is real, we have
\begin{align*}
\frac{(1 + \mu)^{it} - 1}{it} = \mu + O(\mu^2(1 + |t|)),
\end{align*}
and hence $|G(it)| \ll \delta_0 \delta_1 (1 + |t|)$. Moreover, we have $|F(it)| \le \sum_{k} |c_k| \ll x(\log x)^{O(1)}$, and hence
\begin{align*}
\int_{-T_0}^{T_0} |F(it)G(it)| \d t \ll \delta_0\delta_1x T_0^2 (\log x)^{O(1)},
\end{align*}
which is $O(\delta_0 x/S^{\epsilon})$.

Hence, we are left with showing that
\begin{align}
\label{eq:prop2_int}
\left|\int_{\substack{T_0 \le |t| \le T}} F(it)G(it)m^{it} \d t\right| = O\left(\frac{\delta_0 x}{S^{\epsilon}}\right)
\end{align}
for all but $O(R)$ integers $m \in [x/H', 3x/H']$. Assume not, and let $m_1, \ldots , m_R$ be such that the integral in~\eqref{eq:prop2_int} is greater than $3\delta_0 x/S^{\epsilon}$ in absolute value. Choose complex coefficients $\zeta_j$ of absolute value $1$ such that
\begin{align*}
\overline{\zeta_j} \int_{T_0 \le |t| \le T} F(it)G(it)m_j^{it} \d t = \left|\int_{T_0 \le |t| \le T} F(it)G(it)m_j^{it} \d t \right|
\end{align*}
for $1 \le j \le R$. Let $M(s) = \sum_{i = 1}^R \zeta_i m_i^{-s}$. Now
\begin{align*}
\int_{T_0 \le |t| \le T} F(it)G(it)\overline{M(it)} \d t \ge 3R\frac{\delta_0 x}{S^{\epsilon}}.
\end{align*}
By 
$$|G(it)| = \left|\int_{1}^{1 + \delta_0} v^{-it - 1} \d v - \delta_0 \delta_1^{-1} \int_{1}^{1 + \delta_1} v^{-it - 1} \d v\right| \le 2\delta_0,$$
we now have
\begin{align*}
\int_{T_0 \le |t| \le T} |F(it)M(it)| \d t \ge \frac{3}{2}\frac{Rx}{S^{\epsilon}},
\end{align*}
contrary to the assumption \eqref{eq:dir_to_arit}.
\end{proof}

\section{Reduction to Buchstab sums}
\label{sec:to_buchstab}

The purpose of this section is to reformulate Theorems~\ref{thm:0.5} and~\ref{thm:0.45} in terms of Buchstab sums. We start with the following lemma. Recall the notations $\mathcal{A}(m), \mathcal{B}(m)$ and $S(\mathcal{C}, z)$ from Section \ref{sec:parameters}.

\begin{lemma}
\label{lem:to_buch}
Let $0 < d < 1$ be a constant and let $c \in \{0.45, 0.5\}$ be given. Assume that the number of integers $m \in [x/H', 3x/H']$ with
\begin{align}
\label{eq:S-diff}
S(\mathcal{A}(m), 2\sqrt{x}) < d \frac{\delta_0}{\delta_1} S(\mathcal{B}(m), 2\sqrt{x})
\end{align}
is less than $kR$ for some constant $k > 0$. Then the measure of $y \in [x, 2x]$ such that 
\begin{align}
\label{eq:pi-diff}
\pi(y(1 + \delta_0)) - \pi(y) \le (d - \epsilon)\frac{y\delta_0}{\log x}
\end{align}
is at most $kRx^c$. In particular, we then have
$$\sum_{\substack{p_n \in [x, 2x] \\ p_{n+1} - p_n \ge x^{c}}} (p_{n+1} - p_n) \ll Rx^c.$$
\end{lemma}

\begin{proof}
For the last claim, note that any prime gap $[p_n, p_{n+1}]$ with $p_{n+1} - p_n \ge x^c$ gives an interval $[p_n, (p_{n+1} + p_n)/2]$ of length $(p_{n+1} - p_n)/2$ of values of $y$ satisfying~\eqref{eq:pi-diff}. Hence the sum of lengths of such long prime gaps can be at most $2kRx^c = O(Rx^c)$.

Denote the set of $y$ satisfying \eqref{eq:pi-diff} by $\mathcal{I}(x)$. Assume that $\Meas(\mathcal{I}(x)) > kRx^c$. Denoting $R' = \ceil{kR}$, it follows that one may choose points $y_1, \ldots , y_{R'} \in \mathcal{I}(x)$ so that $|y_i - y_j| > H'$ for any $i \neq j$. Let $m_i = 1 + \floor{y_i/H'}$ for every $i = 1, \ldots, R$, so $m_i \in [x/H', 3x/H']$ and $m_i$ are pairwise distinct. We show that $m_i$ satisfy~\eqref{eq:S-diff}, resulting in a contradiction.

Note that, by the Brun-Titchmarsh theorem, we have, for $i = 1, \ldots , R$,
$$|\pi(m_iH') - \pi(y_i)| \ll \frac{H'}{\log x} = o\left(\frac{\delta_0 x}{\log x}\right).$$
Similarly $|\pi(m_iH'(1 + \delta_0)) - \pi(y_i(1 + \delta_0))| = o(\delta_0 x/\log x)$. It follows that
\begin{align}
\label{eq:small_A}
S(\mathcal{A}(m_i), 2\sqrt{x}) = \pi(m_iH'(1 + \delta_0)) - \pi(m_iH') \le \left(d - \frac{\epsilon}{2}\right) \frac{m_iH'\delta_0}{\log x}.
\end{align}

On the other hand, by the prime number theorem with Vinogradov's error term (see e.g. \cite[Corollary 8.30]{IK}), we have, for $i = 1, \ldots , R$,
\begin{align}
\label{eq:large_B}
S(\mathcal{B}(m_i), 2\sqrt{x}) = \int_{m_iH'}^{m_iH'(1 + \delta_1)} \frac{\d t}{\log t} + o(\delta_1x/\log x) \ge \left(1 - \frac{\epsilon}{2}\right)\frac{m_iH'\delta_1}{\log x}.
\end{align}
The equations \eqref{eq:small_A} and \eqref{eq:large_B} contradict the assumption \eqref{eq:S-diff}, from which the result follows.
\end{proof}

Hence, our task is to show that for all but $O(R)$ integers $m \in [x/H', 3x/H']$ we have
$$S(\mathcal{A}(m), 2\sqrt{x}) > d \frac{\delta_0}{\delta_1} S(\mathcal{B}(m), 2\sqrt{x})$$
for some (small) $d > 0$. As explained in Section~\ref{sec:overview}, we accomplish this by utilizing Buchstab's identity and Harman's sieve. Hence, in the next sections we present a method for obtaining asymptotic formulas of the form
\begin{align}
\label{eq:asy_form}
\sum_{\substack{p_1, \ldots, p_n \\ p_i \in [x^{\alpha_i}, x^{\beta_i}] \\ p_n < \ldots < p_1}} S(\mathcal{A}_{p_1 \cdots p_n}(m), z) = \frac{\delta_0}{\delta_1}\sum_{\substack{p_1, \ldots, p_n \\ p_i \in [x^{\alpha_i}, x^{\beta_i}] \\ p_n < \ldots < p_1}} S(\mathcal{B}_{p_1 \cdots p_n}(m), z) + o\left(\frac{\delta_0 x}{\log x}\right).
\end{align}

We note that in the course of establishing the assumption of Lemma~\ref{lem:to_buch} we do not only obtain Theorems~\ref{thm:0.5} and~\ref{thm:0.45}, but we in fact also get the following stronger result.

\begin{theorem}
\label{thm:many}
Fix $c \in \{0.45, 0.5\}$ and $\nu > 0$. There exist constants $C, d' > 0$ such that number of disjoint intervals $[n, n + n^c], n \in \mathbb{Z} \cap [x, 2x]$ with
\begin{align}
\label{eq:many}
\pi(n + n^c) - \pi(n) \le d' \frac{n^c}{\log x}
\end{align}
is less than $CR$, where $R = x^{0.07 + \nu}$ if $c = 0.5$ and $R = x^{0.18 + \nu}$ if $c = 0.45$.
\end{theorem}

\begin{proof}[Proof of Theorem~\ref{thm:many} assuming premise of Lemma~\ref{lem:to_buch}]

Consider an integer $n \in [x, 2x]$ such that $\pi(n + n^c) - \pi(n) \le d' \frac{n^c}{\log x}$. Write the interval $[n, n+n^c]$ as a disjoint union of $\ell = O((\log \log x)^2)$ half-open intervals $I$ of length $|I| \in [3x^c/(\log \log x)^2, 4x^c/(\log \log x)^2]$. At least $\ell/2$ of such $I$ must contain less than 
$$10d'\frac{x^c}{(\log x)(\log \log x)^2}$$
primes. Each of these $\ell/2$ intervals $I$ results in a set $\mathcal{Y}$ of $y \in I \subset [x, 3x]$ satisfying
$$\pi(y(1 + \delta_0)) - \pi(y) < 10d' \frac{y\delta_0}{\log x}$$
with measure $\Meas(\mathcal{Y}) \gg x^c/(\log \log x)^2$. 

Note that the intervals $I$ and thus the resulting sets $\mathcal{Y}$ obtained from different values of $n$ are pairwise disjoint. Hence, denoting by $N$ the number of $n \in [x, 2x]$ with $\pi(n + n^c) - \pi(n) \le d'\frac{n^c}{\log x}$, the measure of $y$ satisfying \eqref{eq:pi-diff} with $d = 11d'$ is $\gg N \ell x^c/(\log \log x)^2 \gg Nx^c$. Assuming that \eqref{eq:S-diff} holds for $d$ small enough, for $d'$ small enough Lemma \ref{lem:to_buch} then gives $N = O(R)$.
\end{proof}

\section{From Buchstab sums to Dirichlet polynomials}
\label{sec:dir}

In this section we reduce the problem of obtaining asymptotics of type~\eqref{eq:asy_form} to the problem of bounding mean values of Dirichlet polynomials. 

Hence, we consider formulas of the form
\begin{align}
\label{eq:buch2}
\sum_{k \in \mathcal{A}(m)} c_k = \frac{\delta_0}{\delta_1} \sum_{k \in \mathcal{B}(m)} c_k + o\left(\frac{\delta_0 x}{\log x}\right),
\end{align}
where
\begin{align*}
c_k = \sum_{\substack{p_1, \ldots , p_n \in \mathbb{P} \\ p_i \in I_i \\ p_n < \ldots < p_1}} 1_{p_1 \cdots p_n \mid k}1_{p \mid k \implies p \ge z}
\end{align*}
and $I_i$ are intervals. Here and in what follows we will have $p_n \ge z$. The idea is to apply Proposition~\ref{prop:dir_to_arit} to reduce the problem to one on Dirichlet polynomials. However, a direct application of the proposition would not work, as the resulting Dirichlet polynomial $F(s)$ would not have certain desirable properties (such as factorizing as a product of shorter polynomials). Hence, we first have to ``clean up'' the sums before applying Proposition~\ref{prop:dir_to_arit}.

We first introduce some notation and preliminary tools, after which we perform the modifications on the sums.

We define
$$\xi(h) = 1_{p \mid h \implies p \ge z_1}$$
and
\begin{align}
\label{eq:def_xi_0}
\xi_0(h) = \begin{cases} \xi(h) & \text{ if } h < L_{\zeta} \\ \sum\limits_{\substack{d \mid h \\ d < z_2 \\ p \mid d \implies p < z_1}} \mu(d)  & \text{ otherwise.}\end{cases}
\end{align}

For $N \in \mathbb{Z}_+$, denote
$$\Delta(N) = \Delta(N, m) = 1_{N \in \mathcal{A}(m)} - \frac{\delta_0}{\delta_1}1_{N \in \mathcal{B}(m)}.$$
For a set $\mathcal{S} = \mathcal{S}(x) \subset \mathbb{R}_+$ of reals, let 
\begin{align*}
\Delta_{\mathcal{S}}(N, m) = \begin{cases} 0, \text{ if } d \in [sx^{-\eta}, sx^{\eta}] \text{ for some } d \mid N, s \in \mathcal{S} \\ \Delta(N, m) \text{ otherwise} \end{cases}
\end{align*}
and
\begin{align}
\label{def:delta_S'}
\Delta_{\mathcal{S}}'(N, m) = \begin{cases} 0, \text{ if } p_1p_2 \mid N \text{ for some } z_1 \le p_1, p_2 \le x^{c/2 - \epsilon} \text{ with } p_1/4 \le p_2 \le 4p_1 \\ \Delta_{\mathcal{S}}(N, m) \text{ otherwise}\end{cases}.
\end{align}
Hence $\Delta_{\mathcal{S}}'(N, m)$ removes those integers which have a divisor lying close to $s$ for some $s \in \mathcal{S}$ or which have two prime factors (of suitable size) which are almost equal in size.

In what follows we assume
\begin{align}
\label{eq:S_assume}
\sup_{s \in \mathcal{S}} s < x^{c - \epsilon} \quad \text{and} \quad |S| \ll (\log \log x)^5.
\end{align}
While the proofs in this section require no additional information on $\mathcal{S}$, we reveal that we will choose
\begin{align}
\label{eq:S_choice}
\mathcal{S} := \left(\{T^{2/n} \ | \ n \ge 4\} \cap [z_1, \infty)\right) \cup \{L_{\zeta}\},
\end{align}
so $\mathcal{S}$ will be of size $O(\log x / \log z_1) = O((\log \log x)^5)$ and $\sup_{s \in \mathcal{S}} s = \max(T^{1/2}, L_{\zeta}) < x^{c - \epsilon}$.

We first present some preliminary tools, after which we perform the modifications on the sums in~\eqref{eq:buch2}.

\subsection{Preliminary tools}

Many of the results and proofs of this section follow closely those given by Heath-Brown in~\cite{HB-V}, in particular Lemmas 3, 5, 6 and 8 there. 

We first note that the contribution of integers divisible by $p^2$ for some $p \ge L_{\zeta}$ to our sums is negligble.

\begin{lemma}
\label{lem:p^2}
Let $D \in \mathbb{Z}_+$ be a constant. We have, for all but $O(x^{\epsilon})$ integers $m \in [x/H', 3x/H']$,
\begin{align}
\label{eq:p^2}
\sum_{p \ge L_{\zeta}} \sum_{\substack{N \in \mathcal{A}(m) \\ p^2 \mid N}} \tau(N)^D = o\left(\frac{\delta_0 x}{\log x}\right)
\end{align}
The corresponding result holds with $\mathcal{A}$ and $\delta_0$ replaced by $\mathcal{B}$ and $\delta_1$.
\end{lemma}

\begin{proof}
We first note that the contribution of $p > 2x^{1/2}$ to the sum in \eqref{eq:p^2} is zero, so we may assume $p \le 2x^{1/2}$.

Then note that $L_{\zeta}^2 > |\mathcal{A}(m)|$, so that for any fixed $m$ and $p \ge L_{\zeta}$ there is at most one $N \in \mathcal{A}(m)$ with $p^2 \mid N$. Hence, bounding $\tau(N)^D \le x^{\epsilon}$,
$$\sum_{L_{\zeta} \le p < \delta_0 x^{1 - 2\epsilon}} \sum_{\substack{N \in \mathcal{A}(m) \\ p^2 \mid N}} \tau(N)^D \ll \sum_{p \le \delta_0 x^{1 - 2\epsilon}} x^{\epsilon} = o\left(\frac{\delta_0 x}{\log x}\right),$$
so the contribution of $p < \delta_0 x^{1 - 2\epsilon}$ is negligble.

Finally note that as $m \in [x/H', 3x/H']$ varies, the intervals $\mathcal{A}(m)$ are disjoint and lie in $[x, 4x]$. Hence the total contribution of a single value $p$ to sums as in \eqref{eq:p^2} is $O(x/p^2)$, so
$$\sum_{m \in [x/H', 3x/H']} \sum_{\delta_0 x^{1-2\epsilon} < p \le 2x^{1/2}} \sum_{\substack{N \in \mathcal{A}(m) \\ p^2 \mid N}} 1 \ll \sum_{\delta_0 x^{1-2\epsilon} < p \le 2x^{1/2}} \frac{x}{p^2} \ll x^{1 - c + 3\epsilon}.$$
Thus the number of $m$ for which \eqref{eq:p^2} does not hold is bounded by $x^{1 - c + 4\epsilon}/(\delta_0 x) \ll x^{5\epsilon}$, which is the desired bound up to redefining $\epsilon$.
\end{proof}

The second result is used to remove integers whose some divisor lies inconveniently close to an element of $\mathcal{S}$.

\begin{lemma}
\label{lem:close}
Let $D \in \mathbb{Z}_+$ be a constant. Assume that $\mathcal{S} \subset \mathbb{R}_+$ is as in~\eqref{eq:S_assume}. We have, for any $m \in [x/H', 3x/H']$,
\begin{align*}
\sum_{s \in \mathcal{S}} \sum_{\substack{d \in [sx^{-\eta}, sx^{\eta}]}} \sum_{\substack{N \in \mathcal{A}(m) \\ d \mid N \\ p \mid N \implies p \ge z_1}} \tau(N)^D = o\left(\frac{\delta_0 x}{\log x}\right).
\end{align*}
The corresponding result holds with $\mathcal{A}$ and $\delta_0$ replaced by $\mathcal{B}$ and $\delta_1$.
\end{lemma}

\begin{proof}
We consider each $s \in S$ individually, and hence have to show
\begin{align}
\label{eq:close}
\sum_{d \in [sx^{-\eta}, sx^{\eta}]} \sum_{\substack{N \in \mathcal{A}(m) \\ d \mid N \\ p \mid N \implies p \ge z_1}} \tau(N)^D = o\left(\frac{\delta_0 x}{|S| \log x}\right).
\end{align}
Write $N = dN'$ in the inner sum and bound $\tau(N) \le \tau(d)\tau(N')$. Applying Lemma \ref{lem:shiu} to the resulting sum over $N'$ (which by \eqref{eq:S_assume} is longer than $x^{\epsilon}$) we obtain
\begin{align*}
\sum_{d \in [sx^{-\eta}, sx^{\eta}]} \sum_{\substack{N \in \mathcal{A}(m) \\ d \mid N \\ p \mid N \implies p \ge z_1}} \tau(N)^D &\ll \sum_{\substack{d \in [sx^{-\eta}, sx^{\eta}] \\ p \mid d \implies p \ge z_1}} \tau(d)^D \frac{\delta_0 x / d}{\log x} \left(\frac{\log x}{\log z_1}\right)^{2^D} \\
&\ll \frac{\delta_0 x}{\log x}(\log \log x)^{O(1)} \sum_{\substack{d \in [sx^{-\eta}, sx^{\eta}] \\ p \mid d \implies p \ge z_1}} \frac{\tau(d)^D}{d}.
\end{align*}
We then perform a dyadic decomposition over $d$. The contribution of the interval $[w, 2w]$ to the sum is, again by Lemma \ref{lem:shiu}, bounded by
$$\ll \frac{1}{w} \cdot \frac{w}{\log X} \left(\frac{\log X}{\log z_1}\right)^{2^D} = \frac{(\log \log x)^{O(1)}}{\log x}.$$
Sum over $O(\eta \log x)$ values of $w$. The left hand side of \eqref{eq:close} is hence bounded by
\begin{align*}
\eta \frac{\delta_0 x}{\log x}(\log \log x)^{O(1)},
\end{align*}
which is sufficient, as $\eta = \exp(-(\log \log \log x)^2)$ and $|S| = (\log \log x)^{O(1)}$.

The proof for $\mathcal{B}$ is similar.
\end{proof}

The next result is similar and used to remove integers which have two (not too large) prime factors close to each other.

\begin{lemma}
\label{lem:dyadic}
For any $m \in [x/H', 3x/H']$ and any $D \in \mathbb{Z}_+$ we have
\begin{align*}
\sum_{\substack{p_1, p_2 \in \mathbb{P} \\ z_1 \le p_1, p_2 < x^{c/2 - \epsilon} \\ p_1/4 \le p_2 \le 4p_1}} \sum_{\substack{N \in \mathcal{A}(m) \\ p_1p_2 \mid N \\ p \mid N \implies p \ge z_1}} \tau(N)^D = o\left(\frac{x \delta_0}{\log x}\right). 
\end{align*}
The corresponding result holds with $\mathcal{A}$ and $\delta_0$ replaced by $\mathcal{B}$ and $\delta_1$.
\end{lemma}

\begin{proof}
Write $N = p_1p_2N'$ in the inner sum and bound $\tau(N)^D \ll \tau(N')^D$. Applying Lemma \ref{lem:shiu} to the resulting sum over $N'$ (which by $p_i \le x^{c/2 - \epsilon}$ is longer than $x^{\epsilon}$) we obtain
\begin{align*}
\sum_{\substack{p_1, p_2 \in \mathbb{P} \\ z_1 \le p_1, p_2 < x^{c/2 - \epsilon} \\ p_1/4 \le p_2 \le 4p_1}} \sum_{\substack{N \in \mathcal{A}(m) \\ p_1p_2 \mid N \\ p \mid N \implies p \ge z_1}} \tau(N)^D &\ll \sum_{\substack{p_1, p_2 \in \mathbb{P} \\ z_1 \le p_1, p_2 < x^{c/2 - \epsilon} \\ p_1/4 \le p_2 \le 4p_1}} \frac{\delta_0 x}{p_1p_2 \log x} \left(\frac{\log x}{\log z_1}\right)^{2^D} \\
&\ll \frac{\delta_0 x}{\log x}(\log \log x)^{O(1)} \sum_{\substack{p_1, p_2 \in \mathbb{P} \\ z_1 \le p_1, p_2 < x^{c/2 - \epsilon} \\ p_1/4 \le p_2 \le 4p_1}} \frac{1}{p_1p_2}.
\end{align*}
The sum over $p_1, p_2$ is bounded by
\begin{align*}
\ll \sum_{\substack{p_1 \in \mathbb{P} \\ z_1 \le p_1 < x^{c/2 - \epsilon}}} \frac{1}{p_1 \log p_1} \ll \frac{1}{\log z_1}.
\end{align*}
The result follows. The proof for $\mathcal{B}$ is similar.
\end{proof}

We then present Heath-Brown's identity (also known as the Heath-Brown decomposition).

\begin{lemma}
\label{lem:HB-dec}
Let $f : \mathbb{Z}_+ \to \mathbb{R}$ be an arbitrary function supported on $[1, 10x]$ and let $k \in \mathbb{Z}_+$ be fixed. Let $g(n) = \Lambda(n)1_{n \not\in \mathbb{P}}$. Assume that for any $N_1, \ldots N_{2k}$ and $N_1', \ldots , N_{2k}'$ and any $f_i \in \{1, \log, \mu, g\}$ satisfying $N_i > 5x^{1/k} \implies f_i \in \{1, \log, g\}$ we have
\begin{align*}
\sum_{\substack{n_1, \ldots , n_{2k} \\ N_i < n_i \le N_i'}} f_1(n_1) \cdots f_{2k}(n_{2k})f(n_1 \cdots n_{2k}) = o\left(\frac{\delta_0 x}{\log x}\right).
\end{align*}
Then
\begin{align*}
\sum_{p \le 10x} f(p) = o\left(\frac{\delta_0 x}{\log x}\right).
\end{align*}
\end{lemma}
There is of course nothing special with the bound $o(\delta_0 x/\log x)$. The function $g$ is an artifact arising from replacing $(\log p)1_{p \in \mathbb{P}}$ by $\Lambda(n)$. In practice when applying the Heath-Brown decomposition, the case where $f_i \in \{1, \log, \mu\}$ for all $i$ is the most difficult one. 

\begin{proof}
First, in order to evaluate $\sum f(p)$, it suffices to evaluate $\sum (\log p)f(p)$. More precisely, by partial summation we have
\begin{align*}
\sum_{p \ge 2} f(p) &= \sum_{p \ge 2} \frac{1}{\log p}(\log p)f(p) \\
&= -\int_{2}^{\infty} \frac{-1}{t(\log t)^2}\sum_{2 \le p \le t} (\log p)f(p) \d t,
\end{align*}
and so it suffices to show
\begin{align}
\label{eq:HB-dec-proof}
\sum_{p \le t} (\log p)f(p) = o\left(\frac{\delta_0 x}{\log x}\right)
\end{align}
for any $t$.

For this we use Heath-Brown's identity (see e.g.~\cite[(13.37)]{IK})
\begin{align*}
\Lambda(n) = \sum_{1 \le j \le k} (-1)^{j-1} \binom{k}{j} \sum_{m_1, \ldots , m_j \le 5X^{1/k}} \mu(m_1) \cdots \mu(m_j) \sum_{m_1 \cdots m_j n_1 \cdots n_j = n} \log n_1, \quad n \le 10x
\end{align*}
which allows us to write 
\begin{align*}
\sum_{n \in I} \Lambda(n)f(n),
\end{align*}
where $I$ is an interval, as $O(1)$ sums 
\begin{align*}
\sum_{\substack{n_1, \ldots , n_{2k} \\ n_i \in [N_i, N_i'] \\ n_1 \cdots n_{2k} \in I}} f_1(n_1) \cdots f_{2k}(n_{2k})f(n_1 \cdots n_{2k})
\end{align*}
for $f_i \in \{1, \mu, \log\}$, where $N_i \ge 5x^{1/k}$ implies $f_i \in \{1, \log\}$. Note that by splitting the sums if necessary we may assume that $N_i' \ge 5x^{1/k}$ implies $N_i \ge 5x^{1/k}$.

Thus, we have
\begin{align*}
\sum_{p \le t} (\log p)f(p) = &\sum_{\substack{(N_i, N_i', f_i)}} \sum_{\substack{n_1, \ldots , n_{2k} \\ n_i \in [N_i, N_i'] \\ n_1 \cdots n_{2k} \le t}} f_1(n_1) \cdots f_{2k}(n_{2k})f(n_1 \cdots n_{2k}) \\
- &\sum_{e \ge 2} \sum_{p^e \le t} (\log p)f(p^e).
\end{align*}
We note that the last sum is of the same form as the others, as we may write 
\begin{align*}
\sum_{e \ge 2} \sum_{p^e \le t} (\log p)f(p^e) = 
\sum_{\substack{n_1 \in [1, t] \\ n_2, \ldots , n_{2k} \in [1, 1] \\ n_1 \cdots n_{2k} \le t}}f_1(n_1)f_2(n_2) \cdots f_{k}(n_{2k})f(n_1 \cdots n_{2k})
\end{align*}
with $f_1 = g, f_2 = \ldots = f_{2k} = 1$. The result follows.
\end{proof}

The next result is used to replace the indicator function $\xi(h) = 1_{p \mid h \implies p \ge z_1}$ with the more convenient function $\xi_0(h)$ defined in \eqref{eq:def_xi_0}.

\begin{lemma}
\label{lem:xi_to_xi0}
For any $m \in [x/H', 3x/H']$ and any $D \in \mathbb{Z}_+$ we have
\begin{align}
\label{eq:xi_to_xi0}
&\sum_{h \in \mathbb{Z}_+} \sum_{\substack{N \in \mathcal{A}(m) \\ h \mid N}} |\xi(h) - \xi_0(h)| \tau(N)^D(\log N)^D = o\left(\frac{\delta_0 x}{\log x}\right).
\end{align}
The corresponding result holds with $\mathcal{A}$ and $\delta_0$ replaced by $\mathcal{B}$ and $\delta_1$.
\end{lemma}

\begin{proof}
By~\cite[Lemma 7]{HB-V} we have
\begin{align*}
|\xi(h) - \xi_0(h)| \le \sum_{\substack{d \mid (h, \Pi_1) \\ z_2 \le d < z_1z_2}} 1,
\end{align*}
where $\Pi_1 = \prod_{p < z_1} p$. Hence the left hand side of~\eqref{eq:xi_to_xi0} is bounded by
\begin{align}
\label{eq:xi_proof_1}
(\log x)^{O(1)} \sum_{\substack{d \mid \Pi_1 \\ z_2 \le d < z_1z_2}} \sum_{\substack{N \in \mathcal{A}(m) \\ d \mid N \\ p \mid N \implies p \ge z_1}} \tau(N)^{D+1}.
\end{align}
Write $N = dN'$ in the inner sum and bound $\tau(N) \le \tau(d)\tau(N')$. As $z_1z_2 < x^{\epsilon}$, we may apply Lemma~\ref{lem:shiu} to bound \eqref{eq:xi_proof_1} by
\begin{align}
\label{eq:xi_proof_2}
\delta_0 x(\log x)^{O(1)} \sum_{\substack{d \mid \Pi_1 \\ z_2 \le d < z_1z_2}} \frac{\tau(d)^{D+1}}{d}.
\end{align}
We bound the sum over $\tau(d)^{D+1}/d$ by Rankin's trick as in the proof of~\cite[Lemma 6]{HB-V}. For a parameter $\theta > 0$, we have
\begin{align*}
\sum_{\substack{d \mid \Pi_1 \\ z_2 \le d < z_1z_2}} \frac{\tau(d)^{D+1}}{d} &\le z_2^{-\theta} \sum_{\substack{d \mid \Pi_1 \\ z_2 \le d < z_1z_2}} \frac{\tau(d)^{D+1}}{d^{1 - \theta}} \\
&\le z_2^{-\theta} \sum_{\substack{d = 1 \\ d \mid \Pi_1}}^{\infty} \frac{\tau(d)^{D+1}}{d^{1 - \theta}} \\
&= z_2^{-\theta} \prod_{p < z_1} \left(1 + \frac{2^{D+1}}{p^{1 - \theta}}\right) \\
&\le z_2^{-\theta} \exp\left(2^{D+1} \sum_{p < z_1} \frac{1}{p^{1 - \theta}}\right).
\end{align*}
Choosing $\theta = 1/\log z_1$ we have, for $x$ large enough
\begin{align*}
\sum_{p < z_1} \frac{1}{p^{1 - \theta}} \le 3 \log \log z_1.
\end{align*}
It follows that \eqref{eq:xi_proof_2} is bounded by
$$\delta_0 x (\log x)^{O(1)} \sum_{\substack{d \mid \Pi_1 \\ z_2 \le d < z_1z_2}} \frac{\tau(d)^{D+1}}{d} \ll \delta_0 x(\log x)^{O(1)}z_2^{-1/\log z_1}.$$
This is sufficient, as $\log z_2 = (\log \log x)^2 \log z_1$.

The proof for $\mathcal{B}$ is similar.
\end{proof}

Finally, we use the following lemma to truncate a certain sum at $T^{1 + \epsilon}$.

\begin{lemma}
\label{lem:long_h}
Let $m \in [x/H', 3x/H']$ be given and fix $\epsilon > 0$. We have, for any fixed $D \in \mathbb{Z}_+$,
\begin{align*}
\sum_{\substack{N' \in \mathbb{Z}_+}} \tau(N')^D(\log N')^D \left|\sum_{h > T^{1 + \epsilon}} \xi_0(h)\Delta(hN')\right| = o\left(\frac{\delta_0 x}{\log x}\right).
\end{align*}
\end{lemma}

\begin{proof}
Let $V = T^{1 + \epsilon}$. We have, for any $N' \in \mathbb{Z}_+$,
\begin{align*}
\sum_{h > V} \xi_0(h)\Delta(hN') &\ll \sum_{d < z_2} \left|\sum_{g > V/d} \Delta(gdN')\right|.
\end{align*}
If $VN' > mH'(1 + \delta_1)$, then the inner sum is empty. If $VN' \le mH'$, then
\begin{align*}
\sum_{g > V/d} \Delta(gdN') = \left(\frac{mH'\delta_0}{dN'} + O(1)\right) - \frac{\delta_0}{\delta_1}\left(\frac{mH'\delta_1}{dN'} + O(1)\right) = O(1).
\end{align*}
If $mH' < VN' \le mH'(1 + \delta_1)$, then
\begin{align*}
\sum_{g > V/d} |\Delta(gdN')| \le \frac{mH'\delta_0}{dN'} + O(1) + \frac{\delta_0}{\delta_1}\left(\frac{mH'\delta_1}{dN'} + O(1)\right) \ll \frac{mH'\delta_0}{dN'} + O(1).
\end{align*}
It follows that
\begin{align*}
&\sum_{\substack{N' \in \mathbb{Z}_+}} \tau(N')^D(\log N')^D \left|\sum_{h > T^{1 + \epsilon}} \xi_0(h)\Delta(hN)\right| \ll \\
&\sum_{N' \le x/T^{1+\epsilon/2}} z_2 \tau(N')^D(\log N')^D + \sum_{d < z_2} \sum_{mH'/V < N' \le mH'(1 + \delta_1)/V} \frac{mH'\delta_0}{dN'}\tau(N')^D(\log N')^D.
\end{align*}
The first sum gives a power saving bound over $\delta_0 x/\log x$. We bound the second sum as
\begin{align*}
&\frac{mH'\delta_0(\log x)^{D}}{mH'/V} \sum_{d < z_2} \frac{1}{d} \sum_{mH'/V < N' \le mH'(1 + \delta_1)/V} \tau(N')^D \ll \\
&V\delta_0(\log x)^D(\log z_2) \sum_{\substack{mH'/V < N' \le mH'(1 + \delta_1)/V}} \tau(N')^D,
\end{align*}
and apply Lemma~\ref{lem:shiu} to the sum over $N'$ to arrive at
\begin{align*}
V\delta_0 (\log x)^D (\log z_2) \frac{mH' \delta_1}{V}(\log x)^{O(1)} \ll \delta_0 x \delta_1 (\log x)^{O(1)},
\end{align*}
which is sufficient, as $\delta_1 \ll (\log x)^{-A}$ for any $A \ge 1$.
\end{proof}

\subsection{Modification of the sums}

We consider asymptotics of the form
\begin{align}
\label{eq:mod_1}
\sum_{\substack{p_1, \ldots , p_n \\ p_i \in I_i \\ p_n < \ldots < p_1}} S(\mathcal{A}_{p_1 \cdots p_n}(m), z) - \frac{\delta_0}{\delta_1} S(\mathcal{B}_{p_1 \cdots p_n}(m), z) = o\left(\frac{\delta_0 x}{\log x}\right)
\end{align}
and
\begin{align}
\label{eq:mod_2}
\sum_{\substack{p_1, \ldots , p_n \\ p_i \in I_i \\ p_n < \ldots < p_1}} S(\mathcal{A}_{p_1 \cdots p_n}(m), p_n) - \frac{\delta_0}{\delta_1}S(\mathcal{B}_{p_1 \cdots p_n}(m), p_n) = o\left(\frac{\delta_0 x}{\log x}\right)
\end{align}
for $m \in [x/H', 3x/H']$ and intervals $I_i \subset [z_1, 10\sqrt{x}]$. In~\eqref{eq:mod_1} we will assume $p_n \ge z$, that is, $I_i \subset [z, 10\sqrt{x}]$. We will always have $z \ge z_1$. 

Our aim is to reduce the statements \eqref{eq:mod_1} and \eqref{eq:mod_2} to statements regarding mean values of Dirichlet polynomials via Proposition \ref{prop:dir_to_arit}. Before applying Proposition \ref{prop:dir_to_arit} we perform several modifications to the sums for the resulting Dirichlet polynomials to have certain desirable properties. For convenience we will mainly consider sums of the form~\eqref{eq:mod_1}, as the sum~\eqref{eq:mod_2} may be handled via similar methods (see Remark~\ref{rem:z=p_n}). We will perform the following modifications to~\eqref{eq:mod_1}.
\begin{itemize}
\item Handle integers divisible by $p^2$ for a large prime $p$.
\item Write the condition on $z$-roughness as sums over integers via Möbius inversion.
\item Discard cases where some product lies close to $s \in \mathcal{S}$ or where we have two prime factors close to each other (i.e. replace $\Delta$ with $\Delta_{\mathcal{S}}'$).
\item Apply Heath-Brown's identity to certain sums.
\item Replace occurrences of $\xi$ with $\xi_0$.
\item Restrict the size of a certain variable.
\item Decompose a certain sum as sums over primes.
\item Perform dyadic decomposition and remove cross conditions.
\end{itemize}
These steps are undertaken in Lemmas~\ref{lem:mod_qrh} to \ref{lem:mod_cross} below.

First, reduce to $m$ satisfying \eqref{eq:p^2} (and the similar conclusion for $\mathcal{B}$ and $\delta_1$). To this end, we let $\mathcal{M} \subset [x/H', 3x/H']$ denote the set of $m$ for which
$$\sum_{\substack{p_1, \ldots , p_n \in \mathbb{P} \\ p_i \in I_i \\ p_n < \ldots < p_1}} \sum_{\substack{\ell \in \mathbb{Z}_+ \\ \exists p \ge L_{\zeta} : p^2 \mid \ell}} \Delta(p_1 \cdots p_n\ell, m) = o\left(\frac{\delta_0 x}{\log x}\right).$$
By Lemma \ref{lem:p^2}, $\mathcal{M}$ contains all but $O(x^{\epsilon}) = O(R)$ values of $m \in [x/H', 3x/H']$.

We then write the condition on $z$-roughness in~\eqref{eq:mod_1} in a more convenient form.

\begin{lemma}
\label{lem:mod_qrh}
Let $n = O(1)$, $z \in [z_1, 10\sqrt{x}]$, intervals $I_1, \ldots , I_n \subset [z, 10\sqrt{x}]$ and $m \in \mathcal{M}$ be given. Assume that for any $0 \le n' \le 4$ we have
\begin{align*}
\sum_{\substack{p_1, \ldots , p_n \in \mathbb{P} \\ p_i \in I_i \\ p_n < \ldots < p_1}} \sum_{\substack{r \in \mathbb{Z}_+ \\ p \mid r \implies \\ z_1 < p < \min(z, L_{\zeta})}} \sum_{\substack{q_1, \ldots , q_{n'} \in \mathbb{P} \\ \min(z, L_{\zeta}) \le q_i < z \\ q_{n'} < \ldots < q_1}} \sum_{h \in \mathbb{Z}_+} \mu(r) \xi(h) \Delta(p_1 \cdots p_n r q_1 \cdots q_{n'} h, m) = o\left(\frac{\delta_0 x}{\log x}\right).
\end{align*}
Then~\eqref{eq:mod_1} holds.
\end{lemma}

\begin{proof}
Note that
\begin{align*}
& \sum_{\substack{p_1, \ldots , p_n \\ p_i \in I_i \\ p_n < \ldots < p_1}} S(\mathcal{A}_{p_1 \cdots p_n}(m), z) - \frac{\delta_0}{\delta_1} S(\mathcal{B}_{p_1 \cdots p_n}(m), z) \\
= & \sum_{\substack{p_1, \ldots , p_n \\ p_i \in I_i \\ p_n < \ldots < p_1}} \sum_{\substack{\ell \in \mathbb{Z}_+ \\ p \mid \ell \implies p \ge z}} \Delta(p_1 \cdots p_n \ell).
\end{align*}
As $m \in \mathcal{M}$, we may reduce to $\ell$ not divisible by $p^2$ for any $p \ge L_{\zeta}$. For such $\ell$ we then have
\begin{align*}
1_{p \mid \ell \implies p \ge z} &= \sum_{\substack{r' \in \mathbb{Z}_+ \\ p \mid r' \implies z_1 \le p < z}} \sum_{\substack{h \in \mathbb{Z}_+ \\ p \mid h \implies p \ge z_1}} \mu(r')1_{hr' = \ell} \\
&= \sum_{n' \le 4} \sum_{\substack{q_1, \ldots , q_{n'} \in \mathbb{P} \\ \min(z, L_{\zeta}) \le q_i < z \\ q_{n'} < \ldots < q_1}} \sum_{\substack{r \in \mathbb{Z}_+ \\ p \mid r \implies z_1 \le p < \min(z, L_{\zeta})}} \sum_{h \in \mathbb{Z}_+} (-1)^{n'}\xi(h)1_{q_1 \cdots q_{n'}hr = \ell},
\end{align*}
from which \eqref{eq:mod_1} follows.
\end{proof}

The next step is to discard cases where some product of the numbers is close to elements of $s \in \mathcal{S}$ or which have two not too large prime factors lying close to each other. In other words, we replace $\Delta$ with $\Delta_{\mathcal{S}}'$.

\begin{lemma}
\label{lem:mod_S}
Let $n, I_i, z$ and $m$ be as in Lemma~\ref{lem:mod_qrh}. Let $\mathcal{S}$ be as in \eqref{eq:S_choice}. Assume that for any $n' \le 4$ we have
\begin{align}
\label{eq:mod_S}
\sum_{\substack{p_1, \ldots , p_n \\ p_i \in I_i \\ p_n < \ldots < p_1}} \sum_{\substack{r \in \mathbb{Z}_+ \\ p \mid r \implies \\ z_1 \le p < \min(z, L_{\zeta})}} \sum_{\substack{q_1, \ldots , q_{n'} \in \mathbb{P} \\ \min(z, L_{\zeta}) \le q_i < z \\ q_{n'} < \ldots < q_1}} \sum_{h \in \mathbb{Z}_+} \mu(r) \xi(h) \Delta_{\mathcal{S}}'(p_1 \cdots p_n r q_1 \cdots q_{n'} h, m) = o\left(\frac{\delta_0 x}{\log x}\right).
\end{align}
Then~\eqref{eq:mod_1} holds.
\end{lemma}

\begin{proof}
We show the sum in Lemma~\ref{lem:mod_qrh} is $o(\delta_0 x/\log x)$ assuming \eqref{eq:mod_S}. First replace $\Delta$ with $\Delta_{\mathcal{S}}$, the difference being bounded in absolute value by
\begin{align*}
\sum_{s \in \mathcal{S}} \sum_{d \in [sx^{-\eta}, sx^{\eta}]} \sum_{\substack{N \in \mathcal{A}(m) \\ d \mid N \\ p \mid N \implies p \ge z_1}} \tau(N)^{O(1)} + \frac{\delta_0}{\delta_1} \sum_{s \in \mathcal{S}} \sum_{d \in [sx^{-\eta}, sx^{\eta}]} \sum_{\substack{N \in \mathcal{B}(m) \\ d \mid N \\ p \mid N \implies p \ge z_1}} \tau(N)^{O(1)},
\end{align*}
which by Lemma~\ref{lem:close} is $o(\delta_0 x/\log x)$. We then replace $\Delta_{\mathcal{S}}$ by $\Delta'_{\mathcal{S}}$, the error being similarly bounded by
\begin{align*}
\sum_{\substack{p_1, p_2 \in \mathbb{P} \\ z_1 \le p_1, p_2 < x^{c/2 - \epsilon} \\ p_1/4 \le p_2 \le 4p_1}} \sum_{\substack{N \in \mathcal{A}(m) \\ p_1p_2 \mid N \\ p \mid N \implies p \ge z_1}} \tau(N)^{O(1)} + \frac{\delta_0}{\delta_1} \sum_{\substack{p_1, p_2 \in \mathbb{P} \\ z_1 \le p_1, p_2 < x^{c/2 - \epsilon} \\ p_1/4 \le p_2 \le 4p_1}} \sum_{\substack{N \in \mathcal{B}(m) \\ p_1p_2 \mid N \\ p \mid N \implies p \ge z_1}} \tau(N)^{O(1)},
\end{align*}
which by Lemma~\ref{lem:dyadic} is small enough. The result follows.
\end{proof}

We then decompose the sums over $p_1, \ldots , p_n$ and $q_1, \ldots , q_{n'}$ by the Heath-Brown decomposition.

\begin{lemma}
\label{lem:mod_hbd}
Let $n = O(1)$, $z \in [z_1, 10\sqrt{x}]$, intervals $I_1, \ldots , I_n \subset [z, 10\sqrt{x}]$ and $m \in \mathcal{M}$ be given. Write $I_{i} = [\min(z, L_{\zeta}), z)$ for $i > n$. Assume that, for
\begin{enumerate}
\item any $n' \le 4$,
\item any $N_{i, j}$, where $1 \le i \le n+n', 1 \le j \le 8$, and any $N_{i, j}' > N_{i, j}$, and
\item any $f_{i, j}$, where $1 \le i \le n+n', 1 \le j \le 8$, with $f_{i, j} \in \{\xi, \xi \cdot \log, \xi \cdot g, \xi \cdot \mu\}$ and $N_{i, j} > L_{\zeta} \implies f_{i, j} \neq \xi \cdot \mu$,
\end{enumerate}
we have
\begin{align*}
\sum_{\substack{N_{i, j} < n_{i, j} \le N_{i, j}' \\ n_{i, 1} \cdots n_{i, 8} \in I_i}}^{\ast} \prod_{\substack{1 \le i \le n+n' \\ 1 \le j \le 8}} f_{i, j}(n_{i, j})  \sum_{\substack{r \in \mathbb{Z}_+ \\ p \mid r \implies z_1 \le p < \min(z, L_{\zeta})}} \sum_{h \in \mathbb{Z}_+} \mu(r) \xi(h) \Delta_{\mathcal{S}}'\left(rh \prod_{\substack{1 \le i \le n + n' \\ 1 \le j \le 8}} n_{i, j}, m\right) \\
= o\left(\frac{\delta_0 x}{\log x}\right),
\end{align*}
where the asterisk $\ast$ denotes that the sums is only over $n_{i, j}$ satisfying
$$\prod_{1 \le j \le 8} n_{i, j} < \prod_{1 \le j \le 8} n_{i-1, j} \qquad \text{for all } i \in \{2, 3, \ldots , n+n'\} \setminus \{n+1\}.$$
Then~\eqref{eq:mod_1} holds.
\end{lemma}

\begin{remark}
\label{rem:optional}
While in this lemma we have decomposed the sums over all of $p_1, \ldots , p_n$, $q_1, \ldots,  q_{n'}$, we could choose to decompose the sums merely over a (possibly empty) subset of them. For clarity, we state the results here and below for the case where all the sums have been decomposed, understanding that we have this additional flexibility.
\end{remark}

\begin{proof}
We start from the sum in Lemma~\ref{lem:mod_S}. We apply Heath-Brown's decomposition (Lemma~\ref{lem:HB-dec}) with $k = 4$ to the sums over $p_1, \ldots , p_n, q_1, \ldots , q_{n'}$ one by one. For example, applying the decomposition to the sum over $p_1$ we take $f$ in Lemma~\ref{lem:HB-dec} to be
\begin{align*}
&f(\ell) = \\
&\xi(\ell)1_{\ell \in I_1} \sum_{\substack{p_2, \ldots , p_n \\ p_i \in I_i \\ p_n < \ldots < p_2 < \ell}} \sum_{\substack{r \\ p \mid r \implies \\ z_1 \le p < \min(z, L_{\zeta})}} \sum_{\substack{q_1, \ldots , q_{n'} \in \mathbb{P} \\ \min(z, L_{\zeta}) \le q_i < z \\ q_{n'} < \ldots < q_1}} \sum_{h \in \mathbb{Z}_+} \mu(r) \xi(h) \Delta'_{\mathcal{S}}(\ell p_2 \cdots p_n r q_1 \cdots q_{n'} h, m).
\end{align*}
Hence, the sum over $f(p_1)$ may be converted to a sum over $n_{1, 1}, \ldots,  n_{1, 8}$ as in Lemma~\ref{lem:HB-dec}. Note that $\xi(\ell)$ is completely multiplicative.

Performing the decomposition for all $p_1, \ldots,  p_n, q_1, \ldots , q_{n'}$ gives the result. Note that the cross conditions $p_{i} < p_{i-1}$ and $q_i < q_{i-1}$ transform to cross conditions of the form $n_{i, 1} \cdots n_{i, 8} < n_{i-1, 1} \cdots n_{i-1, 8}$.

Considering the implication $N_{i, j} > L_{\zeta} \implies f_{i, j} \neq \xi \cdot \mu$, note that for $c = 0.45$, we have simply relaxed the condition $N_i \ge 5x^{1/4}$ in Lemma \ref{lem:HB-dec} to $N_i \ge L_{\zeta}$. For $c = 0.5$, note that the the sum over $n_{i, j} \in [x^{1/4}, 5x^{1/4}]$ is empty as $x^{1/4} \in \mathcal{S}$, and we may thus assume $N_{i, j} > x^{1/4}$ implies $N_{i, j} > 5x^{1/4}$ and hence $f_{i, j} \neq \xi \cdot \mu$.
\end{proof}

Then we replace each occurrence of $\xi$ with $\xi_0$, as the latter will be more convenient to work with.

\begin{lemma}
\label{lem:mod_xi0}
Let $n = O(1)$, $z \in [z_1, 10\sqrt{x}]$, intervals $I_1, \ldots , I_n \subset [z, 10\sqrt{x}]$ and $m \in \mathcal{M}$ be given. Write $I_i = [\min(z, L_{\zeta}), z)$ for $i > n$. Assume that, for
\begin{enumerate}
\item any $n' \le 4$,
\item any $N_{i, j}, 1 \le i \le n+n', 1 \le j \le 8$ and $N_{i, j}' > N_{i, j}$, and
\item any $f_{i, j}, 1 \le i \le n+n', 1 \le j \le 8$ with $f_{i, j} \in \{\xi_0, \xi_0 \cdot \log, \xi_0 \cdot g, \xi_0 \cdot \mu\}$ and $N_{i, j} > L_{\zeta} \implies f_{i, j} \neq \xi_0 \cdot \mu$,
\end{enumerate}
we have
\begin{align*}
\sum_{\substack{N_{i, j} < n_{i, j} \le N_{i, j}' \\ n_{i, 1} \cdots n_{i, 8} \in I_i \\ \text{for } 1 \le i \le n+n', 1 \le j \le 8}}^{\ast} \prod_{\substack{1 \le i \le n+n' \\ 1 \le j \le 8}} f_{i, j}(n_{i, j})  \sum_{\substack{r \in \mathbb{Z}_+ \\ p \mid r \implies \\ z_1 \le p < \min(z, L_{\zeta})}} \sum_{h \in \mathbb{Z}_+} \mu(r) \xi_0(h) \Delta_{\mathcal{S}}'\left(rh \prod_{\substack{1 \le i \le n + n' \\ 1 \le j \le 8}} n_{i, j}, m\right) \\
= o\left(\frac{\delta_0 x}{\log x}\right),
\end{align*}
Then~\eqref{eq:mod_1} holds.
\end{lemma}

\begin{proof}
We start from a sum as in Lemma~\ref{lem:mod_hbd} and replace occurrences of $\xi$ with $\xi_0$ one by one. Note that $|\xi_0(k)| \le \tau(k)$, the sum is over $O(1)$ variables $n_{i, j}, r, h$, and any $N$ such that $\Delta_{\mathcal{S}}'(N, m)$ has non-zero coefficient in the sum is $z_1$-rough. As $\log, g, \mu$ are bounded by $\log$, it follows that every replacement of $\xi$ with $\xi_0$ induces an error bounded by
\begin{align*}
&\sum_{\substack{h' \in \mathbb{Z}_+}}  |\xi(h') - \xi_0(h')| \sum_{\substack{N \in \mathcal{A}(m) \\ h' \mid N}} \tau(N)^{O(1)}(\log N)^{O(1)} \ + \\
\frac{\delta_0}{\delta_1} &\sum_{\substack{h' \in \mathbb{Z}_+}}  |\xi(h') - \xi_0(h')| \sum_{\substack{N \in \mathcal{B}(m) \\ h' \mid N}} \tau(N)^{O(1)}(\log N)^{O(1)}. 
\end{align*}
This error is small enough by Lemma~\ref{lem:xi_to_xi0}.
\end{proof}

Then we restrict the sum over $h$ to $h \le T^{1+\epsilon}$.

\begin{lemma}
\label{lem:mod_hlarge}
Let $n, I_i, z, m$ be as in Lemma \ref{lem:mod_hbd}. Assume that for any $n', N_{i, j}$ and $f_{i, j}$ as in Lemma~\ref{lem:mod_hbd} we have
\begin{align*}
\sum_{\substack{N_{i, j} < n_{i, j} \le N_{i, j}' \\ n_{i, 1} \cdots n_{i, 8} \in I_i \text{ for} \\ 1 \le i \le n+n', 1 \le j \le 8}}^{\ast} \prod_{\substack{1 \le i \le n+n' \\ 1 \le j \le 8}} f_{i, j}(n_{i, j})  \sum_{\substack{r \in \mathbb{Z}_+ \\ p \mid r \implies \\ z_1 \le p < \min(z, L_{\zeta})}} \sum_{h \le T^{1 + \epsilon}} \mu(r) \xi_0(h) \Delta'_{\mathcal{S}}\left(rh \prod_{\substack{1 \le i \le n + n' \\ 1 \le j \le 8}} n_{i, j}, m\right) \\
= o\left(\frac{\delta_0 x}{\log x}\right).
\end{align*}
Then~\eqref{eq:mod_1} holds.
\end{lemma}

\begin{proof}
We start from the sum in Lemma~\ref{lem:mod_xi0}. Let $V = T^{1+\epsilon}$. Consider separately the sums over $n_{i, j}, r$ and the sum over $h$. The sum over $n_{i, j}, r$ has $O(1)$ variables, the coefficients of each variable being bounded by $\tau \cdot \log$. The contribution of $h > V$ is thus bounded by
\begin{align*}
\sum_{N' \le 10x/V} \tau(N')^{O(1)}\log(N')^{O(1)}  \left|\sum_{h > V} \xi_0(h)\Delta(N'h)\right|.
\end{align*}
This is small enough by Lemma~\ref{lem:long_h}.
\end{proof}

We then write the sum over $r$ in a more convenient form as multiple sums over primes.

\begin{lemma}
\label{lem:mod_ri}
Let $n = O(1)$, $z \in [z_1, 10\sqrt{x}]$, intervals $I_1, \ldots , I_n \subset [z, 10\sqrt{x}]$ and $m \in \mathcal{M}$ be given. Write $I_{i} = [\min(z, L_{\zeta}), z)$ for $i > n$. Assume that, for
\begin{enumerate}
\item any $n' \le 4$,
\item any $N_{i, j}, 1 \le i \le n+n', 1 \le j \le 8$ and $N_{i, j}' > N_{i, j}$,
\item any $f_{i, j}, 1 \le i \le n+n', 1 \le j \le 8$ with $f_{i, j} \in \{\xi_0, \xi_0 \cdot \log, \xi_0 \cdot g, \xi_0 \cdot \mu\}$ and $N_{i, j} > L_{\zeta} \implies f_{i, j} \neq \xi_0 \cdot \mu$,
\item any $t \in \mathbb{Z}_{\ge 0}$, and
\item any $R_i \in [z_1, \min(z, L_{\zeta}))$ and $R_i' \ge R_i$ where $1 \le i \le t$,
\end{enumerate}
we have
\begin{align*}
\sum_{\substack{N_{i, j} < n_{i, j} \le N_{i, j}' \\ n_{i, 1} \cdots n_{i, 8} \in I_i \\ \text{for } 1 \le i \le n+n', 1 \le j \le 8}}^{\ast} \prod_{\substack{1 \le i \le n+n' \\ 1 \le j \le 8}} f_{i, j}(n_{i, j}) \sum_{\substack{r_1, \ldots , r_t \in \mathbb{P} \\ r_i \in [R_i, R_i'] \\ r_t < \ldots < r_1}} \sum_{h \le T^{1 + \epsilon}} \xi_0(h) \Delta'_{\mathcal{S}}\left(h\prod_{\substack{1 \le i \le n + n' \\ 1 \le j \le 8}} n_{i, j} \prod_{1 \le i \le t} r_i, m\right) \\
= O\left(\frac{\delta_0 x}{(\log x)^2}\right).
\end{align*}
Then~\eqref{eq:mod_1} holds.
\end{lemma}

\begin{proof}
The result follows from the identity
\begin{align*}
\mu(r)1_{p \mid r \implies z_1 \le p < \min(z, L_{\zeta})} = \sum_{0 \le t \le \frac{\log 10x}{\log z_1}} \sum_{\substack{r_1, \ldots , r_t \in \mathbb{P} \\ r_t < \ldots < r_1 \\ r_1 < \min(z, L_{\zeta}) \\ r_t \ge z_1}} (-1)^t 1_{r_1 \cdots r_t = r}.
\end{align*}
\end{proof}

Finally, we perform a dyadic decomposition and remove the cross conditions.

\begin{lemma}
\label{lem:mod_cross}
Let $n = O(1)$, $z \in [z_1, 10\sqrt{x}]$, intervals $I_1, \ldots , I_n \subset [z, 10\sqrt{x}]$ and $m \in \mathcal{M}$ be given. Write $I_{i} = [\min(z, L_{\zeta}), z)$ for $i > n$. Assume that for
\begin{enumerate}
\item any $n' \le 4$,
\item any $J_1, \ldots , J_{n+n'} \le 8$,
\item any $N_{i, j} \ge z_1, 1 \le i \le n+n', 1 \le j \le J_i$ and $N_{i, j}' \in [N_{i, j}, 2N_{i, j}]$,
\item any $f_{i, j}, 1 \le i \le n+n', 1 \le j \le J_i$ with $f_{i, j} \in \{\xi_0, \xi_0 \cdot \log, \xi_0 \cdot g, \xi_0 \cdot \mu\}$ and $N_{i, j} > L_{\zeta} \implies f_{i, j} \neq \xi_0 \cdot \mu$,
\item any $t \in \mathbb{Z}_{\ge 0}$, and
\item any $R_i \in [z_1, \min(z, L_{\zeta})]$ and $R_i' \in [R_i, 2R_i]$ where $1 \le i \le t$
\end{enumerate}
such that the intervals $[R_i, R_i']$ are pairwise disjoint and 
\begin{align}
\label{eq:prod_interval}
\prod_{1 \le j \le J_i} N_{i, j} \in \left[\frac{\inf(I_i)}{2^{8}}, \sup(I_i)\right]
\end{align}
we have
\begin{align*}
\sum_{\substack{N_{i, j} < n_{i, j} \le N_{i, j}' \\ \text{for } 1 \le i \le n+n', 1 \le j \le 8}} \prod_{\substack{1 \le i \le n+n' \\ 1 \le j \le J_i}} f_{i, j}(n_{i, j}) \sum_{\substack{r_1, \ldots , r_t \in \mathbb{P} \\ r_i \in [R_i, R_i']}} \sum_{h \le T^{1 + \epsilon}} \xi_0(h) \Delta'_{\mathcal{S}}\left(\prod_{\substack{1 \le i \le n + n' \\ 1 \le j \le 8}} n_{i, j} \prod_{1 \le i \le t} r_i h, m\right) \\
\all \delta_0 x.
\end{align*}
Then~\eqref{eq:mod_1} holds.
\end{lemma}

Recall the notation $\all$ from \eqref{eq:all}.

\begin{proof}
Note that the result is trivial if $t > \frac{\log 10x}{\log z_1}$, so assume $t = O((\log \log x)^5)$.

Consider the sum in Lemma~\ref{lem:mod_ri}. By a dyadic decomposition on the $O(1)$ sums over $n_{i, j}$ and $t$ sums over $r_i$, it suffices to obtain a bound of $S^{-\epsilon}\delta_0 x$ for sums of the form
\begin{align*}
\sum_{\substack{N_{i, j} < n_{i, j} \le N_{i, j}' \\ n_{i, 1} \cdots n_{i, J_i} \in I_i}}^{\ast} \prod_{\substack{1 \le i \le n+n' \\ 1 \le j \le J_i}} f_{i, j}(n_{i, j}) \sum_{\substack{r_1, \ldots , r_t \in \mathbb{P} \\ r_i \in [R_i, R_i'] \\ r_t < \ldots < r_1}} \sum_{h \le T^{1 + \epsilon}} \xi_0(h) \Delta'_{\mathcal{S}}\left(\prod n_{i, j} \prod r_i h, m\right)
\end{align*}
with $N_{i, j}' \le 2N_{i, j}$ and $R_i' \le 2R_i$. We may assume~\eqref{eq:prod_interval}, as otherwise the sum is empty and the result is trivial. Similarly, we may assume $R_{i+1} \le R_i'$.

By the definition \eqref{def:delta_S'} of $\Delta_{\mathcal{S}}'$ we may assume that for any $i$ with $R_i < x^{c/2 - \epsilon}$ the intervals $[R_i, R_i']$ and $[R_{i+1}, R_{i+1}']$ are disjoint, as otherwise the sum is empty. Hence we may assume $R_{i+1}' < R_i$ for any $i \ge 5$ (say). The cross conditions $r_{i+1} < r_i$ now follow automatically for $i \ge 5$.

We may remove the $O(1)$ cross conditions on $n_{i, j}$ and any possible $O(1)$ cross conditions on $r_i$ at $(\log x)^{O(1)} = S^{o(1)}$ by decomposing the sums into short intervals (cf. \cite[Section 3.2]{harman}). We may further assume $R_i, N_{i, j} \ge z_1$ as otherwise the corresponding sums are empty or over the set $\{1\}$. As a consequence, the number of variables $n_{i, j}$ may decrease -- hence the new parameters $J_1, \ldots , J_{n + n'}$.
\end{proof}

In terms of Dirichlet polynomials, one could describe our procedure as follows: we wish to establish the assumption of Proposition~\ref{prop:dir_to_arit} for a Dirichlet polynomial of the form
$$P_1(s) \cdots P_n(s)Q(s),$$
where $P_i(s) = \sum_{p \sim P_i} p^{-s}$ correspond to sums over $p_i$ (after a dyadic decomposition) in~\eqref{eq:mod_1} and $Q$ is some Dirichlet polynomial (itself equal to a certain product). We perform the Heath-Brown decomposition for $P_1, \ldots , P_n$, and similarly also decompose $Q$. Hence, we consider polynomials of the form
\begin{align}
\label{def:f}
F(s) = \prod_{\substack{1 \le i \le n+n' \\ 1 \le j \le J_i}} N_{i, j}(s) \prod_{1 \le i \le t} R_i(s) H(s).
\end{align}
We will apply Proposition~\ref{prop:dir_to_arit}. We note that each application of Proposition \eqref{eq:dir_to_arit} loses an exceptional set of $O(R')$ values of $m$, where $R'$ is the $R$-parameter in Proposition~\ref{prop:dir_to_arit}. As we will be applying the proposition more than $O(1)$ times, namely for $S^{o(1)}$ sums obtained by different choices of parameters in (1)--(6) in Lemma \ref{lem:mod_cross}, we have to take $R'$ slightly smaller than $R$ defined as in \eqref{eq:def-R}. The choice $R' = Rx^{-\nu/2}$ works. By redefining $\nu$ as $\nu/2$, we may talk about applying Proposition~\ref{prop:dir_to_arit} with the value of $R$ defined in \eqref{eq:def-R}.

Hence, our task is to show~\eqref{eq:dir_to_arit}. We have the following information on our polynomials (see Lemma \ref{lem:mod_cross}).

\begin{information}
\label{info}
The polynomials in \eqref{def:f} satisfy the following properties.
\begin{itemize}
\item For $i \le n$, the product of $N_{i, j}(s), 1 \le j \le J_i$ has length (approximately) equal to $P_i$. (For convenience, from now on we will write $\prod N_{i, j} = P_i$ or $\prod N_{i, j} \in [P_i, 2P_i]$ instead of the precise condition \eqref{eq:prod_interval}.)
\item All of $N_{i, j}(s), n+1 \le i \le n+n'$ have length lying in $[x^{1/4}, z]$ (and thus $n' = 0$ if $x^{1/4} > z$).
\item All of $R_i(s)$ have length bounded by $\min(x^{1/4}, z)$.
\item $H(s)$ is bounded in length by $T^{1+\epsilon}$.
\item Any polynomial is longer than $z_1$.
\item No product of the polynomials is close to $s, s \in \mathcal{S}$.
\item The coefficients of any polynomial are given by one of the functions $\xi_0, \xi_0 \cdot \log, \xi_0 \cdot g, \xi_0 \cdot \mu, 1_{\mathbb{P}}$.
\item The coefficients of any polynomial longer than $L_{\zeta}$ are given by $\xi_0, \xi_0 \cdot \log$ or $\xi_0 \cdot g$.
\item The coefficients of $F(s)$ are supported on the interval
$$[x/2^{\log x / \log z_1 - C}, x2^{\log x / \log z_1 + C}]$$
for some constant $C > 0$. Note that $2^{\log x / \log z_1} = S^{o(1)}$.
\end{itemize}
\end{information}

We remind that performing the Heath-Brown decomposition to any given polynomial is optional (Remark~\ref{rem:optional}). Our aim is to determine sufficient conditions for the lengths of our polynomials so that $F(s)$ satisfies the assumption of Proposition~\ref{prop:dir_to_arit}.

We highlight a particularly important class of Dirichlet polynomials. (Recall the definition of $\xi_0$ from \eqref{eq:def_xi_0} and $g$ from Lemma \ref{lem:HB-dec}.)

\begin{definition}
\label{def:zeta_sum}
A Dirichlet polynomial $P(s)$ is a \emph{zeta sum} if $P \ge L_{\zeta}$ and the coefficients of $P(s)$ are given by one of the functions $\xi_0, \xi_0 \cdot \log$ and $\xi_0 \cdot g$.
\end{definition}
Note that our definition is nonstandard, as zeta sums usually refer to Dirichlet polynomials whose coefficients are given by $1$ or $\log$. However, for our purposes coefficients $\xi_0, \xi_0 \cdot \log$ or $\xi_0 \cdot g$ work essentially as well as coefficients $1$ or $\log$ -- one only needs care when applying Proposition~\ref{prop:HB_MVT}, as the maximum of $\xi_0$ is quite large, in contrast to $1$ or $\log$.

In what follows we will choose $\mathcal{S}$ as in~\eqref{eq:S_choice}. We have included $L_{\zeta}$ in $\mathcal{S}$ mainly for convenience.

\begin{remark}
\label{rem:z=p_n}
We have above reduced showing~\eqref{eq:mod_1} to establishing the assumption of Lemma~\ref{lem:mod_cross}. The procedure adapts with slight modifications to~\eqref{eq:mod_2}. Namely, we replace $z$ with $\max(I_n)$ in the sums. When arriving at Lemma~\ref{lem:mod_ri}, one has the additional cross condition $r_1 < n_{n, 1} \cdots n_{n, 8}$, corresponding to the condition $p \mid r \implies p < p_n$. This cross condition is removed at Lemma~\ref{lem:mod_cross} similarly to the other cross conditions. We obtain that the polynomials $N_{i, j}, i \ge n+1$ and $R_i(s), i \le t$ are shorter than $P_n(s)$.
\end{remark}

\section{Tools for Dirichlet polynomials}
\label{sec:tools}

For $F(s)$ defined as in~\eqref{def:f}, we aim to show that
\begin{align*}
\int_{T_0 \le |t| \le T} |F(it)M(it)| \d t \all Rx
\end{align*}
for any $M$ as in Proposition~\ref{prop:dir_to_arit}. (Whether or not we succeed depends on the lengths of the factors of $F(s)$ in a fashion we will describe in Section~\ref{sec:ranges}.)

For convenience we shall assume that $t > 0$, the case $t < 0$ being similar. We perform a dyadic decomposition over $t$, and consider $t \in [T_1, 2T_1]$ for some $T_1$ such that $[T_1, 2T_1] \subset [T_0, T]$. Hence our task is to show the following claim.

\begin{claim}
\label{claim}
We have
\begin{align}
\label{eq:target}
\int_{T_1}^{2T_1} |F(it)M(it)|\d t \all Rx
\end{align}
for any $[T_1, 2T_1] \subset [T_0, T]$, $F(s)$ as in Information \ref{info} and $M(s)$ as in Proposition \ref{prop:dir_to_arit}.
\end{claim}

In this section we provide various tools that are employed in the next section, on the course establishing Claim \ref{claim} in certain easy cases. In the next five subsections we give five tools: Vinogradov-type pointwise bound, bound for the coefficients of the relevant Dirichlet polynomials, a fourth moment estimate for zeta sums, reduction to the case where the polynomials give power-saving bounds when $c = 0.5$, and handling the case with at least two zeta factors.

\subsection{Pointwise bound}
\label{sec:tools_pointwise}

\begin{lemma}
\label{lem:vinogradov}
Let $N \ge z_1$, $N' \le 2N$ and $f \in \{\xi_0, \xi_0 \cdot \log, \xi_0 \cdot g, \xi_0 \cdot \mu, 1_{\mathbb{P}}\}$ be given, with $f \neq \xi_0 \cdot \mu$ if $N' \ge L_{\zeta}$. We have
\begin{align*}
\left|\sum_{N < n \le N'} f(n)n^{-it}\right| \le N\exp(-(\log x)^{1/5})
\end{align*}
for all $|t| \in [T_0, T]$, assuming $x$ is large enough.
\end{lemma}

\begin{proof}
The proof is largely the same as in~\cite[Lemma 11]{HB-V}.

The result for $f = 1_{\mathbb{P}}$ is standard, following from Perron's formula and the Vinogradov-Korobov zero-free region for the zeta function. The result is immediate for $f = \xi_0 \cdot g$ by the sparsity of the support of $g$, noting that $\xi_0(p^e) \in \{0, 1\}$ for $e \ge 2$, $p$ prime.

For the case $f = \xi_0$, it suffices to obtain bounds for $N' < L_{\zeta}$ and $N \ge L_{\zeta}$ separately. In the former case, we have
\begin{align*}
\left|\sum_{N < n \le N'} f(n)n^{-it}\right| \le \sum_{\substack{m \le N' \\ P^+(m) \ge z_1}} \left|\sum_{\substack{N/m < p \le N'/m \\ p \ge P^+(m)}} p^{-it}\right|,
\end{align*}
where $P^+(m)$ denotes the largest prime factor of $m$. The inner sum is empty unless $N'/m \ge z_1$, in which case we have a bound of
$$\frac{N}{m}\exp(-2(\log x)^{1/5})$$
for the inner sum. Summing over $m$ results in a harmless log factor.

In the latter case $N \ge L_{\zeta}$ we have
\begin{align}
\label{eq:vinogradov_long}
\left|\sum_{N < n \le N'} f(n)n^{-it}\right| \le \sum_{d \le z_2} \left|\sum_{N/d < m \le N'/d} m^{-it}\right|.
\end{align}
For the inner sum we have the bound
$$\ll \sqrt{\frac{N}{d}}T_1^{1/6} + \frac{N}{d}T_1^{-1/6}$$
when $T_1 \le |t| \le 2T_1$ (see e.g.~\cite[Theorem 5.11]{titchmarsh}). By summing over $d$ we obtain that \eqref{eq:vinogradov_long} is
$$\ll z_2^{1/2}\sqrt{N}T_1^{1/6} + NT_1^{-1/6}\log z_2.$$
This is sufficient, as
$$z_2^{1/2}\sqrt{N}T_1^{1/6} \ll NT^{-1/12 + \epsilon} \ll Nx^{-\epsilon}$$
and
$$NT_1^{-1/6}\log z_2 \ll N\exp(-2(\log x)^{1/5})\log x$$
for $T_1 \in [T_0, T]$.

The case $f = \xi_0 \cdot \log$ follows similarly by partial summation. For $f = \xi_0 \cdot \mu$, by assumption we have $N' < L_{\zeta}$, and hence the same argument as above goes through. 
\end{proof}

Note that since the factors of $F(s)$ have length at least $z_1$, Lemma~\ref{lem:vinogradov} applies to any (non-constant) factor of $F(s)$.

\subsection{Coefficient bound}
\label{sec:tools_coefficient}

\begin{lemma}
\label{lem:coefficients}

\begin{enumerate}[(i)]
\item There is a constant $C = O(1)$ such that the following holds: the coefficients $c_k$ of any product of the polynomials $N_{i, j}(s)$, $R_j(s)$ and $H(s)$ in~\eqref{def:f} are bounded in absolute value by $\tau(k)^C(\log k)^C$. 
\item Let $c_k$ be the coefficients of any product of moments of those polynomials $N_{i, j}(s)$, $R_i(s)$ and $H(s)$ whose lengths do not exceed $L_{\zeta}$. Then, for any $k = x^{O(1)}$, we have $|c_k| = \exp(O((\log \log x)^{16})) = S^{o(1)}$.
\end{enumerate}
\end{lemma}

\begin{proof}
For the first claim, note that $R_i(s) = \sum_{r_i \in [R_i, R_i'] \cap \mathbb{P}} r_i^{-s}$ and $[R_i, R_i']$ are pairwise disjoint for $R_i < x^{c/2 - \epsilon}$, so any product of distinct $R_i(s), R_i < x^{c/2 - \epsilon}$ has coefficients lying in $\{0, 1\}$. As the coefficients of $N_{i, j}(s), H(s)$ and $R_i(s), R_i \ge x^{c/2 - \epsilon}$ are bounded by $\tau(k)^{O(1)}(\log k)^{O(1)}$ and there are only $O(1)$ such polynomials, the coefficients of the product have the same property.

For the second claim, under the assumption of the polynomials being shorter than $L_{\zeta}$ their coefficients $c_{k'}$ are bounded by $1_{p \mid k' \implies p \ge z_1}\log k'$. As we are only considering coefficients $c_k$ with $k = x^{O(1)}$, the number $\ell$ of polynomials in the product satisfies $\ell \ll \log x / \log z_1$. Now, given $k$ with $p \mid k \implies p \ge z_1$, the number of ways one can write $k$ as the product of $\ell$ integers is at most 
$\tau(k)^{\ell}$. Noting that $\tau(k) = \exp(O(\log x / \log z_1))$, we obtain
$$|c_k| \ll (\log k)^{\ell}\tau(k)^{\ell} = \exp(O((\log \log x)^{16})) = S^{o(1)}.$$
\end{proof}

We note that while a bound of type $\tau(k)^C(\log k)^C$ for coefficients is good enough for most purposes, this bound is slightly problematic when applying Heath-Brown's mean value theorem (Proposition~\ref{prop:HB_MVT}). The reason is that Heath-Brown's result requires a bound on the maximum value of the coefficients (in contrast to many large value theorems which consider the mean square), and $\tau(k)$ has a small mean square but a large maximum (of type $\exp(\log k / \log \log k)$). Hence, we have to be slightly careful and distinguish between cases where $F(s)$ has or does not have zeta factors.

\subsection{Moment estimates}
\label{sec:tools_moment}

This subsection is devoted to obtaining fourth moment estimates for zeta sums. We first discard the case where there is a very long zeta factor.

\begin{lemma}
\label{lem:long_zeta}
Assume that $F(s)$ factorizes as $F(s) = P(s)Z(s)$, where $Z \ge \max(L_{\zeta}, T_1z_2)$ and $Z(s)$ is a polynomial whose coefficients are given by $\xi_0$ or $\xi_0 \cdot \log$. Then Claim~\ref{claim} holds.
\end{lemma}

Here and in what follows, when we say ``$F(s)$ factorizes as $A(s)B(s)$ with $X$'', we mean that one may arrange the factors on the right hand side of \eqref{def:f} as two products $A(s)$ and $B(s)$ so that $X$ holds.

\begin{proof}
This is similar to~\cite[Lemma 12]{HB-V}. The idea is to obtain a good pointwise bound for $F(s)$ via $Z(s)$ and to bound $M(s)$ trivially as $|M(it)| \le R$.

Assume first that the coefficients of $Z(s)$ are given by $1_{(Z, Z']}(n)\xi_0(n)$. Then
$$|Z(s)| = \left|\sum_{Z < n \le Z'} \xi_0(n)n^{-s}\right| \le \sum_{\substack{d < z_2}} \left|\sum_{\substack{Z/d < m \le Z'/m}} m^{-s}\right|.$$

It is well-known that (see e.g.~\cite[Theorem 4.11]{titchmarsh})
\begin{align*}
\sum_{N < n \le M} n^{-1/2 - it} \ll M^{1/2}/|t|
\end{align*}
uniformly for $M \ge N \ge |t|/2$. Hence, by partial summation,
\begin{align*}
\sum_{Z/d < m \le Z'/d} m^{-it} \ll \frac{Z}{dT_1}
\end{align*}
for $t \in [T_1, 2T_1]$, and thus
\begin{align}
\label{eq:Z_bound}
|Z(it)| \ll \frac{Z}{T_1}\log z_2.
\end{align}

Any polynomial in the factorization of $F(s)$ is shorter than $T^{1 + \epsilon}$ (see Information \ref{info}), and hence in particular $P(s)$ is non-constant. Thus the pointwise bound (Lemma~\ref{lem:vinogradov}) applies to $P(s)$, so we obtain
\begin{align*}
|F(it)| \ll S^{\epsilon}\frac{x\exp(-(\log x)^{1/5})}{T_1}\log z_2 \ll \frac{x}{S^{\epsilon}T_1},
\end{align*}
and hence
\begin{align*}
\int_{T_1}^{2T_1} |F(it)M(it)| \d t \ll RT_1 \frac{x}{S^{\epsilon}T_1},
\end{align*}
implying~\eqref{eq:target}.

The case where the coefficients of $Z(s)$ are given by $\xi_0 \cdot \log$ follows similarly using partial summation.
\end{proof}

\begin{lemma}
\label{lem:moments}
Let $M > 0$ and $t_1, \ldots , t_M \in [T_1, 2T_1]$ be such that $|t_i - t_j| \ge 1$ for $i \neq j$. Let $Q > x^{\epsilon}$ and $Q' \le 2Q$ be given, and let
$$Q(s) = \sum_{Q < q \le Q'} \frac{1}{q^s}.$$
Assume $T_1 \ge Q/2$. Then
\begin{align}
\label{eq:moment_sum}
\sum_{m = 1}^M |Q(it_m)|^4 \ll T_1Q^2(\log x)^8
\end{align}
and
\begin{align}
\label{eq:moment_int}
\int_{T_1}^{2T_1} |Q(it)|^4 \d t \ll T_1Q^2(\log x)^8.
\end{align}
\end{lemma}

\begin{proof}
The integral bound~\eqref{eq:moment_int} follows from the bound~\eqref{eq:moment_sum} on the sum by decomposing the integral over intervals of the form $[k, k+1], k \in \mathbb{Z}$, bounding the integrands by their maximums and bounding the contribution of odd and even $k$ separately via the bound~\eqref{eq:moment_sum}. Hence it suffices to establish \eqref{eq:moment_sum}.

By Perron's formula (see e.g. \cite[Lemma 1.1]{harman}) we have
\begin{align*}
Q(it) = \frac{1}{2\pi i} \int_{5/4 - iT_1/2}^{5/4 + iT_1/2} \zeta(s+it) \frac{2^s - 1}{s} Q^s ds + O(E(Q)) + O(E(Q')),
\end{align*}
where $E(q), q \in \{Q, Q'\}$ is bounded by
$$E(q) \ll \sum_{n = 1}^{\infty} \left(\frac{q}{n}\right)^{5/4}\min\left(1, \frac{1}{T|\log(q/n)|}\right).$$
The contribution of $n \ge 2q$ and $n \le q/2$ are bounded by $O(q^{5/4}/T)$. The contribution of $q/2 < n < 2q$ is bounded as in the proof of Proposition~\ref{prop:dir_to_arit}: given $J > 10$, the contribution of $J < |n - q| \le 2J$ is $O(q/T)$, and the contribution of $|n - q| \le 10$ is $O(1)$, from which
$$E(q) \ll \frac{q^{5/4}}{T}\log x + O(1) \ll Q^{1/2},$$
say.

Moving the line of integration into the line $\text{Re}(s) = 1/2$ produces an error of
$$\ll \max_{\frac{1}{2} \le \sigma \le \frac{5}{4}} \frac{Q^{\sigma} \cdot \left|\zeta\left(\sigma + \frac{iT_1}{2} + it\right)\right|}{T_1},$$
which by the convexity bound $|\zeta(\sigma + it')| \ll |t'|^{(1 - \sigma)/2 + \epsilon}$ for $0 \le \sigma \le 1$ (see e.g. \cite[Chapter 5.1]{titchmarsh}) is bounded by 
$$T_1^{\epsilon}\left(\frac{Q^{5/4}}{T_1} + \frac{Q^{1/2}}{T_1^{3/4}}\right) = O(Q^{1/2}),$$
say.

Hence, we have
\begin{align*}
|Q(it)| \ll Q^{1/2} \int_{-T_1/2}^{T_1/2} \left|\zeta\left(\frac{1}{2}+i(\tau + t)\right)\right|\frac{1}{1 + |\tau|} d\tau + Q^{1/2}.
\end{align*}
By Hölder's inequality we thus have
\begin{align*}
\sum_{m = 1}^M |Q(it_m)|^4 \ll Q^2\left(T_1 + (\log T_1)^3 \sum_{m = 1}^M \int_{-T_1/2}^{T_1/2} \left|\zeta\left(\frac{1}{2} + i(\tau + t_m)\right)\right|^4 \frac{1}{1 + |\tau|} d\tau\right).
\end{align*}
We note that
\begin{align*}
\int_{-T_1/2}^{T_1/2} \left|\zeta\left(\frac{1}{2} + i(\tau + t_m)\right)\right|^4 \frac{1}{1 + |\tau|} d\tau &\ll \int_{t_m - T_1/2}^{t_m + T_1/2} \left|\zeta\left(\frac{1}{2} + i\tau\right)\right|^4 \frac{d\tau}{1 + |\tau - t_m|} \\
&\ll \int_{T_1/2}^{5T_1/2} \left|\zeta\left(\frac{1}{2} + i\tau\right)\right|^4 \frac{d\tau}{1 + |\tau - t_m|}
\end{align*}
and that for any $\tau \in \mathbb{R}$ we have
$$\sum_{m = 1}^M \frac{1}{1 + |\tau - t_m|} \ll \log T_1,$$
and thus we have
$$\sum_{m = 1}^M |Q(it_m)|^4 \ll Q^2T_1 + Q^2(\log T_1)^4 \int_{T_1/2}^{5T_1/2} \left|\zeta\left(\frac{1}{2} + i\tau\right)\right|^4 d\tau.$$
The fourth moment of the Riemann zeta function (see e.g. \cite[Chapter 7.6]{titchmarsh}) gives a bound of $O(T_1(\log T_1)^4)$ for the integral, giving the result.
\end{proof}

\begin{lemma}
\label{lem:moments_2}
Let $M > 0$ and $t_1, \ldots , t_M \in [T_1, 2T_1]$ be such that $|t_i - t_j| \ge 1$ for $i \neq j$. Let $Z(s) = \sum_{Z < n \le Z'} z_nn^{-s}$ be a polynomial satisfying either
\begin{enumerate}[(i)]
\item $L_{\zeta} \le Z < T_1z_2$ and the coefficients $z_n$ of $Z(s)$ are given by $\xi_0(n)$ or $\xi_0(n)\log(n)$, or
\item $L_{\zeta} \le Z$ and the coefficients $z_n$ of $Z(s)$ are given by $\xi_0(n)g(n)$.
\end{enumerate}
Then
\begin{align*}
\sum_{m = 1}^M |Z(it_m)|^4 \ll T_1Z^2z_2^4
\end{align*}
and
\begin{align*}
\int_{T_1}^{2T_1} |Z(it)|^4 \d t \ll T_1Z^2z_2^4.
\end{align*}
\end{lemma}

\begin{proof}
The second claim follows from the first as in the proof of Lemma~\ref{lem:moments}. 

For (i), note that if the coefficients of $Z$ are given by $\xi_0$, we have
\begin{align*}
|Z(s)| = \left|\sum_{Z < n \le Z'} \xi_0(n)n^{-s} \right| \le \sum_{d < z_2} \left|\sum_{Z/d < m \le Z'/d} m^{-s}\right|,
\end{align*}
and by the power-mean inequality and Lemma~\ref{lem:moments} one thus obtains
\begin{align*}
\sum_{m = 1}^M |Z(it_m)|^4 &\le \sum_{m = 1}^M \left(\sum_{d < z_2} \left|\sum_{Z/d < m \le Z'/d} m^{-s}\right|\right)^4 \\
&\ll z_2^3 \sum_{m = 1}^M \sum_{d < z_2} \left|\sum_{Z/d < m \le Z'/d} m^{-s} \right|^4 \\
&\ll z_2^3 \sum_{d < z_2} T_1\left(\frac{Z}{d}\right)^2(\log x)^8 \\
&\ll T_1Z^2z_2^4.
\end{align*}
The case $\xi_0 \cdot \log$ is follows similarly by partial summation.

For (ii), note that $g$ is supported on proper powers of primes and $|\xi_0(p^e)| \le 1$ for prime powers $p^e$, from which the result immediately follows.
\end{proof}

In conclusion, from now on we may assume that the fourth moment bounds of Lemma~\ref{lem:moments_2} apply for zeta sums $Z(s)$: if the coefficients are given by $\xi_0$ or $\xi_0 \cdot \log$ and $Z \ge T_1z_2$, we are already done proving Claim~\ref{claim} by Lemma~\ref{lem:long_zeta}, and otherwise Lemma~\ref{lem:moments_2} applies.

We note that in the case $c = 0.5$ one may replace the $z_2^{O(1)}$ losses in Lemma~\ref{lem:moments_2} by $S^{\epsilon}$ losses.

\begin{lemma}
\label{lem:moments_3}
Assume $c = 0.5$. Let $M > 0, t_1, \ldots , t_M \in [T_1, 2T_1]$ and $Z(s)$ be given, where $|t_i - t_j| \ge 1$ for $i \neq j$, $x^{1/4} \le Z \le T_1z_2$ and the coefficients of $Z(s)$ are given by $\xi_0$, $\xi_0 \cdot \log$ or $\xi_0 \cdot g$. Then
\begin{align*}
\sum_{m = 1}^M |Z(it_m)|^4 \ll T_1Z^2S^{\epsilon}
\end{align*}
and
\begin{align*}
\int_{T_1}^{2T_1} |Z(it)|^4 \d t \ll T_1Z^2S^{\epsilon}.
\end{align*}
\end{lemma}

\begin{proof}
The second claim follows from the first and the case $\xi_0 \cdot g$ is trivial. For the case $\xi_0$, see~\cite[Lemma 13]{HB-V}. We note that our values of $S$ and $\eta$ are different from that of Heath-Brown, but the exact same proof works. The case $\xi_0 \cdot \log$ follows by partial summation.
\end{proof}

\subsection{Power-saving bounds}
\label{sec:tools_power}

We then note that Claim~\ref{claim} holds if at least one of our polynomials gives only little saving over the trivial bound, assuming $c = 0.5$.

\begin{lemma}
\label{lem:power_0.5}
Assume $c = 0.5$. Write $F(s) = Q_1(s) \cdots Q_k(s)$. Let $\mathcal{T}$ denote the set of $t \in [T_1, 2T_1]$ for which there exists at least one $1 \le i \le k$ with $|Q_i(it)| \ge Q_i^{4/5}$. Then
$$\int_{\mathcal{T}} |Q_1(s) \cdots Q_k(s)M(s)| \d t \all Rx$$
for any $M(s)$ as in Proposition~\ref{prop:dir_to_arit}.
\end{lemma}

\begin{proof}
See~\cite[Section 9]{HB-V}. As with Lemma~\ref{lem:moments_3}, our values of $S$ and $\eta$ are different from those of Heath-Brown, but this changes nothing of improtance. Our value of $R$ is also different, but as the proof is based on the trivial bound $|M(s)| \le R$ and on large value theorems on the polynomials $Q_i(s)$, the proof of this lemma goes through for any value of $R$. Furthermore, while Heath-Brown's Dirichlet polynomials have coefficients $c_k = \xi_0(k)$, we also have the options $c_k = \xi_0(k)\log k$ and $c_k = \xi_0(k)g(k)$. However, the fourth moment estimate of Lemma~\ref{lem:moments_3} applies equally well in all of these cases.
\end{proof}

We note that we could establish a similar lemma when $c = 0.45$ but with $Q_i^{4/5}$ replaced by a larger threshold. However, we will take an approach which will not rely on pointwise bounds (other than Lemma~\ref{lem:vinogradov}).

\subsection{At least two zeta factors}

Recall the definition of zeta factor from Definition \ref{def:zeta_sum}. In this subsection we handle the case where $F(s)$ has at least two zeta factors.

\begin{lemma}
\label{lem:two_zetas}
Assume that $F(s)$ factorizes as $F(s) = P(s)Z_1(s)Z_2(s)$, where $Z_i$ are zeta sums with $Z_i > L_{\zeta}$. Then Claim~\ref{claim} holds.
\end{lemma}

Note that it may be the case that $P(s)$ has a zeta factor.

\begin{proof}
Assume first that $c = 0.45$. If $Z_i \ge T_1z_2$ for some $i \in \{1, 2\}$, we are done by Lemma~\ref{lem:long_zeta}, so assume not. We apply Hölder's inequality, the fourth moment estimate from Lemma~\ref{lem:moments_2} and Heath-Brown's mean value theorem from Proposition~\ref{prop:HB_MVT}. Noting that $RT > R^{7/4}T^{3/4}$ with out choice of parameters and that the coefficients $p_n$ of $P(s)$ satisfy $\max |p_n| \le \exp(\log x / (\log \log x)^{1 - \epsilon}) = x^{o(1)}$ by Lemma~\ref{lem:coefficients} and the classical bound $\tau(k) \ll \exp(O(\log x / \log \log x))$, we have
\begin{align}
\label{eq:two_zeta_0.45}
&\int_{T_1}^{2T_1} |F(it)M(it)| \d t \\
\ll &\left(\int_{0}^{T} |M(it)P(it)|^2 \d t\right)^{1/2}\left(\int_{T_1}^{2T_1} |Z_1(it)|^4 \d t\right)^{1/4}\left(\int_{T_1}^{2T_1} |Z_2(it)|^4 \d t\right)^{1/4} \nonumber \\
\ll &(P^2R^2 + PRTx^{\epsilon})^{1/2}\sqrt{T_1Z_1Z_2}z_2^2 \max |p_n| \nonumber \\
\ll &\frac{Rx\sqrt{T}\exp(\log x / (\log \log x)^{1-\epsilon})}{\sqrt{Z_1Z_2}} + R^{1/2}Tx^{1/2 + \epsilon} \nonumber,
\end{align}
which is acceptable.

If $c = 0.5$, we consider two cases depending on whether $P(s)$ has zeta factors or not. If $P(s)$ has no zeta factors (so its coefficients are bounded by $S^{o(1)}$ by Lemma \ref{lem:coefficients}), we proceed similarly as in \eqref{eq:two_zeta_0.45}, using the stronger fourth moment estimate from Lemma~\ref{lem:moments_3}. We obtain the bound
\begin{align*}
\int_{T_1}^{2T_1} |F(it)M(it)| \d t \ll (P^2R^2 + PRTx^{\epsilon})^{1/2}\sqrt{TZ_1Z_2}S^{o(1)}
\end{align*}
As we have disposed polynomials with length close to $L_{\zeta} \in \mathcal{S}$, we have $Z_i \ge L_{\zeta}x^{\eta}$. Note that $x^{\eta} \gg S$. It follows that $P \le x^{1/2 - \eta/2}$, and hence the above is bounded by
\begin{align*}
S^{o(1)}R\sqrt{TPZ_1Z_2} \cdot \sqrt{P} + x^{\epsilon}T \sqrt{RPZ_1Z_2} \ll
Rx^{1 - \eta/5} + x^{2\epsilon}\sqrt{R}T\sqrt{x},
\end{align*}
which is sufficient.

If $c = 0.5$ and $P(s)$ has a zeta factor $Z_3(s)$, we write $P(s) = Z_3(s)Q(s)$ and use Hölder's inequality to get
\begin{align*}
\int_{T_1}^{2T_1} |F(it)M(it)| \d t \ll I_1^{1/4}I_2^{1/4}I_3^{1/4} \left(\int_{T_1}^{2T_1} |Q(it)M(it)|^4 \d t\right)^{1/4},
\end{align*}
where by Lemma \ref{lem:moments_3}
\begin{align*}
I_i := \int_{T_1}^{2T_1} |Z_i(it)|^4 \d t \ll T_1Z_i^2S^{o(1)}.
\end{align*}
We further bound $|M(it)|^4 \le R^2|M(it)|^2$ and apply Proposition~\ref{prop:HB_MVT} to the polynomial $Q^2$. Note that $Q$ is shorter than $xS^{\epsilon}/Z_1Z_2Z_3 < x^{1/4 - \eta}$ and hence the coefficients of $Q^2$ are bounded by $S^{o(1)}$ by Lemma \ref{lem:coefficients}. We get
\begin{align*}
\int_{T_1}^{2T_1} |F(it)M(it)| \d t &\ll S^{o(1)}T_1^{3/4}\sqrt{RZ_1Z_2Z_3} \left(\int_{T_1}^{2T_1} |Q^2(it)M(it)|^2 \d t \right)^{1/4} \\
&\ll S^{o(1)}T^{3/4}\sqrt{RZ_1Z_2Z_3}\left(Q^4R^2 + x^{\epsilon}Q^2RT\right)^{1/4} \\
&\ll S^{o(1)}T^{3/4}R\sqrt{Qx} + x^{\epsilon}T\sqrt{x}R^{3/4}.
\end{align*}
The first term is small enough by $Q < x^{1/4 - \eta}$ and the second term clearly is small enough.
\end{proof}

\subsection{Writing \texorpdfstring{$F(s) = A(s)B(s)C(s)$}{F(s) = A(s)B(s)C(s)}}

We then take products of the factors of $F(s)$ in order to write $F(s) = A(s)B(s)C(s)$ for some $A(s), B(s), C(s)$. First note that the set of $t$ for which $\min(|A(it)|, |B(it)|, |C(it)|) \le x^{-1}$, say, has a negligible contribution to the integral in~\eqref{eq:target}. We may then partition the rest of $t \in [T_1, 2T_1]$ into $O((\log x)^{3+\epsilon}) = S^{o(1)}$ sets based on the values $\sigma_A, \sigma_B, \sigma_C$ satisfying $|A(it)| \sim A^{\sigma_A}$, $|B(it)| \sim B^{\sigma_B}$ and $|C(it)| \sim C^{\sigma_C}$. Note that we may assume $\sigma_A, \sigma_B, \sigma_C \le 1 - (\log x)^{-4/5}$ by Lemma~\ref{lem:vinogradov} (and $\sigma_A, \sigma_B, \sigma_C \le 4/5$ if $c = 0.5$ by Lemma~\ref{lem:power_0.5}). Given $\sigma_A, \sigma_B, \sigma_C$, we denote the set of such $t$ by $\mathcal{T}_{\sigma}$. 

Hence, it suffices to show that
$$\int_{\mathcal{T}_{\sigma}} |A(it)B(it)C(it)M(it)| \d t \all Rx.$$
We utilize two different strategies for bounding the integral. The first one is the simple bound
\begin{align*}
\int_{\mathcal{T}_{\sigma}} |A(it)B(it)C(it)M(it)| \d t \ll A^{\sigma_A}B^{\sigma_B}C^{\sigma_C}R|\mathcal{T}_{\sigma}|.
\end{align*}
This is sufficient if
\begin{align}
\label{eq:L1-cond}
|\mathcal{T}_{\sigma}| \all A^{1 - \sigma_A}B^{1 - \sigma_B}C^{1-\sigma_C}.
\end{align}

The second strategy is to apply the Cauchy-Schwarz inequality and Proposition \ref{prop:HB_MVT} to get (recall that $RT > R^{7/4}T^{3/4}$ by our choice of parameters)
\begin{align*}
&\int_{\mathcal{T}_{\sigma}} |A(it)B(it)C(it)M(it)| \d t \\
\ll &\left(\int_{T_0}^T |B(it)M(it)|^2 \d t\right)^{1/2}\left(\int_{\mathcal{T}_{\sigma}} |A(it)C(it)|^2\right)^{1/2} \\
\ll &\left((\max |b_n|^2)\left(B^2R^2 + BRTx^{\epsilon}\right)\right)^{1/2}\left(|\mathcal{T}_{\sigma}|A^{2\sigma_A}C^{2\sigma_C}\right)^{1/2}.
\end{align*}
If $F(s)$ has no zeta factors, so that in particular $\max |b_n|^2 = S^{o(1)}$ by Lemma \ref{lem:coefficients}, Claim \ref{claim} reduces to showing that
\begin{align}
\label{eq:cond}
|\mathcal{T}_{\sigma}| \ll \min\left(S^{-\epsilon}A^{2-2\sigma_A}C^{2-2\sigma_C}, H'RA^{1-2\sigma_A}C^{1-2\sigma_C}\right)
\end{align}
holds for any choice of $\sigma_A, \sigma_B, \sigma_C \le 1 - (\log x)^{-4/5}$ (and $\sigma_A, \sigma_B, \sigma_C \le 4/5$ if $c = 0.5$). Note that we have dropped the $x^{\epsilon}$ loss in the second term of \eqref{eq:cond}, as one may decrease the value of $\nu$ in \eqref{eq:def-R} if necessary.

The case where $F(s)$ has a zeta factor is similar. By Lemma \ref{lem:two_zetas} we may assume there is only one zeta factor, which we choose to be $A(s)$, so that the coefficients of $B$ again satisfy $\max |b_n|^2 = S^{o(1)}$.

Our proofs for \eqref{eq:cond} rely on Huxley's large value theorem.

\begin{lemma}
\label{lem:huxley}
Let $P(s)$ be a Dirichlet polynomial, let $T$ and $\sigma \le 1$ be given and let $V$ denote the measure of $t \in [T, 2T]$ for which $|P(it)| \sim P^{\sigma}$. Write $G = \sum |c_p|^2$, where the sum is over the coefficients $c_p$ of $P(s)$. Then
\begin{align*}
V \all \left(GP^{1 - 2\sigma} + T\min(GP^{-2\sigma}, G^3P^{1 - 6\sigma})\right).
\end{align*}
\end{lemma}

\begin{proof}
See~\cite[Theorem 9.7 and Corollary 9.9]{IK}.
\end{proof}

When we apply Lemma~\ref{lem:huxley}, $P(s)$ will always be a moment of products of factors of $F(s)$, where those factors are shorter than $L_{\zeta}$, with $P = x^{O(1)}$. Thus Lemma~\ref{lem:coefficients} gives $G \ll PS^{o(1)}$, and Lemma \ref{lem:huxley} implies
\begin{align}
\label{eq:huxley}
|\mathcal{T}_{\sigma}| \ll S^{o(1)}\left(P^{2 - 2\sigma} + T\min(P^{1 - 2\sigma}, P^{4 - 6\sigma})\right).
\end{align}

\section{Ranges of \texorpdfstring{$(A, B, C)$}{(A, B, C)}}
\label{sec:ranges}

In this section we determine certain cases where \eqref{eq:cond} and thus Claim \ref{claim} hold. We first consider the case $c = 0.5$.

\begin{proposition}[Ranges for $c = 0.5$]
\label{prop:range_0.5}
Let $c = 0.5$ and $R = x^{0.07 + \nu}$. Assume $F(s) = A(s)B(s)C(s)$, where $F(s), A(s), B(s)$ and $C(s)$ satisfy at least one of the following conditions:
\begin{itemize}
\item[(i)] $F(s)$ has no zeta factors, $A, B, C \ge z_1$ and $A, B \ge x^{0.43}$.
\item[(ii)] $F(s)$ has no zeta factors, $A, B, C \ge z_1$ and $B < x^{0.43}, AC^{3/5} \le x^{0.57}$ and $A \le x^{0.56}\min(x/C^8, 1)$.
\item[(iii)] $A(s)$ is a zeta sum, $A \ge L_{\zeta}$, $BS^3 \le x^{1/2}$ and $C \le x^{0.32}$.
\end{itemize}
Then Claim~\ref{claim} holds.
\end{proposition}

Note that in (iii) we allow $B(s)$ or $C(s)$ to be constant polynomials (though in the most difficult case $A = x^{1/4 + \epsilon}$ the conditions imply that $B(s)$ and $C(s)$ are non-constant).

The set of all $A \ge B \ge C$ with $ABC = x^{1 + o(1)}$ satisfying (i) or (ii) are illustrated in Figure~\ref{fig:kuva1} in black, the $x$-axis denoting the value of $\log A / \log x$ and the $y$-axis $\log B / \log x$. Note that the length of $C$ is then determined by $ABC = x^{1 + o(1)}$. We further mark the outlines of the five other symmetric cases in the figure, the axes of symmetry denoted by line segments. The triangle corresponds to the region $AB \le x^{1 + o(1)}, A, B \ge 1$.

\begin{figure}[ht]
	\centering
	\includegraphics[scale=0.75]{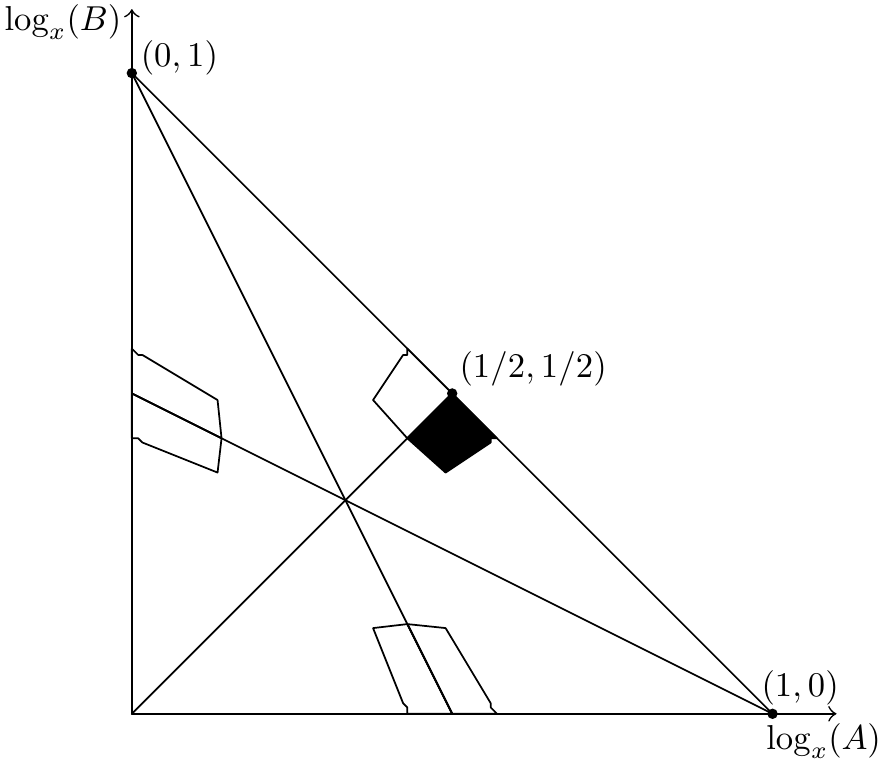}
	\caption{Set of $(A, B, C)$ covered by Proposition~\ref{prop:range_0.5}(i)-(ii).}
	\label{fig:kuva1}
\end{figure}

\begin{proof}
We aim to show~\eqref{eq:cond}. By Lemma~\ref{lem:power_0.5} we may assume $\sigma_A, \sigma_B, \sigma_C \le 4/5$. Our proof is somewhat similar to the proof of \cite[Proposition 3]{HB-V}.

(i): We may assume $A \ge B$. Noting that $AC \le xS^{\epsilon}/B \le H'R$, it suffices to show 
\begin{align}
\label{eq:proof_i}
|\mathcal{T}_{\sigma}| \all A^{2-2\sigma_A}C^{2-2\sigma_C}.
\end{align}
By~\eqref{eq:huxley} we have
$$|\mathcal{T}_{\sigma}| \ll S^{o(1)}\left(A^{2 - 2\sigma_A} + T\min(A^{1 - 2\sigma_A}, A^{4 - 6\sigma_A})\right).$$
If the former term dominates, we are done, as $C \ge z_1$ and $\sigma_C \le 1 - (\log x)^{-4/5}$. Hence we may assume 
\begin{align}
\label{eq:proof_iA}
|\mathcal{T}_{\sigma}| \ll S^{o(1)}T\min(A^{1-2\sigma_A}, A^{4-6\sigma_A}).
\end{align}
For any $w \ge 2$ such that $C^w = x^{O(1)}$ we have, by \eqref{eq:huxley},
\begin{align}
\label{eq:proof_iC}
|\mathcal{T}_{\sigma}| \ll S^{o(1)}\left(C^{2w - 2w\sigma_C} + TC^{w - 2w\sigma_C}\right).
\end{align}
(Note the implied constant does not depend on $w$.) We hence have, by taking weighted averages of \eqref{eq:proof_iA} and \eqref{eq:proof_iC},
\begin{align*}
|\mathcal{T}_{\sigma}| &\ll S^{o(1)}\left(TA^{1 - 2\sigma_A}\right)^{1 - 3/2w}\left(TA^{4 - 6\sigma_A}\right)^{1/2w}\left(C^{2w - 2w\sigma_C} + TC^{w - 2w\sigma_C}\right)^{1/w} \\
&\ll S^{o(1)}T^{1 - 1/w}A^{(2w+1)/2w - 2\sigma_A}\left(C^{2 - 2\sigma_C} + TC^{1 - 2\sigma_C}\right).
\end{align*}
Hence \eqref{eq:proof_i} follows if both
\begin{align}
\label{eq:proof_iTarget}
T^{(w-1)/w} \all A^{(2w-1)/2w} \quad \text{ and } \quad T \all A^{(2w-1)/2w}C
\end{align}
hold. The former condition may be written as $T^{(2w-2)/(2w-1)} \all A$.

Let now $w$ be the integer such that $T^{2/(2w+1)} < C \le T^{2/(2w-1)}$. Clearly $C^w = x^{O(1)}$. As $T^{2/(2w+1)}, T^{2/(2w-1)} \in \mathcal{S}$, we then have $T^{2/(2w+1)}S^{10} < C < T^{2/(2w-1)}S^{-10}$, say. Now
\begin{align*}
A \ge \sqrt{F/C} > \sqrt{x/C}S^{-\epsilon} > T^{1 - 1/(2w-1)}S^{\epsilon}
\end{align*}
and
\begin{align*}
A^{(2w-1)/2w}C &\ge \sqrt{ABC}^{(2w-1)/(2w)}C^{\frac{1}{2} + \frac{1}{4w}} \ge \sqrt{x}^{(2w-1)/2w}C^{\frac{1}{2} + \frac{1}{4w}}S^{-\epsilon} \\
&> T^{(2w-1)/2w}C^{(2w+1)/4w}S^{-3} > TS^{\epsilon},
\end{align*}
implying \eqref{eq:proof_iTarget}.

(ii): We have $AC \ge S^{-\epsilon}x/B \ge H'RS^{-2\epsilon}$, so it suffices to show 
\begin{align}
\label{eq:proof_ii}
|\mathcal{T}_{\sigma}| \all H'RA^{1-2\sigma_A}C^{1-2\sigma_C}.
\end{align}

We have, by \eqref{eq:huxley},
\begin{align*}
|\mathcal{T}_{\sigma}| \ll S^{o(1)}\left(A^{2 - 2\sigma_A} + T\min(A^{1 - 2\sigma_A}, A^{4 - 6\sigma_A})\right).
\end{align*}
Consider first the case where $A^{2 - 2\sigma_A}$ dominates. We then have, by using the assumption on $AC^{3/5}$ and the fact $\sigma_C \le 4/5$,
\begin{align*}
|\mathcal{T}_{\sigma}| \ll S^{o(1)}A^{2 - 2\sigma_A} \ll S^{o(1)}A^{1 - 2\sigma_A}\frac{x^{0.57}S^{-\epsilon}}{C^{3/5}} \ll H'R A^{1 - 2\sigma_A}C^{1 - 2\sigma_C}S^{-\epsilon/2}.
\end{align*}
This suffices.

Assume then that $|\mathcal{T}_{\sigma}| \ll S^{o(1)}T\min(A^{1 - 2\sigma_A}, A^{4 - 6\sigma_A})$. Let $w = 4$. Using also \eqref{eq:huxley} to $C^w$ we obtain 
\begin{align*}
|\mathcal{T}_{\sigma}| &\ll S^{o(1)}(TA^{1 - 2\sigma_A})^{1 - 3/2w}(TA^{4 - 6\sigma_A})^{1/2w}(C^{2w - 2w\sigma_C} + TC^{w - 2w\sigma_C})^{1/w} \\
&\ll S^{o(1)} T^{1 - 1/w}A^{1 + 1/2w - 2\sigma_A}C^{2 - 2\sigma_C} + TA^{1 + 1/2w - 2\sigma_A}C^{1 - 2\sigma_C} \\
&= S^{o(1)} T^{\frac{3}{4}}A^{\frac{9}{8} - 2\sigma_A}C^{2 - 2\sigma_C} + TA^{\frac{9}{8} - 2\sigma_A}C^{1 - 2\sigma_C}.
\end{align*}
This yields \eqref{eq:proof_ii} assuming
$$T^{\frac{3}{4}}A^{\frac{1}{8}}C \all H'R \quad \text{and} \quad TA^{\frac{1}{8}} \all H'R,$$
which hold under our assumptions.

(iii): If $B(s)C(s)$ has a zeta factor, we are done by Lemma~\ref{lem:two_zetas}. Assume this is not the case.

We have
\begin{align*}
|\mathcal{T}_{\sigma}| &\ll S^{o(1)}(TA^{2 - 4\sigma_A})^{1/2}(C^{4 - 4\sigma_C} + TC^{2 - 4\sigma_C})^{1/2} \\
&\ll S^{o(1)}\frac{T^{1/2}}{A}A^{2 - 2\sigma_A}C^{2-2\sigma_C} + S^{o(1)}\frac{T}{AC}A^{2 - 2\sigma_A}C^{2 - 2\sigma_C}.
\end{align*}
This implies \eqref{eq:cond}. Indeed, both of the terms above are dominated by $S^{-\epsilon}A^{2 - 2\sigma_A}C^{2 - 2\sigma_C}$, as $A \ge L_{\zeta}$ and $L_{\zeta} \in \mathcal{S}$ imply $A \ge L_{\zeta}x^{\eta} \ge T^{1/2}S^{2\epsilon}$ and $AC > TS^{2\epsilon}$ follows from $BS^3 \le x^{1/2}$. Both terms are also dominated by $H'RA^{1-2\sigma_A}C^{1-2\sigma_C}$, as $T^{1/2}C < H'RS^{-\epsilon}$ by assumption and $T  < H'RS^{-\epsilon}$ by our choice of parameters.
\end{proof}

We then give the following ranges in the case $c = 0.45$.

\begin{proposition}[Ranges for $c = 0.45$]
\label{prop:range_0.45}
Let $c = 0.45$ and $R = x^{0.18 + \nu}$. Assume $F(s) = A(s)B(s)C(s)$, where $F(s), A(s), B(s)$ and $C(s)$ satisfy at least one of the following conditions:
\begin{itemize}
\item[(i)] $F(s)$ has no zeta factors, $A, B, C \ge z_1$, $AC \le H'R$, and for some $w \in \mathbb{Z}_+$ with $C^w = x^{O(1)}$ one has both $T^{2w}S^{w} \le A^{2w-1}C^{2w}$ and $T^{2w-2}S^{w} \le A^{2w-1}$.
\item[(ii)] $F(s)$ has no zeta factors, $A, B, C \ge z_1$, $AC > H'R$, and for some $w \in \mathbb{Z}_+$ with $C^w = x^{O(1)}$ one has $A^{1/2w}S \le H'RT^{-1}$ and $A^{1/2w}CS \le H'RT^{-1 + 1/w}$ and $B^{2w-1} \ge T^{2w-2}/R^{2w-3}$ and $B^{6w-1}C^{4w} \ge T^{6w}/R^{6w-3}$.
\item[(iii)] $A(s)$ is a zeta sum, $A \ge L_{\zeta}$, $BS^3 \le H'$ and $CS \le H'RT^{-1/2}$.
\end{itemize}
Then Claim~\ref{claim} holds.
\end{proposition}

Items (i) and (iii) are analogous to Proposition~\ref{prop:range_0.5}, with the proof of part (ii) requiring more work. For a given $(A, B, C)$, one should take $w$ so that $C^w$ is approximately $T$.

Figure~\ref{fig:kuva2} illustrates (an approximation of) the regions encompassed by (i) and (ii) (when $\nu \approx 0$). One sees that the regions are much more complicated than in the case $c = 0.5$. Furthermore, there are now six connected components instead of three. These matters make the task of finding suitable decompositions of $F(s)$ more difficult.

\begin{figure}[ht]
	\centering
	\includegraphics{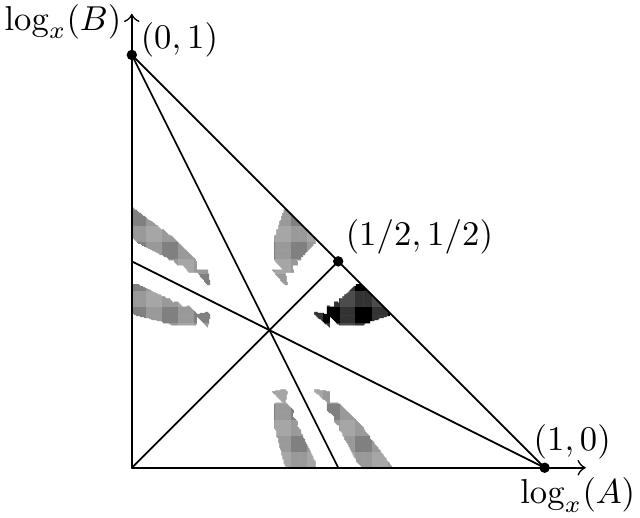}
	\caption{Set of $(A, B, C)$ covered by Proposition~\ref{prop:range_0.45}(i)--(ii).}
	\label{fig:kuva2}
\end{figure}

\begin{proof}
The proofs of (i) and (iii) follow from the proofs of the corresponding parts of Proposition~\ref{prop:range_0.5}. We are left with proving (ii).

Since $AC > H'R$, one of \eqref{eq:L1-cond} and \eqref{eq:cond} follows once we show
\begin{align}
\label{eq:proof_iiT}
|\mathcal{T}_{\sigma}| \all \max\left(A^{1-\sigma_A}B^{1-\sigma_B}C^{1-\sigma_C}, H'RA^{1-2\sigma_A}C^{1-2\sigma_C}\right),
\end{align}
which then implies Claim~\ref{claim}.

We utilize Huxley's large value theorem \eqref{eq:huxley} to the polynomials $A, B$ and $C^w$, obtaining the bounds
\begin{align*}
|\mathcal{T}_{\sigma}| &\ll S^{o(1)}\left(A^{2-2\sigma_A} + T\min(A^{1-2\sigma_A}, A^{4-6\sigma_A})\right), \\
|\mathcal{T}_{\sigma}| &\ll S^{o(1)}\left(B^{2-2\sigma_B} + TB^{4-6\sigma_B}\right) \\
|\mathcal{T}_{\sigma}| &\ll S^{o(1)}\left(C^{2w - 2w\sigma_C} + T\min(C^{w - 2w\sigma_C}, C^{4w - 6w\sigma_C})\right).
\end{align*}

We will consider separate cases according to which terms in these bounds dominate. We use the shorthand $T\min(P^k)$ for $T\min(P^{k - 2k\sigma_P}, P^{4k - 6k\sigma_P})$.

First, assume that we have $|\mathcal{T}_{\sigma}| \ll S^{o(1)}T\min(A)$. This implies
\begin{align*}
|\mathcal{T}_{\sigma}| &\ll S^{o(1)}\left(TA^{1 - 2\sigma_A}\right)^{1 - 3/2w}\left(TA^{4-6\sigma_A}\right)^{1/2w}\left(C^{2w-2w\sigma_C} + TC^{w - 2w\sigma_C}\right)^{1/w} \\
&= S^{o(1)}T^{1 - 1/w}A^{1 + 1/2w - 2\sigma_A}\left(C^{2 - 2\sigma_C} + T^{1/w}C^{1 - 2\sigma_C}\right),
\end{align*}
and \eqref{eq:proof_iiT} follows if $A^{1/2w}CS^{\epsilon} \ll H'RT^{-1 + 1/w}$ and $A^{1/2w}S^{\epsilon} \ll H'RT^{-1}$.

Hence, we may from now on assume $|\mathcal{T}_{\sigma}| \ll S^{o(1)}A^{2 - 2\sigma_A}$. If $|\mathcal{T}_{\sigma}| \ll S^{o(1)}B^{2 - 2\sigma_B}$, then we have, by weighted averages and $\sigma_C \le 1 - (\log x)^{-4/5}$,
\begin{align*}
|\mathcal{T}_{\sigma}| \ll S^{o(1)}(A^{2-2\sigma_A})^{1/2}(B^{2-2\sigma_B})^{1/2} \ll S^{-\epsilon}A^{1-\sigma_A}B^{1-\sigma_B}C^{1-\sigma_C},
\end{align*}
implying \eqref{eq:proof_iiT}. Hence from now on we may also assume $|\mathcal{T}_{\sigma}| \ll S^{o(1)}TB^{4 - 6\sigma_B}$.

There are two cases to check, one where $|\mathcal{T}_{\sigma}| \ll S^{o(1)}C^{2w-2w\sigma_C}$ and one where $|\mathcal{T}_{\sigma}| \ll S^{o(1)}T\min(C^w)$. 

Consider first the former case. To show \eqref{eq:proof_iiT} it suffices to show that the system
\begin{align*}
\begin{cases}
A^{2 - 2\sigma_A} &> A^{1 - \sigma_A}B^{1 - \sigma_B}C^{1 - \sigma_C}S^{-\epsilon} \\
A^{2 - 2\sigma_A} &> H'RA^{1 - 2\sigma_A}C^{1 - 2\sigma_C}S^{-\epsilon} \\
TB^{4 - 6\sigma_B} &> A^{1 - \sigma_A}B^{1 - \sigma_B}C^{1 - \sigma_C}S^{-\epsilon} \\
C^{2w - 2w\sigma_C} &> H'RA^{1 - 2\sigma_A}C^{1 - 2\sigma_C}S^{-\epsilon}
\end{cases}
\end{align*}
of inequalities has no solution in reals $\sigma_A, \sigma_B, \sigma_C$, when $\epsilon > 0$ is small enough. (Note that, after taking logarithms, this is a system of linear inequalities.) We first eliminate $\sigma_A$, by plugging the first inequality into the third and fourth. It follows that any solution to the above must also be a solution to the system
\begin{align*}
\begin{cases}
1 &> \frac{H'R}{AC}C^{2 - 2\sigma_C}S^{-\epsilon} \\
TB^{4 - 6\sigma_B} &> B^{2 - 2\sigma_B}C^{2 - 2\sigma_C}S^{-2\epsilon} \\
C^{2w - 2w\sigma_C} &> \frac{H'R}{AC}B^{2 - 2\sigma_B}C^{4 - 4\sigma_C}S^{-3\epsilon}.
\end{cases}
\end{align*}
We raise the first inequality to power $w - 5/2$, the second one to power $1/2$ and multiply all of the three inequalities together. We obtain
\begin{align*}
(TB^{4 - 6\sigma_B})^{1/2}C^{2w - 2w\sigma_C} > \left(\frac{H'R}{AC}\right)^{w - 3/2}B^{3 - 3\sigma_B}C^{2w(1 - \sigma_C)}S^{-(w+3/2)\epsilon}.
\end{align*}
Simplifying we obtain
$$T^{1/2}B > \left(\frac{H'R}{AC}\right)^{w - 3/2}S^{-(w+3/2)\epsilon}.$$
Note that $H'R/(AC) = S^{o(1)}H'R/(x/B) \ge SBR/T$. Hence no solutions exist if
$$T^{1/2}B \le \left(\frac{BR}{T}\right)^{w - 3/2},$$
which finally rearranges to
$$B \ge \frac{T^{(2w-2)/(2w-1)}}{R^{(2w-3)/(2w-1)}}.$$

Consider then the latter case where $|\mathcal{T}_{\sigma}| \ll S^{o(1)}T\min(C^w)$. We again reduce to a system of linear inequalities, namely
\begin{align*}
\begin{cases}
A^{2 - 2\sigma_A} &> A^{1 - \sigma_A}B^{1 - \sigma_B}C^{1 - \sigma_C}S^{-\epsilon} \\
A^{2 - 2\sigma_A} &> H'RA^{1 - 2\sigma_A}C^{1 - 2\sigma_C}S^{-\epsilon} \\
TB^{4 - 6\sigma_B} &> A^{1 - \sigma_A}B^{1 - \sigma_B}C^{1 - \sigma_C}S^{-\epsilon} \\
TC^{4w - 6w\sigma_C} &> H'RA^{1 - 2\sigma_A}C^{1 - 2\sigma_C}S^{-\epsilon}.
\end{cases}
\end{align*}
As before, we eliminate $\sigma_A$, and obtain
\begin{align*}
\begin{cases}
1 &> \frac{H'R}{AC}C^{2 - 2\sigma_C}S^{-\epsilon} \\
TB^{4 - 6\sigma_B} &> B^{2 - 2\sigma_B}C^{2 - 2\sigma_C}S^{-2\epsilon} \\
TC^{4w - 6w\sigma_C} &> \frac{H'R}{AC}B^{2 - 2\sigma_B}C^{4 - 4\sigma_C}S^{-3\epsilon}.
\end{cases}
\end{align*}
Similarly to before, by raising the first inequality to power $3w-5/2$, the second to $1/2$ and multiplying all of the resulting inequalities together on obtains, after applying $H'R/(AC) \ge SBR/T$ and rearrangement, that no solutions exist if
\begin{align*}
B^{6w-1}C^{4w} \ge \frac{T^{6w}}{R^{6w-3}}.
\end{align*}
\end{proof}

In the case $C$ is short (namely shorter than $x^{0.08-\epsilon}$), we give a simple approximation of the range of $(A, B, C)$ covered by (i) and (ii) in Proposition~\ref{prop:range_0.45}. The idea is that the regions in Figure~\ref{fig:kuva2} are well approximated as the region between two lines when one of the polynomials is short. This description is easier to work with when proving certain theoretical results in Section~\ref{sec:0.45}.

\begin{lemma}[Ranges for $c = 0.45$, simple approximation]
\label{lem:range_simple}
Let $c = 0.45$ and assume $R = x^{0.18 + \nu}$. Write $F(s) = A(s)B(s)C(S)$, and assume $F(s)$ has no zeta factors and that $(A, B, C)$ satisfies $A, B, C \ge z_1$, $C \le x^{0.08-\epsilon}$ and
$$x^{0.37+\epsilon}C^{-1/5} \le B \le x^{0.45-\epsilon}C^{-1/2}.$$
Then either (i) or (ii) of Proposition~\ref{prop:range_0.45} is satisfied.
\end{lemma}
The exponent $0.37$ comes from $T/R \approx x^{0.37}$ and $0.08$ comes from $H'R/T \approx x^{0.08}$.

\begin{proof}
We first consider the case $AC \le H'R$. We aim to find an integer $w \in \mathbb{Z}_+, w = O(1)$ so that (i) of Proposition~\ref{prop:range_0.45} is satisfied, i.e.
\begin{align}
\label{eq:proof_simple}
T^{2w}S^w \le A^{2w-1}C^{2w} \quad \text{and} \quad T^{2w-2}S^w \le A^{2w-1}.
\end{align}
Note that
\begin{align*}
A^{2w-1}C^{2w} = (AC)^{2w-1}C > (x^{0.55+\epsilon/2}C^{1/2})^{2w-1}C = x^{(0.55+\epsilon/2)(2w-1)}C^{(2w+1)/2}
\end{align*}
and
\begin{align*}
A^{2w-1} > x^{(0.55+\epsilon/2)(2w-1)}C^{-(2w-1)/2}.
\end{align*}
Recalling that $T = x^{0.55}S^2$, \eqref{eq:proof_simple} follows if $w$ satisfies
\begin{align*}
C^{(2w-1)/2} \le x^{0.55} \le C^{(2w+1)/2}.
\end{align*}
Such an integer $w$ clearly exists.

We then consider the case $AC > H'R$. We aim to find $w$ such that (ii) of Proposition~\ref{prop:range_0.45} is satisfied. We restrict our search to $w \ge 4$, in which case we have
\begin{align*}
A^{1/2w}S < S^2(x/BC)^{1/2w} \le (x^{1 - 0.37})^{1/8} < x^{0.08 - \epsilon} < H'R/T
\end{align*}
and
\begin{align*}
A^{1/2w}CS \le x^{1/2w}x^{0.08-\epsilon} < T^{1/w}\frac{H'R}{T}x^{-\epsilon},
\end{align*}
so the first two conditions of Proposition~\ref{prop:range_0.45}(ii) are satisfied. Hence, we are left with finding $w \ge 4$ such that
\begin{align*}
B^{2w-1} \ge T^{2w-2}/R^{2w-3} \quad \text{and} \quad B^{6w-1}C^{4w} \ge T^{6w}/R^{6w-3}.
\end{align*}
Writing $C = x^\alpha$ and using $B \ge x^{0.37 + \epsilon - \alpha/5}$ and $R \ge x^{0.18}$, the first inequality is satisfied if
\begin{align*}
(2w-1)(0.37 - \alpha/5) &> 0.37(2w-2) + 0.18 \Leftrightarrow \\
\frac{2\alpha}{5}w &< 0.37 - 0.18 + \frac{\alpha}{5} 
\end{align*}
and the second one if
\begin{align*}
(6w-1)(0.37 - \alpha/5) + 4w\alpha &> 0.37 \cdot 6w + 0.54 \Leftrightarrow \\
\left(4 - \frac{6}{5}\right)\alpha w &> 0.37 + 0.54 - \frac{\alpha}{5}.
\end{align*}
It follows that $w$ satisfies both of these inequalities if
$$\frac{0.325}{\alpha} - \frac{1}{14} < w < \frac{0.475}{\alpha} + \frac{1}{2}$$
To show that there is an integer solution for $w$, it suffices to check that the difference between the upper and lower bounds is greater than one. This indeed is the case, as
$$\frac{0.15}{\alpha} + \frac{8}{14} \ge \frac{0.15}{0.08} + \frac{8}{14} > 1.$$
Finally, note that there exists a solution with $w \ge 4$, since
$$\frac{0.475}{\alpha} + \frac{1}{2} \ge \frac{0.475}{0.08} + \frac{1}{2} > 4.$$
\end{proof}

\section{Applying Harman's sieve: \texorpdfstring{$c = 0.5$}{c = 0.5)}}
\label{sec:0.5}

We have above established that we may obtain an asymptotic
\begin{align}
\label{eq:0.5_asy}
\sum_{\substack{p_1, \ldots , p_n \\ p_i \in I_i \\ p_n < \ldots < p_1}} S(\mathcal{A}_{p_1 \cdots p_n}(m), z) = \frac{\delta_0}{\delta_1}\sum_{\substack{p_1, \ldots , p_n \\ p_i \in I_i \\ p_n < \ldots < p_1}} S(\mathcal{B}_{p_1 \cdots p_n}(m), z) + o\left(\frac{\delta_0 x}{\log x}\right)
\end{align}
(for all but $O(R)$ values of $m$) under certain assumptions. Namely, we assume that the polynomial
$$F(s) = P_1(s) \cdots P_n(s)Q(s)H(s),$$
may be written as $F(s) = A(s)B(s)C(s)$, where $A, B, C$ satisfy the conditions of Proposition~\ref{prop:range_0.5}. Here $P_i(s)$ corresponds to the sum over $p_i$ and hence $P_i \in I_i$. If one wishes, one may apply the Heath-Brown decomposition to $P_i(s)$. Furthermore, $Q(s)$ is a product of polynomials shorter than $z$ and $H(s)$ is a polynomial of length bounded by $T^{1+\epsilon}$. In the case $z > L_{\zeta}$ one may also decompose $Q(s)$ by the Heath-Brown decomposition.

Our ultimate aim is to show that, for some constant $d > 0$, we have
$$S(\mathcal{A}(m), 2\sqrt{x}) \ge d\frac{\delta_0}{\delta_1} S(\mathcal{B}(m), 2\sqrt{x})$$
for all but $O(R)$ integers $m$ (recall Lemma \ref{lem:to_buch}). To this end, we utilize our asymptotics of form \eqref{eq:0.5_asy} together with Harman's sieve. Recall the basic idea of Harman's sieve: First, one uses the Buchstab identity to write $S(\mathcal{A}(m), 2\sqrt{x})$ as a linear combination of sums as in the left hand side of \eqref{eq:0.5_asy}. For example, one could write, by two applications of Buchstab's identity,
\begin{align}
\label{eq:harman_example}
S(\mathcal{A}(m), 2\sqrt{x}) &= S(\mathcal{A}(m), z) - \sum_{z \le p < 2\sqrt{x}} S(\mathcal{A}_p, p) \nonumber \\
&= S(\mathcal{A}(m), z) - \sum_{z \le p < 2\sqrt{x}} S(\mathcal{A}_p, z') + \sum_{\substack{z \le p < 2\sqrt{x} \\ z' \le q < p}} S(\mathcal{A}_{pq}, q).
\end{align}
(In practice we often apply Buchstab's identity four or six times.) For some of the sums one may apply asymptotics of the form \eqref{eq:0.5_asy}. Note that the right hand side of \eqref{eq:0.5_asy} is easy to evaluate, as $\mathcal{B}(m)$ is a long interval. For some of the sums one might not have an asymptotic as in \eqref{eq:0.5_asy}. When applying Harman's sieve, one arranges things so that such ``difficult'' sums have a positive sign (such as the first and third sum in \eqref{eq:harman_example}), so that they may be discarded and what remains is a lower bound for $S(\mathcal{A}(m), 2\sqrt{x})$. Of course, one has to be careful to not discard too many of the sums, so that the lower bound is strictly positive. The contribution of discarded terms is called \emph{loss}, normalized so that aim is to keep the loss strictly below $1$.

The problem has been thus reduced to a combinatorial task of finding ranges of $p_1, \ldots , p_n$ such that \eqref{eq:0.5_asy} holds, i.e. that for any choice of the polynomials $P_i(s)$, $Q(s)$ and $H(s)$ a suitable factorization $F(s) = A(s)B(s)C(s)$ may be found, and then applying Buchstab's identity suitably to deduce a lower bound for $S(\mathcal{A}(m), 2\sqrt{x})$.

In Section~\ref{sec:0.5_theoretical} we first present theoretical results covering certain situations where~\eqref{eq:0.5_asy} may be evaluated (mainly in the cases $n \le 2$). For the cases where we have several polynomials and the execution of Harman's sieve we employ a computational procedure presented in Section~\ref{sec:0.5_algorithm}.

\subsection{Theoretical results}
\label{sec:0.5_theoretical}

We start with the main lemma of this section.

\begin{lemma}
\label{lem:MN}
Let $0 \le n \ll 1$, $P_1, \ldots,  P_n \ge z_1$ and $z \le x^{0.07 - \epsilon}$ be given. Assume that there is a subset $I \subset \{1, \ldots , n\}$ such that
$$M := \prod_{i \in I} P_i \quad \text{and} \quad N := \prod_{i \not\in I} P_i$$
satisfy $M < x^{1/2 - \epsilon}, Nz < x^{0.32 - \epsilon}$ and $MN < x^{3/4 - \epsilon}$. Then
\begin{align*}
\sum_{\substack{p_1, \ldots , p_n \\ z_1 < p_i \le P_i \\ p_n < \ldots < p_1}} S(\mathcal{A}_{p_1 \cdots p_n}(m), z) = \frac{\delta_0}{\delta_1} \sum_{\substack{p_1, \ldots, p_n \\ z_1 < p_i \le P_i \\ p_n < \ldots < p_1}} S(\mathcal{B}_{p_1 \cdots p_n}(m), z) + o\left(\frac{\delta_0 x}{\log x}\right).
\end{align*}
for all except $O(R)$ integers $m \in [x/H', 3x/H']$.
\end{lemma}

\begin{proof}
Write
$$F(s) = P_1(s) \cdots P_n(s)R_1(s) \cdots R_k(s)H(s)$$
as in \eqref{def:f}, and consider cases according to the length of $H(s)$.

Assume first that $H > L_{\zeta}$. We partition $P_i(s)$ and $R_i(s)$ as $(B', C')$ so that Proposition~\ref{prop:range_0.5}(iii) is satisfied with $(A, B, C) = (H, B', C')$. This is done via the following process: define $B_0 = M, C_0 = N$, and for each $1 \le i \le k$ define $(B_i, C_i)$ either by $(B_i, C_i) = (R_iB_{i-1}, C_{i-1})$ or by $(B_i, C_i) = (B_{i-1}, R_iC_{i-1})$. We claim that at each step we may define $(B_i, C_i)$ so that $B_i \le x^{1/2}S^{-3}$ and $C_i \le x^{0.32}$. By assumption this holds for $i = 0$. For $i \ge 1$, we cannot have both $B_{i-1}R_i > x^{1/2}S^{-3}$ and $C_{i-1}R_i > x^{0.32}$, as this would imply
$$HB_{i-1}C_{i-1}R_i \ge x^{1/4}B_{i-1}C_{i-1}R_i^2x^{-0.07+\epsilon} > x^{1/4 + 1/2 + 0.32 - 0.07 + \epsilon}S^{-3} > x^{1 + \epsilon/2},$$
a contradiction. Executing the process in this manner and choosing $(B', C') = (B_k, C_k)$ allows us to apply Proposition~\ref{prop:range_0.5}(iii).

We may then assume that $F(s)$ has no zeta factors. (Note that we did not decompose the polynomials $P_i(s)$, cf. Remark \ref{rem:optional}.) Note that necessarily $k \ge 1$, as $HMN < x^{1 - \epsilon}$.

Consider first the case $M \ge x^{0.43}$. Construct a pair $(A', B')$ by the following process: Begin with $A_0 = M, B_0 = NH$. At each step $1 \le i < k$, adjoin $R_i$ to the shorter of $A_{i-1}, B_{i-1}$. In the end we must have $B_{k-1} \ge x^{0.43}$. Indeed, this is by construction the case if $A_{k-1} \neq A_0$, and if $A_{k-1} = A_0$, we have $B_{k-1} \ge xS^{-\epsilon}/(MR_k) > x^{0.43+\epsilon/2}$. Furthermore, $A_{k-1} \ge A_0 = M \ge x^{0.43}$. Hence Proposition~\ref{prop:range_0.5}(i) is satisfied with $(A, B, C) = (A', B', R_k)$.

Consider then the case $M < x^{0.43}$. By $Nz < x^{0.32 - \epsilon}$ we have 
$$MNHR_k < x^{0.43 + 0.32 + 1/4 - \epsilon/2} = x^{1 - \epsilon/2}$$
and hence $k \ge 2$. We adjoin $R_1$ to $M$, and in general keep adjoining $R_2, \ldots , R_{k-1}$ to $M$ as long as $M < x^{0.43}$. In the end we must have $x^{0.43} \le M \le x^{1/2 - \epsilon}$ as $NHR_k < x^{0.57 - \epsilon/2}$, and we may apply the process of the previous paragraph.
\end{proof}

We obtain the following lemma as an immediate consequence.

\begin{lemma}
\label{lem:SR}
We have
$$S(\mathcal{A}(m), x^{0.07 - \epsilon}) = \frac{\delta_0}{\delta_1}S(\mathcal{B}(m), x^{0.07 - \epsilon}) + o\left(\frac{\delta_0 x}{\log x}\right)$$
for all except $O(R)$ integers $m \in [x/H', 3x/H']$.
\end{lemma}

\begin{proof}
Apply Lemma~\ref{lem:MN} with $n = 0$.
\end{proof}

We now start the task of evaluating $S(\mathcal{A}(m), 2x^{1/2})$. We apply the Buchstab identity twice to obtain
\begin{align}
\label{eq:two_buchs}
S(\mathcal{A}(m), 2x^{1/2}) &= S(\mathcal{A}(m), x^{0.07 - \epsilon}) \nonumber\\
&- \sum_{x^{0.07 - \epsilon} \le p < 2x^{1/2}} S(\mathcal{A}_p(m), x^{0.07 - \epsilon}) \\
&+ \sum_{x^{0.07 - \epsilon} \le q < p < 2x^{1/2}} S(\mathcal{A}_{pq}(m), q).\nonumber
\end{align}
By Lemma~\ref{lem:SR}, we have an asymptotic for the first term (for all but $O(R)$ exceptional values of $m$).

Next, we dispose of the awkward case $p \approx x^{1/2}$.
\begin{lemma}
\label{lem:awkward}
We have
\begin{align*}
\sum_{x^{1/2 - \epsilon} \le p < 2x^{1/2}} S(\mathcal{A}_p(m), x^{0.07 - \epsilon}) = \frac{\delta_0}{\delta_1} \sum_{x^{1/2 - \epsilon} \le p < 2x^{1/2}} S(\mathcal{B}_p(m), x^{0.07 - \epsilon}) + o\left(\frac{\delta_0 x}{\log x}\right)
\end{align*}
for all except $O(R)$ integers $m \in [x/H', 3x/H']$.
\end{lemma}

\begin{proof}
We apply the Heath-Brown decomposition to the sum over $p$, and thus write
$$F(s) = \left(\prod_{i, j} N_{i, j}(s)\right)R_1(s) \cdots R_k(s)H(s),$$
where $N_{i, j}(s), R_i(s)$ and $H(s)$ are as in Information \ref{info}.

If $H < L_{\zeta}$ and no $N_{i, j}(s)$ is longer than $L_{\zeta}$, we apply Proposition~\ref{prop:range_0.5}(i) with 
$$(A, B, C) = (P, R_1 \cdots R_{k-1}H, R_k).$$

If $H \ge L_{\zeta}$, by Lemma~\ref{lem:two_zetas} we are done if some $N_{i, j}(s)$ is longer than $L_{\zeta}$. Assume not. Combine any of $N_{i, j}(s)$ as long as their product is shorter than $L_{\zeta}$. In the end one has two or three polynomials, all shorter than $L_{\zeta}$ (otherwise their product would exceed $L_{\zeta}^2x^{\eta}$). In any case one can partition the polynomials into two sets, so that the product $B(s)$ of the first set satisfies $B < x^{1/2 - \epsilon}$ and the product $C(s)$ of the second set satisfies $C < x^{0.32 - \epsilon}$. The result now follows similarly as in the proof of Lemma~\ref{lem:MN} by adjoining polynomials $R_i(s)$ suitably one-by-one to $B(s)$ or $C(s)$.

If $H < L_{\zeta}$ and some $N_{i, j}(s)$ is longer than $L_{\zeta}$, we swap $H(s)$ and $N_{i, j}(s)$ and apply the argument of the previous paragraph.
\end{proof}

Lemmas~\ref{lem:MN} and~\ref{lem:awkward} together give an asymptotic for the second term on the right hand side of~\eqref{eq:two_buchs}. We further note that in the third term
$$\sum_{\substack{x^{0.07 - \epsilon} \le q < p < 2x^{1/2}}} S(\mathcal{A}_{pq}, q)$$
one may drop the region $p > x^{1/2 - \epsilon}$. The loss caused by this operation is $O(\epsilon)$ (see \eqref{eq:full_loss} below), which is negligible. Hence, our task is to show that the loss arising from the sum
$$\sum_{x^{0.07 - \epsilon} \le q < p < x^{1/2 - \epsilon}} S(\mathcal{A}_{pq}(m), q)$$
is bounded from above by $1 - c'$ for some constant $c' > 0$ independent of $\epsilon$.

%In what follows, we will no longer be as careful when disposing such ``small' ranges. It is not essential whether certain sums are over, say, $p > x^{1/4}$ or $p > x^{1/4}S^{O(1)}$ or $p > x^{1/4 + \epsilon}$, as the resulting differences in contribution to various sums are negligible. For the sake of exposition we will suppress terms such as $S^{O(1)}$ and $x^{\epsilon}$. (Note that  we have discarded the case where some of our polynomials has length close to a power of $x$ via Lemma~\ref{lem:close}, and hence the last condition in Proposition~\ref{prop:range_0.5}(i) may be forgotten.) For example, we now consider only consider the region $p, q > x^{0.07}$, understanding that the contribution of $q \in [x^{0.07 - \epsilon}, x^{0.07}]$ is negligible, and when assuming that $q > x^{1/4}$ we may as well assume that $q > x^{1/4 + \epsilon}$.

Next, we show that there are certain regions with $q > x^{1/4}$ where $S(\mathcal{A}_{pq}, q)$ may be evaluated. Namely, let
\begin{align}
\label{eq:B_1}
B_1 &= \{(x, y) : x > y > \frac{1}{4}+\epsilon,  x+y < 0.57-\epsilon, x + 0.7y > 0.43+\epsilon, x - 3y > -0.56+\epsilon\}, \\
B_2 &= \{(x, y) : (1-x-y, y) \in B_1\} \nonumber
\end{align}
and denote $p = x^{\alpha_p}, q = x^{\alpha_q}$. See Figure \ref{fig:kuva3} for an illustration of the regions where $(\alpha_p, \alpha_q) \in B_j$.

\begin{figure}[ht]
	\centering
	\includegraphics{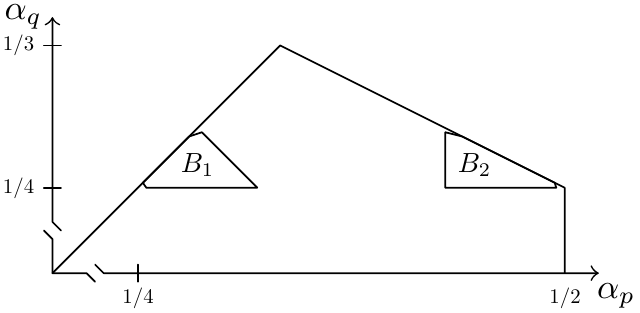}
	\caption{Sets $B_1$ and $B_2$.}
	\label{fig:kuva3}
\end{figure}

\begin{lemma}
\label{lem:sum_B}
For $i \in \{1, 2\}$, we have
\begin{align*}
\sum_{(\alpha_p, \alpha_q) \in B_i} S(\mathcal{A}_{pq}(m), q) = \frac{\delta_0}{\delta_1} \sum_{(\alpha_p, \alpha_q) \in B_i} S(\mathcal{B}_{pq}(m), q) + o\left(\frac{\delta_0 x}{\log x}\right)
\end{align*}
for all except $O(R)$ integers $m \in [x/H', 3x/H']$.
\end{lemma}

\begin{proof}
The sums count products of three primes, essentially of size $x^{\alpha_p}, x^{\alpha_q}$ and $x^{1 - \alpha_p - \alpha_q}$. Hence, the cases $i = 1$ and $i = 2$ are analogous, and it suffices to consider $i = 1$.

Hence, consider
$$F(s) = P(s)Q(s)R_1(s) \cdots R_k(s)H(s),$$
where $R_i \le Q$ and the pair $(\log_x(P), \log_x(Q))$ lies in $B_1$. We note that in this proof we treat $R_1 \cdots R_kH$ as a single polynomial.

We apply the Heath-Brown decomposition to the polynomial $Q(s)$. Taking products of any two polynomials shorter than $L_{\zeta}$ and noting that $Q(s)$ is shorter than $x^{1/3+\epsilon}$, we may thus consider the case where $Q$ decomposes as the product of at most two polynomials, with a polynomial longer than $L_{\zeta}$ a zeta sum. Assume that we get two polynomials $Q_1, Q_2$ from the decomposition with $Q_1 \ge Q_2$, where possibly $Q_2 = 1$.

If $F(s)$ has at least two zeta factors, we are done by Lemma~\ref{lem:two_zetas}. Assume not.

If $Q_1 \ge L_{\zeta}$, we apply Proposition~\ref{prop:range_0.5}(iii) with
$$A = Q_1, B = R_1 \cdots R_kHQ_2, C = P.$$
Note that $AC \ge x^{1/4}P \ge x^{1/2+\epsilon}$ and hence $B < x^{1/2-\epsilon/2}$, and that $C = P < x^{0.32-\epsilon}$ by the condition $x+y < 0.57 - \epsilon$ in~\eqref{eq:B_1}.

Assume then that $Q_1 < L_{\zeta}$ (and hence $Q_2 > 1$). If $H(s)$ is a zeta sum, we swap $Q_1$ and $H$ and apply the argument above. Hence assume that there are no zeta factors.

We take
$$(A, B, C) = (\max(PQ_1, R_1 \cdots R_kH), \min(PQ_1, R_1 \cdots R_kH), Q_2),$$
and show that (i) or (ii) of Proposition~\ref{prop:range_0.5} holds. 

We first note that by assumption $PQ < x^{0.57-\epsilon}$, and thus $R_1 \cdots R_kH > x^{0.43+\epsilon/2}$. Hence, we are done by Proposition \ref{prop:range_0.5}(i) if $PQ_1 > x^{0.43}$, and hence we may assume the contrary. In particular, $B = PQ_1$ and $A = R_1 \cdots R_kH$.

We apply Proposition \ref{prop:range_0.5}(ii). To do so, we have to check that the conditions
\begin{align*}
\frac{x}{PQ}Q_2^{3/5} < x^{0.57-\epsilon/2} \quad \text{ and } \quad \frac{x}{PQ} < x^{0.56-\epsilon/2}\min(1, x/Q_2^8)
\end{align*}
hold. The first one follows by
$$\frac{PQ}{Q_2^{3/5}} \ge PQ^{7/10} > x^{0.43+\epsilon}$$
and the second one follows from
$$PQ > x^{1/2} > x^{0.44+\epsilon/2} \quad \text{and} \quad \frac{PQ}{Q_2^8} \ge PQ^{-3} \ge x^{-0.56 + \epsilon}.$$
\end{proof}

\subsection{Computational procedure}
\label{sec:0.5_algorithm}

As the computations get very laborious to do by hand when the Buchstab identity is applied twice or even four times more, we will from now on rely on computer calculation for bounding the loss. Below we describe the algorithm used for the computation.

Consider the task of bounding the loss arising from
$$\sum_{\substack{(p_1, \ldots,  p_n) \in I \\ p_n < \ldots < p_1}} S(\mathcal{A}_{p_1 \cdots p_n}, p_n)$$
for some product of intervals $I \subset [0, 1/2 - \epsilon]^n$. We first cover the set $I$ with a union of boxes
$$\mathcal{B} := [x^{\alpha_1}, x^{\beta_1}) \times \cdots \times [x^{\alpha_n}, x^{\beta_n})$$
for $0 \le \alpha_i < \beta_i \le 1/2 - \epsilon$, and consider the sum over a single box $\mathcal{B}$. We may assume $\beta_i \ge \alpha_{i+1}$, as otherwise the condition $p_{i+1} < p_i$ is not satisfied in $\mathcal{B}$, and that $\alpha_1 + \ldots + \alpha_{n-1} + 2\alpha_n \le 1 + \epsilon$, as otherwise the sum is empty. In practice we will choose the decompositions so that $\beta_i - \alpha_i$ are small (e.g. less than $1/100$) -- we specify the details in the end.

Next, we determine whether the Buchstab identity can be applied twice more. More precisely, the question is whether there exist parameters $z$ and $z_{p_{n+1}}$ so that, writing
\begin{align*}
\sum_{\substack{(p_1, \ldots,  p_n) \in \mathcal{B} \\ p_n < \ldots < p_1}} S(\mathcal{A}_{p_1 \cdots p_n}, p_n) &= \sum_{\substack{(p_1, \ldots,  p_n) \in \mathcal{B} \\ p_n < \ldots < p_1}} S(\mathcal{A}_{p_1 \cdots p_n}, z) \\
&- \sum_{\substack{(p_1, \ldots,  p_n) \in \mathcal{B}, p_{n+1} \\ p_{n+1} < p_n < \ldots < p_1}} S(\mathcal{A}_{p_1 \cdots p_np_{n+1}}, z_{p_{n+1}}) \\
&+ \sum_{\substack{(p_1, \ldots,  p_n) \in \mathcal{B}, p_{n+1}, p_{n+2} \\ p_{n+2} < p_{n+1} < p_n < \ldots < p_1}} S(\mathcal{A}_{p_1 \cdots p_n}, p_{n+2}),
\end{align*}
the first and second sums on the right hand side may be evaluated asymptotically. By Lemma~\ref{lem:MN}, one has an asymptotic for the first sum if $\beta_1 + \ldots + \beta_n < 3/4 - \epsilon$ and $\{1, \ldots , n\}$ may be partitioned into two sets $M, N$ such that
$$\sum_{i \in M} \beta_i < 1/2 - \epsilon \quad \text{and} \quad \sum_{i \in N} \beta_i < 0.32 - \epsilon.$$
Moreover, the better bound one has for the sum over $N$, the larger one may take $z$. Similarly, an asymptotic for the second sum is found if the sum $\beta_1 + \ldots + \beta_n + \beta_n$ is less than $3/4 - \epsilon$ and may be partitioned into two subsums smaller than $1/2 - \epsilon$ and $0.32 - \epsilon$, and better bounds for the latter subsum allow one to choose larger values of $z_{p_{n+1}}$.

We apply the Buchstab identity in this way until we can no more or until $n = 6$, after which the benefits from further applications of the identity would be negligible.

The question, then, is whether we have an asymptotic formula for
$$\sum_{\substack{(p_1, \ldots , p_n) \in \mathcal{B} \\ p_n < \ldots < p_1}} S(\mathcal{A}_{p_1 \cdots p_n}, p_n),$$
which corresponds to asking whether the polynomial
$$F(s) = P_1(s) \cdots P_n(s)R_1(s) \cdots R_k(s)H(s)$$
necessarily satisfies Claim~\ref{claim}. We do not utilize the Heath-Brown decomposition to $P_i(s)$ or $R_i(s)$ here. We have $R_i \le P_n$ and $P_i \in [x^{\alpha_i}, x^{\beta_i}]$ for all $i$.

If $H > L_{\zeta}$, we consider whether
$$P_1(s) \cdots P_n(s)R_1(s) \cdots R_k(s)$$
may be written as $B(s)C(s)$ with $B \le x^{1/2 - \epsilon}, C \le x^{0.32 - \epsilon}$, so that Proposition~\ref{prop:range_0.5}(iii) is satisfied. We note that
$$P_1 \le x^{\beta_1}, \ldots, P_n \le x^{\beta_n} \text{ and } R_1 \cdots R_k \le x^{1 - \alpha_1 - \ldots - \alpha_n - 1/4 + \epsilon},$$
and hence a suitable decomposition $(B, C)$ may be found (if one exists) by considering partitions of the multiset
$$\{\beta_1, \ldots , \beta_n, 3/4 - \alpha_1 - \ldots - \alpha_n\}$$
into two multisets and checking whether in any partition the two parts have sums less than $1/2 - \epsilon$ and $0.32 - \epsilon$.

If $H < L_{\zeta}$, there are no zeta factors, and we consider whether $F(s)$ may be written as $(A, B, C)$ so that Proposition~\ref{prop:range_0.5}(i) or (ii) is satisfied. We utilize two strategies. 

The first strategy is a crude one, where we combine all of $R_1(s), \ldots , R_k(s)$ and $H(s)$ into one polynomial $Q(s)$, and go through all ways of writing $P_1(s) \cdots P_n(s)Q(s)$ as $A(s)B(s)C(s)$. The number of such ways is bounded by $3^{n+1}$.

The second strategy is slightly more careful, though it requires $\beta_1 + \ldots + \beta_n < 3/4$ so that $k \ge 1$. We perform a casework on the length of $R_1(s)$, combine all of $R_2(s) \cdots R_k(s)H(s)$ into one polynomial $Q(s)$, and check whether a suitable decomposition $A(s)B(s)C(s)$ for $P_1(s) \cdots P_n(s)R_1(s)Q(s)$ may be found for every possible length of $R_1(s)$. The benefit of this strategy is that we have more polynomials and in particular the short polynomial $R_1$ at our disposal.

In any case, we end up considering several decompositions $F(s) = A(s)B(s)C(s)$. Lower and upper bounds on the factors of $F(s)$ yield bounds on the lengths of $A(s), B(s), C(s)$ via the following (trivial) lemma. In the lemma and afterwards we denote lower and upper bounds on the length of $P(s)$ by $P_l$ and $P_u$ so that $P \in [x^{P_l}, x^{P_u}]$.

\begin{lemma}
\label{lem:operations}

\begin{itemize}
\item[(i)] Let $A(s)$ and $B(s)$ be Dirichlet polynomials. Then
$$AB \in [x^{A_l + B_l}, x^{A_u + B_u}].$$
\item[(ii)] Let $A(s), B(s)$ and $C(s)$ be Dirichlet polynomials with $C(s) = A(s)B(s)$. Then
$$B \in [x^{C_l - A_u}, x^{C_u - A_l}].$$
\end{itemize}
\end{lemma}

\begin{proof}
(i): Since $A \ge x^{A_l}$ and $B \ge x^{B_l}$, we have $AB \ge x^{A_l + B_l}$. The upper bound is proven similarly.

(ii): The upper bound follows from $x^{A_l}B \le AB = C \le x^{C_u}$, the lower bound being similar.
\end{proof}

Now, given lower and upper bounds on the lengths $A, B$ and $C$, Proposition~\ref{prop:range_0.5} applies assuming that 
\begin{align}
\label{eq:loose_0.5}
&(\text{if } B_u > 0.43 - \epsilon, \text{ then } A_l > 0.43 + \epsilon) \text{ and} \nonumber \\
&(\text{if } B_l < 0.43 + \epsilon, \text{ then } A_u + 0.6C_u < 0.57 - \epsilon \text{ and } A_u < 0.56 + \min(0, 1 - 8C_u) - \epsilon). 
\end{align}

In the case we do not have an asymptotic formula, the loss arising from discarding the sum is equal to (see~\cite{BH})
\begin{align}
\label{eq:full_loss}
\int_{x_1 = \alpha_1}^{\beta_1} \int_{x_2 = \alpha_2}^{\beta_2} \cdots \int_{x_n = \alpha_n}^{\beta_n} \omega\left(\frac{1 - x_1 - \ldots - x_n}{x_n}\right) 1_{x_n < \ldots < x_1} \frac{\d x_1 \cdots \d x_n}{x_1 \cdots x_{n-1}x_n^2},
\end{align}
where $\omega$ is the Buchstab function. Discarding the indicator function (which often has no effect, as the differences $\beta_i - \alpha_i$ are small and we have assumed $\beta_i \ge \alpha_{i+1}$)  and bounding the integrand by its supremum, we obtain an upper bound of
\begin{align}
\label{eq:loss_bound}
\sup_{u \in J} \omega(u)\frac{1}{\alpha_n}\prod_{i = 1}^n \frac{\beta_i - \alpha_i}{\alpha_i},
\end{align}
where the supremum over $u$ is over the interval
$$J := \left[\frac{1 - \beta_1 - \ldots - \beta_n}{\beta_n}, \frac{1 - \alpha_1 - \ldots - \alpha_n}{\alpha_n}\right].$$
We apply the bounds
\begin{align*}
\omega(u) \le \begin{cases}
0, \qquad \qquad \qquad \ u < 1 \\
\frac{1}{u}, \qquad \qquad \ 1 \le u \le 2 \\
\frac{1 + \log(u-1)}{u}, \quad 2 \le u \le 3 \\
0.565, \qquad \quad  3 < u
\end{cases}
\end{align*}
to bound such supremums. (The first three items here are equalities.)

In practice, beginning from 
$$\sum_{x^{0.07} < q < p \le x^{1/2 - \epsilon}} S(\mathcal{A}_{pq}, q),$$
we will decompose the sums over $p$ and $q$ into intervals of the form $[x^{\alpha_i}, x^{\beta_i}]$ with $\beta_i - \alpha_i = 1/3000$. In further applications of the Buchstab identity we will take $\beta_i - \alpha_i = 1/400$.

There are some additional implementation issues not discussed in detail here: In practice it suffices to consider only decompositions $F(s) = A(s)B(s)C(s)$ where $C(s)$ is equal to $P_n(s)$ or $R_1(s)$. Given a box $\mathcal{B} \subset \mathbb{R}^2$, we check whether $\mathcal{B}$ lies in the region $B_1 \cup B_2$ of Lemma~\ref{lem:sum_B} to handle the case $n = 2$. We take $\epsilon = 10^{-9}$ in various lemmas, and in general impose margins of $10^{-9}$ at various situations to avoid mistakes from rounding errors. The interested reader is invited to read the implementation.

The computation takes approximately fifteen minutes on a usual consumer laptop, giving an upper bound of $0.991 < 1$. As one would expect, most of the loss arises when $p$ is large (e.g. the case $p < x^{1/4}$ gives a loss of less than $0.03$). The program prints more detailed information during runtime.

\begin{remark}
\label{rem:easy}
There is an easier way (both computationally and conceptually) to obtain non-rigorous estimates for the loss. Instead of considering intervals of possible polynomial lengths, one takes a sample with the polynomial lengths being, for example, of the form $x^{k/n}, k \in \mathbb{Z}_+$ for some fixed $n$ (e.g. $n = 200$) to approximate the loss. Such a computation suggests that the value of $R$ could be somewhat improved from $x^{0.07}$, but not by much -- it seems to us that reaching the value $R = x^{0.06}$ would require new ideas.
\end{remark}

\section{Applying Harman's sieve: \texorpdfstring{$c = 0.45$}{c = 0.45}}
\label{sec:0.45}

We assume the reader has read Section~\ref{sec:0.5} before reading this section.

The case $c = 0.45$ is largely similar to the case $c = 0.5$. The central differences are that the results of Proposition~\ref{prop:range_0.45} are more complicated than those of Proposition~\ref{prop:range_0.5}, the resulting ranges of $(A, B, C)$ are more disconnected (see Figures~\ref{fig:kuva1} and~\ref{fig:kuva2}) and that we employ the Heath-Brown decomposition. Nevertheless, the modification is relatively straightforward.

We first give necessary theoretical results in Section~\ref{sec:theoretical}, after which we explain the computational procedure used in this case.

\subsection{Theoretical tools}
\label{sec:theoretical}

In this section we present tools which our computational procedure is based on. At many places we need results that rely only on lower and upper bounds on the length of relevant polynomials $P$. We denote these bounds by $P \in [x^{P_l}, x^{P_u}]$. These bounds behave well under multiplication and division, see Lemma~\ref{lem:operations}.

We extend Proposition~\ref{prop:range_0.45}(i) and (ii) to the case where we only have loose bounds on the lengths of polynomials (cf.~\eqref{eq:loose_0.5}). In what follows we write $h = 0.45, t = 0.55$ and $r = 0.18$. 

\begin{lemma}
\label{lem:uncertain_typeII}
Let $F(s) = A(s)B(s)C(s)$ and let $\epsilon > 0$ be fixed. Assume $F(s)$ has no zeta factors and $A, B, C \ge z_1$. Denote by $\mathcal{C}_1$ the condition
\begin{align*}
& B_u \le t - r - \epsilon \quad \text{or for some } 1 \le w \le 20 \text{ we have} \\
& [2wt + \epsilon \le (2w-1)A_l + 2wC_l \text{ and } (2w-2)t + \epsilon \le (2w-1)A_l]
\end{align*}
and by $\mathcal{C}_2$ the condition
\begin{align*}
& B_l \ge t-r+\epsilon \quad \text{or for some } 1 \le w \le 20 \text{ we have} \\
& [A_u/2w \le h+r-t-\epsilon \text{ and } A_u/2w + C_u \le h+r-\left(1 - \frac{1}{w}\right)t-\epsilon \text{ and } \\
& (2w-1)B_l \ge (2w-2)t - (2w-3)r + \epsilon \text{ and } (6w-1)B_l + 4wC_l \ge 6wt - (6w-3)r + \epsilon]. 
\end{align*}
Assuming that both $\mathcal{C}_1$ and $\mathcal{C}_2$ hold, then Claim~\ref{claim} holds.
\end{lemma}

Note that even though we assume upper and lower bounds for $A, B, C$, we still have $F = xS^{o(1)}$ independent of those bounds and that the polynomials $A(s), B(s), C(s)$ are longer than $z_1$ (assuming they are non-constant).

\begin{proof}
This is a direct consequence of Proposition~\ref{prop:range_0.45}(i) and (ii). Note that the $\epsilon$-terms in Lemma~\ref{lem:uncertain_typeII} handle the $S^{O(1)}$-terms in Proposition~\ref{prop:range_0.45}. We have restricted to considering only $w \le 20$ for practical reasons (the exact threshold $20$ being somewhat arbitrary).
\end{proof}

We then give the corresponding result for Proposition~\ref{prop:range_0.45}(iii).

\begin{lemma}
\label{lem:uncertain_typeI}
Let $F(s) = A(s)B(s)C(s)$ and let $\epsilon > 0$ be fixed. Assume that $A_l > t/2 + \epsilon$ and that $A(s)$ is a zeta sum. If
$$B_u \le h-\epsilon \quad \text{and} \quad C_u \le h + r - t/2-\epsilon,$$
then Claim~\ref{claim} holds.
\end{lemma}

\begin{proof} Note that $A > L_{\zeta}$. The result follows from Proposition~\ref{prop:range_0.45}(iii).
\end{proof}

We next note that if $F(s)$ has a zeta factor, then polynomials shorter than $H'^2R/x$ do not cause us problems. With our choice of parameters we have $H'^2R/x = x^{0.08 + o(1)}$.

\begin{lemma}
\label{lem:ignore_short}
Let $A, B', C', P$ be Dirichlet polynomials with $AB'C'P \le xS^{o(1)}$ and $P \le H'^2Rx^{-1}$. Assume that $A > L_{\zeta}$ is a zeta sum and that
$$B'S^3 < H' \quad \text{and} \quad C'S < \frac{H'R}{\sqrt{T}}.$$
Defining $(B_1, C_1) = (B'P, C')$ and $(B_2, C_2) = (B', C'P)$, there is some $i \in \{1, 2\}$ such that
$$B_iS^3 < H' \quad \text{and} \quad C_iS < \frac{H'R}{\sqrt{T}}.$$
\end{lemma}

\begin{proof}
If not, then one has both
$$B'P \ge \frac{H'}{S^3} \quad \text{and} \quad C'P \ge \frac{H'R}{S\sqrt{T}},$$
so that
$$B'C'P \ge \frac{H'^2R}{\sqrt{T}S^4P} \ge \frac{x}{\sqrt{T}S^4},$$
which contradicts $A > L_{\zeta}$ and $AB'C'P \le xS^{o(1)}$.
\end{proof}

Hence, recalling Proposition~\ref{prop:range_0.45}(iii), in the presence of a zeta sum we may ignore polynomials $R_1(s), \ldots , R_k(s)$ assuming $z < x^{0.08 - \epsilon}$.

Our next result concerns the case where the polynomial
$$F(s) = P_1(s) \cdots P_n(s)R_1(s) \cdots R_k(s)H(s)$$
has many short factors $R_i < z$ with $z$ small. Heuristically, one should be able to find a suitable decomposition $F(s) = A(s)B(s)C(s)$ in this case, since having many short polynomials gives one plenty of options for adjusting the lengths $A, B, C$. The next result formalizes this intuition. The result is stronger the longer $R_1(s) \cdots R_k(s)$ is.

\begin{lemma}
\label{lem:MN-2}
Let $Q_1, \ldots , Q_n \ge z_1$ and $H \ge z_1$ be given. Let $\epsilon > 0$ be small and fixed and let $z \le x^{0.061}$. Assume that $Q_1 \cdots Q_nH < x/S$ and that at least one of the following conditions hold:
\begin{itemize}
\item[(i)] There is some subproduct $M(s)$ of $Q_1(s) \cdots Q_n(s)H(s)$ such that
$$M \in [x^{0.37 + \epsilon}, x^{0.45-\epsilon}z^{-1/2}] \cup [x^{0.55+\epsilon}, x^{0.63-\epsilon}z^{-4/5}].$$
\item[(ii)] We have $Q_1 \cdots Q_nH < x^{0.9-3\epsilon}/z$.
\item[(iii)] There is some subproduct $M(s)$ of $Q_1(s) \cdots Q_n(s)H(s)$ and some integer $K \ge 2$ such that
$$Z := \frac{x}{HQ_1 \cdots Q_n}$$
satisfies $Z > z^{K-1}S$ and
$$M \in [x^{0.37+\epsilon}Z^{-(5K-4)/5K}, x^{0.45-\epsilon}z^{-1/2}] \cup [x^{0.55+\epsilon}Z^{-(2K-1)/2K}, x^{0.63-\epsilon}z^{-4/5}].$$
\end{itemize}
Then, assuming
$$F(s) = Q_1(s) \cdots Q_n(s)H(s)R_1(s) \cdots R_k(s)$$
has no zeta factors, $F(s)$ satisfies Claim~\ref{claim} (for any $k \ge 1$, $R_i < z$).
\end{lemma}

As Lemma~\ref{lem:ignore_short} already essentially handles the case where $z$ is small and one has a zeta factor, restricting to the case where $F(s)$ has no zeta factors is not an issue. Note that the condition $Q_1 \cdots Q_nH < x/S$ implies $k \ge 1$. In (iii) the bound $Z > z^{K-1}S$ implies that $k \ge K$, so $K$ is a lower bound on the number of factors $R_i(s)$. In practice the polynomials $Q_i(s)$ correspond to the polynomials $P_i(s)$ or factors arising from applying the Heath-Brown decomposition to them.

\begin{proof}
We aim to write $F(s) = A(s)B(s)C(s)$ so that the conditions of Lemma~\ref{lem:range_simple} are satisfied. Note first that since $k \ge 1$, if a subproduct $M$ as in (i) may be found, we may simply take
$$(A, B, C) \in \left\lbrace\left(M, \frac{F}{MR_1}, R_1\right), \left(\frac{F}{MR_1}, M, R_1\right)\right\rbrace,$$
depending on which of the two intervals in (i) $M$ lies in.

Assume then that we are in the situation of (ii) or (iii). We have
$$Q_1 \cdots Q_n H = \frac{x}{Z},$$
where $Z$ is defined as in (iii), and hence
$$R_1 \cdots R_k \gg \frac{Z}{S^{1/2}}.$$
It follows that $k \ge K$. Note that we may assume $R_1 \ge \ldots \ge R_k$.

We choose $C = R_k$. It suffices to find a subproduct $P$ of $F/R_k$ such that $P$ has length
$$P \in [x^{0.37+\epsilon}R_k^{-1/5}, x^{0.45-\epsilon}R_k^{-1/2}] \cup [x^{0.55+\epsilon}R_k^{-1/2}, x^{0.63-\epsilon}R_k^{-4/5}] =: I_1 \cup I_2.$$
Indeed, if $P$ lies in the former interval, we take $B(s) = P(s)$ and $A(s) = F(s)/(P(s)C(s))$ in Lemma~\ref{lem:range_simple}. In the latter case one takes $B(s) = F(s)/(P(s)C(s))$ and $A(s) = P(s)$.

Let $L(s) = R_1(s) \cdots R_{k-1}(s)$. Since $C \le z \le x^{0.061}$, the lengths of the intervals $I_i$ are $x^{0.08-2\epsilon}C^{-3/10} > x^{0.0615} \ge z.$ Hence, it suffices to find a subset of $Q_1(s), \ldots , Q_n(s), H(s)$ whose product $Q(s)$ satisfies
\begin{align*}
Q \in [x^{0.37+\epsilon}R_k^{-1/5}L^{-1}, x^{0.45-\epsilon}R_k^{-1/2}] \cup [x^{0.55+\epsilon}R_k^{-1/2}L^{-1}, x^{0.63-\epsilon}R_k^{-4/5}] =: J_1 \cup J_2,
\end{align*}
as then one can construct a desired subproduct $P(s)$ of $F(s)/R_k(s)$ by adjoining factors of $L(s)$ to $Q(s)$ one by one until $Q$ lies in $I_1$ or $I_2$.

We note that if $L \ge x^{0.1 + 2\epsilon}$, then $J_1 \cup J_2$ is a single interval of length
$$\frac{x^{0.63-\epsilon}R_k^{-4/5}}{x^{0.37+\epsilon}R_k^{-1/5}L^{-1}} > x^{0.32}$$
It is easy to see that in this case a suitable subproduct $Q$ exists. This gives (ii).

Note that since $R_k(s)$ is the shortest of $R_1(s), \ldots , R_k(s)$, we have
$$R_k^{1/5}L = \frac{R_1 \cdots R_k}{R_k^{4/5}} \ge \frac{R_1 \cdots R_k}{(R_1 \cdots R_k)^{4/5k}} \ge Z^{(5k-4)/5k}$$
and similarly
$$R_k^{1/2}L \ge Z^{(2k-1)/2k}.$$
Hence the union $J_1 \cup J_2$ contains
$$[x^{0.37+\epsilon}Z^{-(5k-4)/5k}, x^{0.45-\epsilon}z^{-1/2}] \cup [x^{0.55+\epsilon}Z^{-(2k-1)/2k}, x^{0.63-\epsilon}z^{-4/5}].$$
As $k \ge K$ and the intervals are the longer the larger $k$ is, by the assumption of (iii) there is a subproduct of $Q_1(s) \cdots Q_n(s)H(s)$ lying in this set, implying the result.
\end{proof}

The next result is used when applying the Heath-Brown decomposition to a polynomial $P(s)$ to bound the lengths of the factors.

\begin{lemma}
\label{lem:HB-dec-lengths}
Let $L_{\zeta} \le P \le 10x^{1/2}$ be given, let $J = O(1)$ and let $N_1(s), \ldots , N_J(s)$ be such that $N_1 \cdots N_J = P$. Assuming that $N_i < L_{\zeta}$ for all $i$, one may partition $\{N_1(s), \ldots , N_J(s)\}$ into two sets $\mathcal{Q}_1, \mathcal{Q}_2$ such that the products $Q_i(s)$ of elements of $\mathcal{Q}_i$ satisfy
$$\max(Q_1, Q_2) \in [\sqrt{P}, \max(L_{\zeta}, P^{2/3})].$$
\end{lemma}

\begin{proof}
We first use a recursive algorithm for reducing the number of factors $N_i(s)$. As long as there exist $i \neq j$ such that $N_iN_j < L_{\zeta}$, replace $N_i(s)$ and $N_j(s)$ by their product $N_i(s)N_j(s)$, reducing the number of polynomials by one. In the end the number of polynomials must be two or three, as otherwise we would have $P < L_{\zeta}$ or $P > L_{\zeta}^2 > 10x^{1/2}$. If there remain two polynomials, we are done. If there remain three polynomials, combine the shortest two of them. The resulting polynomial has length not exceeding $P^{2/3}$.
\end{proof}

For the case $p_2 > L_{\zeta}$ we use the following lemma.

\begin{lemma}
\label{lem:p_2_large}
Let $x^{0.275+\epsilon} \le P_2 \le P_1 \le 10x^{1/2}$ be given with $P_1P_2^2 \le 10x$. Assume that $x/P_1P_2 \ge L_{\zeta}$ and $\min(P_1, x/P_1P_2) \le x^{0.355 - \epsilon}$. Then, assuming that the polynomial $F(s)$ obtained by applying the Heath-Brown decomposition to any polynomials longer than $L_{\zeta}$ has at least one zeta factor, $F(s)$ satisfies Claim~\ref{claim}.
\end{lemma}

\begin{proof}
If $F(s)$ has at least two zeta factors, we are done by Lemma~\ref{lem:two_zetas}. We let $Q(s)$ be the product of factors of $F(s)$ that are not factors resulting from the Heath-Brown decomposition applied to $P_1(s)$ or $P_2(s)$, so that $P_1P_2Q = xS^{o(1)}$. We may assume that $P_1 \le x^{0.355 - \epsilon}$, as the case $Q \le x^{0.355 - \epsilon}$ is symmetric.  

If the zeta factor $Z(s)$ arises from decomposing $P_1(s)$, denote by $P_1'(s)$ the remaining polynomial of length $P_1/Z$. We apply Proposition \ref{prop:range_0.45}(iii) with
$$(A, B, C) = (Z, P_1'Q, P_2).$$
Note that $P_1'Q \le xS^{\epsilon}/ZP_2 \le xS^{\epsilon}/x^{0.55+\epsilon} < x^{0.45-\epsilon/2}$ and $P_2 \le x^{0.355 - \epsilon}$ by assumption.

If the zeta factor $Z(s)$ arises from decomposing $P_2(s)$, we similarly as above take $(A, B, C) = (Z, P_2'Q, P_1)$.

If the zeta factor $Z(s)$ is a factor of $Q(s)$, denote the product of the other factors of $Q(s)$ by $Q'(s)$, and take
$$(A, B, C) = (Z, P_2Q', P_1).$$
Now $P_2Q' \le xS^{\epsilon}/ZP_1 \le xS^{\epsilon}/x^{0.55+\epsilon} < x^{0.45-\epsilon/2}$ and $P_1 \le x^{0.355 - \epsilon}$ by assumption. 
\end{proof}

For the case where $P_2 \ge x^{0.275+\epsilon}$ and there are no zeta factors we will employ a casework on the lengths of the polynomials arising from the Heath-Brown decomposition.

\subsection{Details of the procedure and results}
\label{sec:comp_results}

We employ Harman's sieve in a similar manner as in Section \ref{sec:0.5}. First, starting from the Buchstab sum $S(\mathcal{A}(m), 2x^{1/2})$, we apply Buchstab's identity twice to get
\begin{align*}
&S(\mathcal{A}(m), 2x^{1/2}) \\
= &S(\mathcal{A}(m), x^{0.06}) - \sum_{x^{0.06} \le p_1 < 2x^{1/2}} S(\mathcal{A}_{p_1}(m), x^{0.06}) + \sum_{x^{0.06} \le p_2 < p_1 \le 2x^{1/2}} S(\mathcal{A}_{p_1p_2}(m), p_2)
\end{align*}
(cf. \eqref{eq:two_buchs}). An asymptotic for the first term on the right hand side is obtained from Lemma~\ref{lem:MN-2}(ii) (if $H < L_{\zeta}$) and Lemma~\ref{lem:ignore_short} (if $H > L_{\zeta}$). We also have asymptotics for the second term: First, apply the Heath-Brown decomposition to $P_1(s)$. In case the resulting polynomial
$$F(s) = \prod_{1 \le i \le J} N_i(s)R_1(s) \cdots R_k(s)H(s)$$
has no zeta factors we may apply Lemma \ref{lem:MN-2}, and in the presence of a zeta factor one sees that the non-zeta factors of $N_1(s) \cdots N_J(s)H(s)$ may be partitioned into $B'(s)C'(s)$ such that $B'(s) < x^{0.45 - \epsilon}, C'(s) < x^{0.355 - \epsilon}$, from which the result follows via $k$ applications of Lemma \ref{lem:ignore_short} and Proposition \ref{prop:range_0.45}(iii).

Hence, the computation starts from
$$\sum_{\substack{x^{0.06} \le p_2 < p_1 < 2x^{1/2}}} S(\mathcal{A}_{p_1p_2}(m), p_2),$$
with the aim of showing that the Buchstab identity can be applied in such a manner that the resulting loss is less than one. As in Section \ref{sec:0.5_algorithm}, we split the sum into sums over $p_1 \in I_1, p_2 \in I_2$ for shorter intervals $I_i$.

We handle the case $p_2 > L_{\zeta}$ separately. In this case we apply the Heath-Brown decomposition to $P_1(s)$, $P_2(s)$ and any potential $R_i(s)$ longer than $L_{\zeta}$. We perform a casework on the lengths of the resulting factors, utilizing Lemma~\ref{lem:HB-dec-lengths} and using Lemma~\ref{lem:p_2_large} to discard the case with zeta factors, then considering ways of combining the factors to polynomials $A(s), B(s), C(s)$ and checking whether any satisfy Lemma~\ref{lem:uncertain_typeII}. An asymptotic is obtained if in all cases a suitable decomposition $F(s) = A(s)B(s)C(s)$ is found. (The loss arising from $p_2 > L_{\zeta}$ is roughly $0.09$ with our choice of parameters below.) From now on, assume that $p_2 < L_{\zeta}$.

We implement a procedure that determines if an asymptotic for
\begin{align*}
\sum_{\substack{p_1, \ldots , p_n \\ p_i \in I_i \\ p_n < \ldots < p_1}} S(\mathcal{A}_{p_1 \cdots p_n}(m), z)
\end{align*}
may be obtained, where $z$ is a function of $x$ or $z = p_n$. In practice $n \ge 2$ and $z$ is either at most $x^{0.06}$ or equal to $p_n$. Once again the problem is determining whether the polynomial
$$F(s) = P_1(s) \cdots P_n(s)R_1(s) \cdots R_k(s)H(s)$$
may be written as $A(s)B(s)C(s)$ so that Proposition \ref{prop:range_0.45} is satisfied, for any polynomials $R_i \le z$ and $H$. We may apply the Heath-Brown decomposition to $P_1(s)$ if we wish.

Our procedure is as follows (recall the notation $P_l, P_u, h, t$ and $r$ from the beginning of Section \ref{sec:theoretical}):
\begin{enumerate}
\item Check if the case $H > L_{\zeta}$ can be handled. (The answer is trivially positive if $(P_1)_l + \ldots + (P_n)_l > 1 - t/2 + \epsilon$. Assume otherwise.)
\begin{enumerate}[(i)]
\item Write $Q(s) = R_1(s) \cdots R_k(s)$. Consider whether $P_1(s) \cdots P_n(s)Q(s)$ may be partitioned as $B(s)C(s)$ as in Lemma~\ref{lem:uncertain_typeI}. By Lemma~\ref{lem:ignore_short}, the polynomial $Q(s)$ may be dropped if $z \le x^{0.06}$.
\item If this fails and $P_1 > L_{\zeta}$, apply the Heath-Brown decomposition to $P_1(s)$. Consider a casework on the lengths of the factors $Q_1(s), Q_2(s)$ of $P_1(s)$ (see Lemma~\ref{lem:HB-dec-lengths}), and in each case consider partitions $B(s)C(s)$ of the polynomial $Q_1(s)Q_2(s)P_2(s) \cdots P_n(s)Q(s)$ and check whether any satisfy Lemma~\ref{lem:uncertain_typeI}. Again, $Q(s)$ may be dropped if $z \le x^{0.06}$.
\end{enumerate}
\item Check if the case $H < L_{\zeta}$ can be handled without applying the Heath-Brown decomposition to $P_1(s)$.
\begin{enumerate}[(i)]
\item Write $Q(s) = H(s)R_1(s) \cdots R_k(s)$ and consider whether $P_1(s) \cdots P_n(s)Q(s)$ may be written as a product $A(s)B(s)C(s)$ satisfying Lemma~\ref{lem:uncertain_typeII}.
\item If not, and $P_1 \cdots P_n < x^{1 - t/2 - \epsilon}$ so that $k \ge 1$, consider cases depending on the size of $H(s)$. In each case, write $Q(s) = R_1(s) \cdots R_k(s)$ and consider partitions of $P_1(s) \cdots P_n(s)H(s)Q(s)$, again checking whether Lemma~\ref{lem:uncertain_typeII} applies. If $z \le x^{0.06}$, it also suffices if some condition of Lemma~\ref{lem:MN-2} is satisfied.
\end{enumerate}
\item Check if the case $H < L_{\zeta}$ can be handled by applying the Heath-Brown decomposition to $P_1(s)$ (assuming $P_1 > L_{\zeta}$).
\begin{enumerate}[(i)]
\item Check the case where $P_1(s)$ outputs a zeta sum. Let $P_1(s)$ output $Q_1(s), Q_2(s)$ with $Q_1 \ge L_{\zeta}$ a zeta sum (and $Q_2(s)$ possibly constant), and write $Q(s) = R_1(s) \cdots R_k(s)H(s)$. Consider decompositions of $Q_2(s)P_2(s) \cdots P_n(s)Q(s)$ as $B(s)C(s)$, and check whether any satisfy the conditions of Lemma~\ref{lem:uncertain_typeI}. The factors $R_1(s), \ldots , R_k(s)$ may be dropped if $z \le x^{0.06}$.
\item Check the case where $P_1(s)$ does not output a zeta sum. Perform a casework on the lengths of the factors $Q_1(s), Q_2(s)$ (see Lemma~\ref{lem:HB-dec-lengths}), write $Q(s) = R_1(s) \cdots R_k(s)H(s)$ and consider decomposition $A(s)B(s)C(s)$ of the polynomial $Q_1(s)Q_2(s)P_2(s) \cdots P_n(s)Q(s)$, checking whether any satisfy Lemma~\ref{lem:uncertain_typeII}.
\end{enumerate}
\end{enumerate}

If (1) fails, we return that an asymptotic cannot be established. Assuming (1) succeeds, we perform step (2), and only if it fails we perform step (3). Success of either (2) or (3) results in finding an asymptotic formula.

If no asymptotic formula is found, the loss is bounded as in~\eqref{eq:full_loss}
and~\eqref{eq:loss_bound}.

Values of $z$ for which we may apply the Buchstab identity twice more are determined by trial and error with the candidates $z = 0.06, 0.059, \ldots , 0.001$. 

The intervals $I_i = [x^{\alpha_i}, x^{\beta_i}]$ are chosen so that $\beta_i - \alpha_i = 1/6000$ if $i \le 2$ and $\beta_i - \alpha_i = 1/450$ otherwise, and we take $\epsilon = 10^{-9}$ in various places. The computation gives the upper bound $0.996 < 1$ for the loss. As with $c = 0.5$, the program prints more detailed information on the contribution of different values of $p_1$ on the loss.

\begin{remark}
As in Remark~\ref{rem:easy}, one may approximate the loss by more straightforward means. Such approximations indicate that reaching $R = x^{0.17}$ would require new ideas.
\end{remark}

\section{Applications}
\label{sec:applications}

In this section we discuss the applications of Theorems~\ref{thm:PRF},~\ref{thm:bin} and~\ref{thm:approx} and show how the theorems follow from Theorem~\ref{thm:many}. We remark that likely one could obtain improvements to our results by proving variants of Theorem~\ref{thm:many} for different values of $c$.

We also note that Theorem \ref{thm:0.5} gives an improvement in a recent results of Kosyak, Moree, Sofos and Zhang \cite{KMSZ} on the maximum coefficients of cyclotomic polynomials. The author thanks Moree for pointing this out.

\subsection{Prime-representing functions}
\label{sec:PRF}

A folklore question in number theory is finding simple (non-trivial) functions that generate primes, i.e. functions $f : \mathbb{Z}_+ \to \mathbb{Z}_+$ such that $f(n)$ is a prime for all $n$. Mills~\cite{mills} famously showed that there exists a constant $A > 1$ such that $\lfloor A^{3^n} \rfloor$ is a prime for every $n \in \mathbb{Z}_+$. In short, the idea is to inductively construct a convergent sequence of constants $A_1, A_2, \ldots$ for which $\lfloor A_k^{3^n} \rfloor$ is a prime for any $n \le k$, and take $A = \lim_{k \to \infty} A_k$. The constant $3$ in the exponent arises from there being primes in intervals of the form $[x, x + x^{1 - 1/3}]$ for $x$ large enough, and stronger results on the length of such intervals allow one to reduce the constant $3$.

While we do not know whether there exist primes in intervals of length $x^{1 - 1/2}$, nevertheless Matomäki~\cite{matomaki-PRF} has shown that there exist constants $A > 1$ such that $\lfloor A^{2^n} \rfloor$ is a prime for any $n \ge 1$. The idea is to consider merely almost all intervals $[x, x + x^{1 - 1/2}]$ instead of all of them. With some modification the proof of Mills adapts to this case, assuming one has a strong enough bound for the set of exceptional $x$ for which $[x, x + x^{1 - 1/2}]$ has no or only few primes. Such a result is given by Matomäki in~\cite{matomaki}. Here again one may reduce the constant $2$ assuming one has analogous results for shorter intervals.

An improvement of the constant $2$ has been given by Islam in \cite{islam} by extending the result of Matomäki in~\cite{matomaki} to intervals slightly shorter than $\sqrt{x}$, reducing the constant to $\approx 1.946$.

The bound of Theorem~\ref{thm:0.45}, or Theorem~\ref{thm:many} to be precise, is strong enough that Matomäki's proof adapts to intervals of length $x^{0.45} = x^{1 - 1/(20/11)}$, leading to prime-representing functions of the form $\lfloor A^{(20/11)^n} \rfloor$. Numerically $20/11 \approx 1.8181\ldots$

\begin{proof}[Proof of Theorem~\ref{thm:PRF}]
Our proof follows those given by Mills~\cite{mills} and Matomäki~\cite[Corollary 4]{matomaki-PRF}.

Fix $\alpha \ge 20/11$ and $\epsilon > 0$ small enough. We inductively construct a sequence $p_0, p_1, \ldots$ of primes such that the interval
$$I_n = [p_n^\alpha, (p_n + 1)^\alpha - 1)$$
contains at least $\epsilon p_n^{\alpha-1}/\log p_n$ primes and
$$p_{n+1} \in I_n$$
for all $n \ge 0$. (Here we override the notation in Section \ref{sec:intro}, where $p_n$ denoted the $n$th prime.) Let $S_n = I_n \cap \mathbb{P}$.

Choose $p_0$ as a large prime, and assume we have already constructed $p_0, \ldots , p_n$ as above. We aim to construct $p_{n+1} \in S_n$ so that $[p_{n+1}^\alpha, (p_{n+1} + 1)^\alpha - 1)$ contains many primes. Note that for $p \in S_n$, the intervals
$$[p^\alpha, (p + 1)^\alpha - 1) \subset [p_n^{\alpha^2}, 2p_n^{\alpha^2}]$$
are disjoint and of length $(\alpha + o(1))p^{\alpha-1} = (\alpha + o(1))(p^\alpha)^{(\alpha-1)/\alpha} > (p^\alpha)^{0.45}$. By Theorem~\ref{thm:many} all but $O(p_n^{(0.18 + \epsilon)\alpha^2})$ primes $p \in S_n$ are such that $[p^\alpha, (p+1)^\alpha - 1)$ contains at least $\epsilon p^{\alpha-1}/\log p$ primes. By the induction hypothesis, $|S_n| \ge \epsilon p_n^{\alpha-1}/\log p_n$, which is much larger than $p_n^{(0.18 + \epsilon)\alpha^2}$. Hence, one may choose $p_{n+1}$ as desired.

Now, let
$$a_n = p_n^{\alpha^{-n}} \quad \text{and} \quad b_n = (p_n + 1)^{\alpha^{-n}}.$$
We trivially have $a_n < b_n$, by construction we have $a_{n+1} \ge a_n$ (as $a_{n+1} = p_{n+1}^{\alpha^{-(n+1)}} \ge (p_n^\alpha)^{\alpha^{-(n+1)}} = a_n$), and we have
$$b_{n+1} = (p_{n+1} + 1)^{\alpha^{-(n+1)}} < \left((p_n + 1)^{\alpha}\right)^{\alpha^{-(n+1)}} = (p_n + 1)^{\alpha^{-n}} = b_n.$$
It follows that $a_n$ is a bounded by $b_1$ and increasing. Furthermore, if one defines
$$A = \lim_{n \to \infty} p_n^{\alpha^{-n}},$$
we have, for all $n \ge 0$, $a_n \le A < b_n$ by above and thus $\lfloor A^{\alpha^n} \rfloor = p_n$.
\end{proof}

Similarly to~\cite{matomaki-PRF}, the proof could be generalized to prime-representing functions of the form $\lfloor A^{c_1 \cdots c_n} \rfloor$, where $c_i \ge 20/11$, and one sees that there are uncountably many admissible $A$ for any given $\alpha$ or $c_i$.

\subsection{Binary digits of primes}

In the last years there have been numerous results on primes with restricted digits. Mauduit and Rivat~\cite{mauduit-rivat} showed that the sum of digits function of prime numbers in a given base is equidistributed modulo $m$ for any fixed $m \in \mathbb{Z}_+$ (except in certain trivial cases). Bourgain~\cite{bourgain} has shown that one may prescribe a positive proportion of the binary digits of an integer at arbitrary places and find primes in the resulting set (assuming the final digit has not been set to $0$). Maynard~\cite{maynard} proved that, for any $d \in \{0, 1, \ldots , 9\}$, there are infinitely many primes without the digit $d$ in their decimal representation.

We consider the problem of finding primes with many digits $d$ in their binary representation for a given $d \in \{0, 1\}$. This is similar to the problem considered by Bourgain, differing in that we do not prescribe the places of the digit $d$ in the binary expansion. We note that the corresponding problem for smooth numbers was very recently studied by Hauck and Shparlinski \cite{hauck-shparlinski}.

We first give a useful lemma.

\begin{lemma}
\label{lem:thin_tail}
Fix $d \in \{0, 1\}$. For any $\epsilon > 0$ there exists a constant $c_{\epsilon} > 0$ such that the following holds: The number of integers $n \in [0, 2^k)$ whose binary expansion contains at most $(1/2 - \epsilon)k$ digits $d$ is $O((2^k)^{1 - c_{\epsilon}})$.
\end{lemma}

Note that while the lemma is stated for integers in the interval $[0, 2^k)$, the result may be applied to any $2^k$ consecutive integers, showing that most of those integers have approximately equal amounts of zeros and ones among their final $k$ binary digits.

\begin{proof}
The number of such $n$ is bounded by
$$\sum_{0 \le i \le (1/2 - \epsilon)k} {k \choose i}.$$
Via Striling's approximation one may show that for $i \le (1/2 - \epsilon)k$ we have ${k \choose i} \ll (2^k)^{1 - c_{\epsilon}}$ for some constant $c_{\epsilon} > 0$, from which the result follows.
\end{proof}

We then note that given $\epsilon > 0$ and $x = 2^k$ large enough (in terms of $\epsilon$), there are primes $p < x$ such that at least $(1/2 - \epsilon)k$ of the binary digits of $p$ are ones. Indeed, the number of integers $n < x$ having less than $k(1/2 - \epsilon)$ binary ones is $O(x^{1 - c_{\epsilon}})$ by Lemma \ref{lem:thin_tail} whereas the prime number theorem states that there are roughly $x/\log x$ primes less than $x$.

The $50\% - \epsilon$ bound may be improved by adapting the argument to short intervals. Let us sketch this argument: By~\cite{BHP}, intervals of length $x^{0.525}$ contain $\gg x^{0.525}/\log x$ primes. Consider then the interval
$$I := [2^k - 2^{k \cdot 0.525}, 2^k)$$
for $k$ large. The first $0.475k$ digits of any integer $n \in I$ are ones. Furthermore, by Lemma \ref{lem:thin_tail} there must be primes $p \in I$ such that out of the last $0.525k$ digits of $p$, at least a proportion of $50\% - \epsilon$ are ones.  Hence the number of ones is at least
$$0.475k + \left(\frac{0.525}{2} - \epsilon\right)k = (0.7375 - \epsilon)k,$$
i.e. a proportion of $73.75\% - \epsilon$ of the digits are ones. A natural barrier for this method is $75\% - \epsilon$, which is what one would get if one could find primes in intervals of length $x^{1/2 + \epsilon}$. 

One may improve the argument by considering merely almost all intervals. As in Section~\ref{sec:PRF}, this requires strong enough quantitative bounds on the size of the exceptional set.

\begin{proof}[Proof of Theorem~\ref{thm:bin}]
Let $k$ be a large enough integer divisible by $20$, let $x = 2^{k-1}$ and denote
$$I = [x, 2x) = [2^{k-1}, 2^k).$$
Any integer $n \in I$ has exactly $k$ binary digits.

Let $\epsilon > 0$ be small enough and let $t = 0.517 + \epsilon$. Let $n_1 < \ldots < n_m$ denote the integers $n \in I$ for which $2^{0.45k} \mid n$ and whose first $0.55k$ digits contain at least $\lfloor tk \rfloor$ instances of the digit $d$. (Recall that $20 \mid k$, so $0.45k$ and $0.55k$ are integers.) Hence
$$m \ge \binom{0.55k-1}{\lfloor tk \rfloor},$$
and from Stirling approximation we obtain
$$m \ge \left(\frac{0.55^{0.55}}{t^{t}(0.55 - t)^{0.55 - t}} - \epsilon\right)^k > 1.1329^k > 2^{(0.18 + \epsilon)k} > x^{0.18 + \epsilon}$$
for $\epsilon > 0$ small enough and $k$ large enough in terms of $\epsilon$. For each $n_i$, consider the interval
$$I_i = [n_i, n_i + 2^{0.45k}).$$
Note that $2^{0.45k} > n_i^{0.45}$ and that $I_i$ are pairwise disjoint. Since $m \ge x^{0.18 + \epsilon}$, by Theorem~\ref{thm:many} there exists $1 \le i \le m$ such that the interval $I_i$ contains $\gg 2^{0.45k}/k$ primes. Consider then the primes $p \in I_i$ for such $i$. By Lemma \ref{lem:thin_tail}, there exist primes $p \in I_i$ such that out of the last $0.45k$ digits of $p$ at least $(0.225 - \epsilon/2)k$ are equal to $d$. As the first $0.55k$ digits of $p$ have at least $tk = (0.517 + \epsilon)k$ digits equal to $d$, in total $p$ has at least
$$(0.517 + 0.225 + \epsilon/2)k = (0.742 + \epsilon/2)k$$
digits equal to $d$.
\end{proof}

We note that the under the Lindelöf hypothesis one may go beyond the $75\% - \epsilon$ barrier: It is known that the Lindelöf hypothesis implies that for any $c > 0$ the number of prime gaps $p_{n+1} - p_n$ longer than $x^c$ is at most $x^{1 - 2c + \epsilon}$ \cite{yu}. Applying this result with $c = 0.4$ in the above argument would allow one to replace $0.55$ by $0.6$ and $t = 0.517 + \epsilon$ by $0.542 + \epsilon$, resulting in a proportion of $76.3\%$ of the digit $d$.

\subsection{Approximation by multiplicative functions}

Harman~\cite{harman-approx} has considered the approximation of real numbers by multiplicative functions. More precisely, for a given real $\alpha > 1$, the aim is to find $n$ with
\begin{align}
\label{eq:harman-approx}
\left|\frac{\sigma(n)}{n} - \alpha\right|
\end{align}
as small as possible, where $\sigma(n) = \sum_{d \mid n} d$ is the sum-of-divisors function. Harman shows that there are infinitely many $n$ for which~\eqref{eq:harman-approx} is smaller than $n^{-0.52}$, improving on a result of Wolke~\cite{wolke}.

Harman~\cite[Theorem 2.2]{harman-approx} shows that if the number of disjoint intervals of length $x^c$ containing few primes up to $x$ is $o(x^{c/(2 - c)}/\log x)$, then~\eqref{eq:harman-approx} may be bounded by $n^{-(1 - c) + \epsilon}$ infinitely often. By a result of Peck~\cite{peck} this is true for $c \approx 0.471$, giving the bound $n^{-0.52}$ above. Theorem~\ref{thm:many} allows one to take $c = 0.45$, as our exponent $0.18 + \epsilon$ is smaller than $0.45/(2 - 0.45) \approx 0.29$. This implies Theorem~\ref{thm:approx}.

\bibliography{references}
\bibliographystyle{plain}
\end{document}